\documentclass[12pt]{imsart}
\usepackage{amsthm,amsmath}
\usepackage[colorlinks,citecolor=blue,urlcolor=blue]{hyperref}

\usepackage{amsmath,amsfonts,amssymb,amsthm}

\usepackage{enumitem}

\RequirePackage{amsthm,amsmath,amsfonts,amssymb}
\RequirePackage[numbers]{natbib}
\RequirePackage[colorlinks,citecolor=blue,urlcolor=blue]{hyperref}
\RequirePackage{graphicx}
\usepackage{lineno,hyperref}

\startlocaldefs
\modulolinenumbers[5]

\usepackage[utf8]{inputenc}
\usepackage[english]{babel}

\usepackage[margin=1in]{geometry}
\usepackage{bm}
\usepackage{graphicx}
\usepackage{amssymb,amsmath,amsthm,mathrsfs}
\usepackage{epstopdf}
\usepackage{algorithm}
\usepackage{amsfonts,delarray}
\usepackage{algorithm,algcompatible,amsmath}
\usepackage{rotating,color}
\usepackage{subfigure}
\usepackage{multirow}
\usepackage{titlesec,textcase}
\usepackage{longtable}
\usepackage{bbm}
\usepackage{rotating}

\newtheorem{Theorem}{Theorem}[section]
\newtheorem{Lemma}[Theorem]{Lemma}
\newtheorem{Corollary}[Theorem]{Corollary}
\newtheorem{Proposition}[Theorem]{Proposition}
\newtheorem{Assumption}{Assumption}

\newtheorem{Definition}[Theorem]{Definition}

\theoremstyle{definition}

\RequirePackage[normalem]{ulem} 

\definecolor{rp}{RGB}{83,54,106}

\def\boxit#1{\vbox{\hrule\hbox{\vrule\kern6pt\vbox{\kern6pt#1\kern6pt}\kern6pt\vrule}\hrule}}

\endlocaldefs

\allowdisplaybreaks

\begin{document}
\begin{frontmatter}
\title{ The weak law of large numbers for the friendship paradox index}

\runtitle{asymptotic distribution of  }
\runauthor{ }
\begin{aug}

\author[A]{\fnms{Mingao} \snm{Yuan}\ead[label=e1]{myuan2@utep.edu}}

\address[A]{Department of Mathematical Sciences,
The University of Texas at El Paso, El Paso, TX, USA\\
\printead{e1}}
\end{aug}

\begin{abstract}
The friendship paradox index is a network summary statistic used to quantify the friendship paradox, which describes the tendency for an individual's friends to have more friends than the individual. 
In this paper, we utilize Markov’s inequality to derive the weak law of large numbers for the friendship paradox index in a random geometric graph, a widely-used model for networks with spatial dependence and geometry.  For uniform random geometric graph, where the nodes are uniformly distributed in a space, the friendship paradox index is asymptotically equal to $1/4$. On the contrary, in nonuniform random geometric graphs, the nonuniform node distribution leads to distinct limiting properties for the index. In the relatively sparse regime, the friendship paradox index is still asymptotically equal to $1/4$, the same as in the uniform case. In the intermediate sparse regime, however, the index converges in probability to $1/4$ plus a constant that is explicitly dependent on the node distribution. Finally, in the relatively dense case, the index diverges to infinity as the graph size increases. Our results highlight the sharp contrast between the uniform case and its nonuniform counterpart.
\end{abstract}



\end{frontmatter}

\section{Introduction}
\label{S:1}

Social networks frequently exhibit the friendship paradox, where the mean degree of an individual's contacts consistently outstrips their own number of connections \cite{F91,PYN19,CKN21,KKF24}.
In graph-theoretic terms, this implies that a node's neighbors typically possess higher degrees than the node itself. To quantify this, the friendship paradox index of a node is defined as the difference between the mean degree of its neighbors and its own degree. By extending this local quantity to the entire network, the average friendship paradox index serves as a global metric for the intensity of the friendship paradox \cite{CKN21}. For brevity, we refer to this average index simply as the friendship paradox throughout this paper.

The friendship paradox serves as a vital tool in network data analysis, where it is utilized to reduce variance in survey-based polling, efficiently estimate power-law degree distributions, and quickly detect disease outbreaks \cite{NK19,J19,NK21}. 
 Recent research has further generalized it to more complex structures, including weighted graphs with attributes \cite{EK24}. Beyond its practical applications, several studies have investigated the theoretical properties of the friendship index in random graphs. For instance, \cite{PYN19}  analytically quantified the expected friendship index across various random graph models, while \cite{Y24} characterized the asymptotic distribution of the index specifically for a dense uniform random geometric graph.

Random graphs are graphs where vertices and edges are formed by a random process, modeling real-world complex networks. The famous Erd\H{o}s-R\'{e}nyi  model constructs a graph by connecting each of the \(n\) possible pairs of nodes with an independent probability \(p\). In contrast to the Erd\H{o}s-R\'{e}nyi random graph, the Random Geometric Graph (RGG) is generated by placing nodes at random according to a probability density function \(f(x)\) on the unit interval \([0,1]\) and connecting any two nodes whose distance falls below a specific threshold. Due to their distance-based construction, RGGs provide a more realistic framework for modeling spatial dependence and geometry in empirical data \cite{G21,DC23,GMPS23}.

In this paper, we leverage Markov’s inequality to establish the weak law of large numbers for the friendship index in  random geometric graphs. This task is not as straightforward as it may appear.  Establishing the weak law necessitates the computation of the index's expectation and variance; however, obtaining simple closed-form expressions for these moments is analytically intractable due to the edge dependencies inherent in RGGs. For instance, since the friendship paradox is defined by node degrees, the inherent edge dependencies in RGGs render even the derivation of a simple degree distribution complex. We utilize an approximation method based on the conditional expectation given a node's spatial location.

Our results indicate that in uniform RGGs, where nodes are uniformly distributed, the friendship index converges in probability to \(1/4\).  Conversely, in nonuniform RGGs,  the nonuniform node distribution leads to distinct limiting properties.   In a relatively sparse RGG, the friendship index converges in probability to \(1/4\), which is the limit observed in the uniform case. For intermediate sparse RGG, the friendship index is asymptotically equal to \(1/4\) plus a constant that depends on the specific node distribution. For denser RGG, the friendship index tends to infinity as the graph size  increases. These results underscore the fundamental disparity between the uniform RGGs and nonuniform RGGs.

The remainder of this paper is organized as follows. Section \ref{pre} introduces basic concepts and necessary lemmas. In Section \ref{expecfp}, we derive an asymptotic expression for the mean of the friendship paradox, while Section \ref{fdvar} establishes the asymptotic second moment. These results culminate in Section \ref{wllnfp}, where we prove the weak law of large numbers for the friendship paradox using Markov’s inequality. Finally, Section \ref{fplemprf} provides the detailed proofs for all lemmas.

\medskip

Notations: Throughout this paper, we adopt the Bachmann–Landau notation for asymptotic analysis. Let $c_1,c_2$ be two positive constants. For two positive sequence $a_n$, $b_n$, denote $a_n=\Theta( b_n)$ if $c_1\leq \frac{a_n}{b_n}\leq c_2$; denote $a_n=O(b_n)$ if $\frac{a_n}{b_n}\leq c_2$; $a_n=o(b_n)$ or $b_n=\omega(a_n)$ if $\lim_{n\rightarrow\infty}\frac{a_n}{b_n}=0$. Let $X_n$ be a sequence of random variables. $X_n=O_P(a_n)$ means $\frac{X_n}{a_n}$ is bounded in probability.  $X_n=o_P(a_n)$ means $\frac{X_n}{a_n}$ converges to zero in probability. The notation $\sum_{i\neq j\neq k\neq l}$ represents summation over indices $i,j,k,l\in\{1,2,3,\dots,n\}$ with $i\neq j,i\neq k, i\neq l, j\neq k, j\neq l, k\neq l$. $I[E]$ is the indicator function of event $E$. $E^c$ represents the complement of event $E$. For a set $B$, $|B|$ denote the number of elements in the set $B$. For a function $g(x)$, $g^{(k)}(x)$ denotes the $k$-th derivative of $g(x)$.  We also use $f'(x)$, $f''(x)$ and $f'''(x)$ denote the first, second, and third derivatives of \(f(x)\), respectively.

\section{Preliminaries}\label{pre}

Let $\mathcal{G}=(\mathcal{V},\mathcal{E})$  be an undirected graph of size $n$, with a vertex set \(\mathcal{V}=\{1,2,\dots ,n\}\) and an edge set \(\mathcal{E}\subseteq \{\{i,j\}:i,j\in \mathcal{V},i\ne j\}\).  The adjacency matrix \(A\) is defined such that \(A_{ij}=1\) if \(\{i,j\}\in \mathcal{E}\), and \(A_{ij}=0\) otherwise. The degree of node $i$ is the the number of edges adjacent to it, that is, $d_i=\sum_{j}A_{ij}$.

The friendship paradox is a summary statistic of graph. It measures the strength of the paradox that on average the degrees of the neighbours of a node 
are larger than the degree of the node itself \cite{F91,CKN21}. Next we introduce the friendship paradox presented in \cite{CKN21}.

 \begin{Definition}\label{def0}
 Let $\mathcal{G}=(\mathcal{V},\mathcal{E})$ be a graph. The friendship index of node $i$ in $\mathcal{G}$ is defined as
 \[\Delta_i=\frac{1}{d_i}\sum_{j\in[n]\setminus \{i\}}A_{ij}d_j-d_i,\]
 where $\Delta_i=0$ if $d_i=0$. The friendship paradox of $\mathcal{G}$ is defined as 
\begin{equation}\label{randic}
\mathcal{F}_{n}=\frac{1}{n}\sum_{i=1}^n\Delta_i.
\end{equation}
 \end{Definition}
The friendship paradox  is zero if and only if   the graph is  regular ($d_i=d_j$ for all pairs
of neighboring nodes).  In all other cases, this value is strictly greater than zero \cite{CKN21}. In addition, the maximum of $\Delta_i$ is $n-2$.

In a random graph model, the presence of an edge between any two nodes is governed by a probability distribution.  The most widely studied version is the Erd\H{o}s-R\'{e}nyi graph, where every potential edge is included independently with probability \(p\). One notable variation—the random geometric graph—is often used to model complex networks that exhibit geometry and dependence structures \cite{DC23,GMPS23}.

\begin{Definition}\label{def1}
Let $r_n\in[0,0.5]$ be a real number and $f(x)$ be a probability density function on $[0,1]$. Given independent random variables $X_1,X_2,\dots,X_n$ distributed according to $f(x)$, the Random Geometric Graph (RGG) $\mathcal{G}_{n}(f,r_n)$ is defined as
\[A_{ij}=I[d(X_i,X_j)\leq r_n],\]
 where $A_{ii}=0$ and  $d(X_i,X_j)=\min\{|X_{i}-X_{j}|,1-|X_{i}-X_{j}|\}$.
\end{Definition}

In this model, \(n\) nodes are randomly distributed on $[0,1]$ according to $f(x)$. Two nodes are connected if their distance is less than $r_n$. This setup allows the graph to represent real-world networks with the spatial geometry and structural dependencies.

When $f(x)$ is the uniform density, we say $\mathcal{G}_{n}(f,r_n)$  is the \textit{uniform random geometric graph}. Conversely, we refer to the model as \textit{nonuniform random geometric graph} when $f(x)$ is non-constant. Current research has largely been restricted to uniform random geometric graphs \cite{HRP08, GRK05, GMPS23, YF25, Y25c, Y23b}. However, nonuniform random geometric graphs have recently gained attention for their ability to fit real networks better \cite{PBGKL23,G21, Y25}.  

In general, RGGs are more analytically complex than Erd\H{o}s-R\'{e}nyi graph due to dependence among edges. The theoretical analysis of nonuniform RGGs is more involved than that of the uniform RGGs, particularly for network statistics, such as the friendship paradox, that are non-polynomial functions of the adjacency matrix.
One of the reasons for this complexity is that while edges in RGGs are generally dependent, those incident to a fixed node in a uniform RGG are independent \cite{Y23b,YF25}. However, this property does not hold in the nonuniform case.  For example, the edges $A_{12}$, $A_{13}$ and $A_{23}$ are correlated in both uniform and nonuniform RGGs due to the triangle inequality and the shared spatial proximity of the nodes. However, a critical distinction arises:  the edges \(A_{12}\) and \(A_{13}\) in uniform RGGs are independent \cite{Y23b,YF25}, but they are dependent in nonuniform RGGs.

The following assumptions are fundamental to the derivation of our main results.

\begin{Assumption}\label{assumptiona}
 Let $A$ be sampled from $\mathcal{G}_{n}(f,r_n)$. Suppose $r_n=o(1)$, $nr_n=\omega(1)$, and $f(x)=g(x)I[0\leq x\leq 1]$ with \(g(x)\)  satisfying \(g(x+1)=g(x)=g(x-1)\) for all \(x\in \mathbb{R}\).  In addition, we assume $g(x)$ is bounded away from zero with a bounded fourth derivative. 
\end{Assumption}

Before proceeding, we briefly discuss Assumption \ref{assumptiona}, which underpin our main results. By Lemma \ref{propmain} (stated below), the expected degree of a node is of order \(nr_{n}\). The conditions  $r_n=o(1)$ and $nr_n=\omega(1)$ in Assumption \ref{assumptiona} imply that the network is neither too dense nor too sparse. Such assumptions are standard in the context of network data analysis \cite{A17,GMPS23}.

The condition \(f(x)=g(x)I[0\le x\le 1]\) with \(g\) satisfying \(g(x+1)=g(x)=g(x-1)\) for all \(x\in \mathbb{R}\), characterizes \(f(x)\) as a probability density function on a circle with circumference 1. This assumption is essential for obtaining a concise expression for the leading terms in (\ref{prope0})--(\ref{2prope3}) of Lemma \ref{propmain}. Were this assumption to be violated, the resulting expressions for the leading terms would be prohibitively complex.

The assumption that \(g(x)\) has a bounded fourth derivative is a technical requirement. It allows us to use the Taylor expansion to ensure the remainders of the conditional expectations in (\ref{prope0})--(\ref{2prope3}) of Lemma \ref{propmain} are \(O(r_{n}^{5})\) or \(O(r_{n}^{6})\), and the remainder terms are independent of \(X_{1}\).

The assumption that \(g(x)\) is bounded away from zero ensures the leading terms of the conditional expectations in Lemma \ref{propmain} do not vanish. This assumption is essential for tractability.  For instance, it allows us to approximate the denominator \(d_{i}\) of $\Delta_i$ by the leading term of its conditional expectation—\(2(n-1)r_{n}f(X_{i})\), as derived in (\ref{prope0})—and ensures that this term remains non-vanishing.

The conditions imposed on the probability density \(f(x)\) in Assumption \ref{assumptiona} are mild and easily satisfied; for instance, they hold for both the uniform density and the widely used von Mises density. We provide a more detailed discussion of these cases in Section \ref{wllnfp}.

In the following, we provide a series of lemmas that serve as the foundation for establishing  the weak law of large numbers. Additionally, these results motivate Assumption \ref{assumptiona}. As the proofs of the lemmas are quite lengthy, they are deferred to Section \ref{fplemprf}. The first lemma provides asymptotic expressions for the conditional probabilities of  occurrence for subgraphs such as edges and triangles. These quantities are necessary to evaluate the expectation and variance of the friendship paradox.

For convenience, we define several notations. Let $\mu_i = \mathbb{E}[d_i \mid X_i]$. 
Given an index $t\in\{1,2\}$, define $\bar{A}_{tl}=A_{tl}-\mathbb{E}[A_{tl}|X_t]$, $P_t = \sum_{\substack{j \ne k\neq t } }T_{tjk}$, $Q_t=\sum_{j\ne k \ne t}S_{tjk}$, $R_t=\sum_{j \ne k \ne l \ne t}
    T_{tjk}\bar{A}_{tl}$, where $T_{tjk}=A_{tj}A_{jk}
    -A_{tj}A_{tk}$ and $S_{tjk}=A_{tj}A_{tk}
    -A_{tj}A_{jk}A_{kt}$.

\begin{Lemma}
    \label{propmain}
    Under Assumption \ref{assumptiona}, we have
\begin{align}\label{prope0}
    \mathbb{E}[A_{12}|X_1]&=2r_nf(X_1) + \frac{f''(X_1)}{3} r_n^3 + O(r_n^5),\\     \label{prope1}
   \mathbb{E}[A_{12}A_{13}|X_1]& = 4r_n^2f^2(X_1) + \frac{4r_n^4 }{3} f(X_1) f''(X_1) + O(r_n^6),\\
   \label{prope01}
   \mathbb{E}[A_{12}A_{23}|X_1]&=4r_n^2f^2(X_1) +\frac{r_n^4}{3}\big[4(f^{\prime}(X_1))^2+6f(X_1)f^{\prime\prime}(X_1)\big]+ O(r_n^6),\\
\label{prope3}
    \mathbb{E}[A_{12}A_{13}A_{23}|X_1]&=3r_n^2f^2(X_1)+ \frac{5 r_n^4}{12} \left[(f'(X_1))^2 + 2f(X_1) f''(X_1)\right] + O(r_n^6),\\ \label{1prope3}
    \mathbb{E}[A_{12}A_{23}A_{24}|X_1]&=\mathbb{E}[A_{12}A_{13}A_{34}|X_1]=\mathbb{E}[A_{12}A_{13}A_{14}|X_1]=8r_n^3f^3(X_1)+ O(r_n^5),\\ \label{2prope3}
    \mathbb{E}[A_{12}A_{13}A_{23}A_{34}|X_1]&=\mathbb{E}[A_{12}A_{13}A_{23}A_{14}|X_1]=6r_n^3f^3(X_1)+ O(r_n^5).
\end{align}
Here, the remainder terms $O(r_n^5)$ and $O(r_n^6)$ in (\ref{prope0})-(\ref{2prope3}) do not depend on $X_1$.
\end{Lemma}

Lemma \ref{propmain} is central to approximating the expectation of \(\Delta _{i}\) and the variance of \(\mathcal{F}_{n}\). More specifically, it provides the necessary tools to approximate the conditional expectations \(\mathbb{E}[\Delta _{1}|X_{1}]\), \(\mathbb{E}[\Delta _{1}^{2}|X_{1}]\), and \(\mathbb{E}[\Delta _{1}\Delta _{2}|X_{1},X_{2}]\).

The next lemma provides an upper bound on the probability that a node has a lower-than-expected degree. Under Assumption \ref{assumptiona} and Lemma \ref{propmain}, the result follows directly from the Chernoff bound for the binomial distribution.
\begin{Lemma}\label{lem2}
Let $\delta\in (0,1)$ be a fixed constant and $f(x)\geq \lambda>0$ for a constant $\lambda$.
   Under Assumption \ref{assumptiona}, we have
\[\mathbb{P}(d_1\leq \delta \lambda nr_n)\leq \exp\big(-cnr_n(1+o(1))\big),\]
where $c$ is a positive constant that depends on $\delta$ and $f(x)$.
\end{Lemma}
 
The following lemma provides a bound on the higher-order conditional central moments of node degree \(d_{1}\).
\begin{Lemma}\label{lnem3}
   Under Assumption \ref{assumptiona}, we have
   \begin{align*}
    \mathbb{E}\left[(d_1-\mu_1)^8\right]=O(n^4r_n^4).
\end{align*} 
\end{Lemma}

The following technical lemmas are required to approximate the mean and variance of the friendship paradox index.

\begin{Lemma}\label{lem1}
   Under Assumption \ref{assumptiona}, we have
\begin{align}\label{fieq1}
    \mathbb{E}[
    P_1^2]&= O\big(n^4r_n^8+n^3r_n^3+n^2r_n^2\big),\\ \label{fieq2}
    \mathbb{E}[P_1^4]&=O(n^8r_n^{16}+n^7r_n^8+n^6r_n^6).
\end{align}
\end{Lemma}

\begin{Lemma}\label{lem3}
   Under Assumption \ref{assumptiona}, we have
   \begin{align}\label{fieq6}
    \mathbb{E}\left[\left| \frac{P_1(d_1-\mu_1)^2}{d_1\mu_1^2} \right|I[d_1\geq1]\right]
    =o(1),
\end{align} 
   \begin{align}\label{fieq7}
    \mathbb{E}\left[\left| \frac{P_1^2(d_1-\mu_1)^4}{d_1^2\mu_1^4} \right|I[d_1\geq1]\right]
    =o(1).
\end{align} 
\end{Lemma}

\begin{Lemma}\label{addlem1}
   Under Assumption \ref{assumptiona}, we have
    \begin{align}\label{9lem5eqpoof7}
\mathbb{E}[Q_1^2|X_1]=n^4r_n^4f^4(X_1)+O(n^3r_n^3).
\end{align}
\end{Lemma}

\begin{Lemma}\label{adlemnew1}
   Under Assumption \ref{assumptiona}, we have
\begin{align}\label{lemnewww2}
\mathbb{E}[R_1^2|X_1]&=O(n^5r_n^9+n^4r_n^4+n^3r_n^3).
\end{align}
\end{Lemma}

\begin{Lemma}\label{lemnew1}
   Under Assumption \ref{assumptiona}, we have
    \begin{align}\label{lemneweq1}
    \mathbb{E}\left[\frac{P_1R_2}{\mu_1^2\mu_2^2}\right]= \mathbb{E}\left[\frac{P_2R_1}{\mu_1^2\mu_2^2}\right]=O\left(r_n^3+\frac{1}{n}+\frac{1}{n^2r_n^2}\right),\\ \label{lemneweq01}
        \mathbb{E}\left[\frac{(n+1)\mathbb{E}\left[A_{12}\mid X_1\right]P_1R_2}{\mu_1^2\mu_2^2}\right]=O\left(nr_n^4+r_n+\frac{1}{nr_n}\right),\\
      \label{lem2neweq2}
    \mathbb{E}\left[\frac{R_1Q_2}{\mu_1^2\mu_2^2}\right]= \mathbb{E}\left[\frac{R_2Q_1}{\mu_1^2\mu_2^2}\right]=O\left(r_n^3+\frac{1}{nr_n}\right).
\end{align}
\end{Lemma}

\begin{Lemma}\label{lemnew3}
   Under Assumption \ref{assumptiona}, we have
    \begin{align}\label{lem2neweq1}
    \mathbb{E}\left[\frac{R_1R_2}{\mu_1^2\mu_2^2}\right]=O\left(nr_n^5+r_n+\frac{1}{nr_n}\right).
\end{align}
\end{Lemma}

\begin{Lemma}\label{lemmapp}
      Under Assumption \ref{assumptiona}, we have
\begin{align}  \label{franeq1} 
&\mathbb{E}\left[\frac{(n+1)\mathbb{E}\left[A_{12}\mid X_1\right]P_1P_2}{\mu_1^2\mu_2^2}\right]=O(nr_n^5+r_n),\\  \label{franeq2} 
&\mathbb{E}\left[\frac{P_1P_2}{\mu_1^2\mu_2^2}\right]=O\left(r_n+\frac{1}{nr_n}\right),\\ \nonumber
&\mathbb{E}\left[\frac{(n+1)^2\mathbb{E}\left[A_{12}\mid X_1\right]\mathbb{E}\left[A_{12}\mid X_2\right]P_1P_2}{\mu_1^2\mu_2^2}\right]\\  \label{franeq3} 
&=\left(n\mathbb{E}\left[\frac{\mathbb{E}\left[(A_{13}A_{34}
    -A_{13}A_{14})|X_1\right]}{\mathbb{E}\left[A_{12}\mid X_1\right]}\right]\right)^2+O\left(r_n+\frac{1}{n}+nr_n^4\right).
\end{align}
\end{Lemma}

\begin{Lemma}\label{lemqq}
     Under Assumption \ref{assumptiona}, we have
    \begin{align*} 
\mathbb{E}\left[\frac{Q_1Q_2}{\mu_1^2\mu_2^2}\right]=\left(\mathbb{E}\left[\frac{\mathbb{E}\left[\big(A_{12}A_{13}
    -A_{12}A_{13}A_{23}\big)|X_1\right]}{\big(\mathbb{E}[A_{12}|X_1]\big)^2}\right]\right)^2+O\left(\frac{1}{nr_n}\right).
\end{align*}
\end{Lemma}

\begin{Lemma}\label{lempq}
      Under Assumption \ref{assumptiona}, we have
\begin{align*}  
&\mathbb{E}\left[\frac{(n+1)\mathbb{E}\left[A_{12}\mid X_1\right]P_1Q_2}{\mu_1^2\mu_2^2}\right]\\
&=n\mathbb{E}\left[\frac{\mathbb{E}[(A_{13}A_{34}-A_{13}A_{14})|X_1]}{\mathbb{E}\left[A_{12}\mid X_1\right]}\right]\mathbb{E}\left[\frac{\mathbb{E}[(A_{25}A_{26}-A_{25}A_{26}A_{56})|X_2]}{(\mathbb{E}\left[A_{12}\mid X_2\right])^2}\right]+O\left(r_n+\frac{1}{nr_n}\right).
\end{align*}
\end{Lemma}

\section{The expectation of the friendship paradox }\label{expecfp}

Equipped with Lemma \ref{propmain}-Lemma \ref{lempq}, we now derive the weak law of large numbers for the friendship paradox.
To apply Markov’s inequality, we first determine the asymptotic behavior of the mean and variance of the friendship index.
As a preliminary step, this section provides an asymptotic expression for the expectation of the friendship paradox index.

The expectation of the friendship paradox index is equivalent to \(\mathbb{E}[\Delta _{1}]\). The primary obstacle in evaluating this expectation is the presence of the random variable \(d_{1}\) in the denominator of \(\Delta _{1}\). While one might attempt to compute this via the law of total expectation—\(\mathbb{E}[\Delta _{1}]=\sum _{k=1}^{n-1}\mathbb{E}[\Delta _{1}|d_{1}=k]\mathbb{P}(d_{1}=k)\)—this approach necessitates a complex summation involving binomial coefficients. Furthermore, inherent edge dependencies render a closed-form expressions for \(\mathbb{E}[\Delta _{1}|d_{1}=k]\) and \(\mathbb{P}(d_{1}=k)\) analytically intractable. A common heuristic is to approximate the denominator \(d_{1}\) by its mean \(\mathbb{E}[d_{1}]\), assuming the resulting error is negligible. This approximation is valid for Erdős-Rényi graphs, where the degree concentrates such that \(d_{1}=\mathbb{E}[d_{1}]+O_{P}(\sqrt{\mathbb{E}[d_{1}]})\). However, this concentration fails to hold for non-uniform RGGs, where the degree only satisfies \(d_{1}=\mathbb{E}[d_{1}]+O_{P}(\mathbb{E}[d_{1}])\). Our analysis demonstrates that a more robust strategy is to substitute \(d_{1}\) in the denominator with its conditional expectation \(\mathbb{E}[d_{1}|X_{1}]\). We then obtain an asymptotic expression for the expectation and show that the approximation error is negligible, as stated in Proposition \ref{lem4}.
\begin{Proposition}\label{lem4}
   Under Assumption \ref{assumptiona}, we have
\begin{align*}    
\mathbb{E}[\mathcal{F}_{n}]=\mathbb{E}\left[\Delta_1\right]=\frac{nr_n^3}{3}\int_0^1  [f^{\prime}(x)]^2dx+\frac{1}{4}+o(1).
\end{align*}
\end{Proposition}

Interestingly, according to Proposition \ref{lem4}, the expectation of the friendship paradox is governed by the node distribution via the derivative of its probability density. When \(f(x)\) is the uniform density, the expectation is asymptotically \(1/4\). Otherwise, it may approach an arbitrary positive constant or diverge to infinity. We provide a more detailed discussion in a later section

{\bf Proof of Proposition \ref{lem4}:} 
By Definition \ref{def0}, if $d_1=0$, then $\Delta_1=0$. 
Then, we have $\Delta_1=\Delta_1I[d_1\geq1]$ and $\mathbb{E}[\Delta_1]=\mathbb{E}[\Delta_1I[d_1\geq1]]$. 
It is straightforward to verify that
\begin{align}\label{lem4poofeq1}
    \Delta_1=\frac{P_1}{d_1}
    = \frac{\mu_1P_1-\bar d_1 P_1}{\mu_1^2}
    +\frac{P_1 (d_1-\mu_1)^2}{d_1\mu_1^2} = \frac{2\mu_1P_1-d_1 P_1}{\mu_1^2}
    +\frac{P_1 (d_1-\mu_1)^2}{d_1\mu_1^2}.
\end{align}
When $d_1=0$, $A_{1j}=0$ for all $j\in\{1,2,\dots,n\}$. In this case, $P_1=d_1=0$.  By Lemma \ref{lem3} and the properties of conditional expectation, we have
\begin{align}  \nonumber
\mathbb{E}\left[\Delta_1\right]
    &=\mathbb{E}\left[\frac{2\mu_1 P_1- d_1 P_1}{\mu_1^2}I[d_1\geq1]\right]
    +\mathbb{E}\left[\frac{P_1 (d_1-\mu_1)^2}{d_1\mu_1^2}I[d_1\geq1]\right]\\   \nonumber
    &=\mathbb{E}\left[\frac{2\mu_1 P_1- d_1 P_1}{\mu_1^2}\right]
    +\mathbb{E}\left[\frac{P_1 (d_1-\mu_1)^2}{d_1\mu_1^2}I[d_1\geq1]\right]\\  \label{lem4poofeq2}
    &=\mathbb{E}\left[\frac{\mathbb{E}\left[(2\mu_1 P_1- d_1 P_1)\big|X_1\right]}{\mu_1^2}\right]
    +o(1).
\end{align}
Next, we evaluate the expectation in (\ref{lem4poofeq2}).
Firstly, we calculate the conditional expectation. Note that $\mu_1=\mathbb{E}[d_1|X_1]=(n-1)\mathbb{E}\left[A_{12}\mid X_1\right]$. Direct calculation yields:
\begin{align}\label{revneweq1}
    d_1P_1
    &=  \sum_{j \ne k \ne l \ne 1}
    T_{1jk}A_{1l}
    + \sum_{j \ne k \ne 1}
    T_{1jk}-\sum_{j\ne k \ne 1}S_{1jk},\\ \label{revneweq2}
    2\mu_1 P_1 
    &= 2(n-1)\mathbb{E}\left[A_{12}\mid X_1\right]\sum_{j\ne k \ne 1}
    T_{1jk}.
\end{align}
For \(j\ne k\ne l\ne 1\), the terms \(T_{1jk}=A_{1j}A_{jk}-A_{1j}A_{1k}\) and \(A_{1l}\) are conditionally independent given \(X_{1}\).
Taking the expectation of both sides of the previous two equations, conditioned on \(X_{1}\), yields
 \begin{align*}   
    2\mathbb{E}\left[\mu_1 P_1 \mid X_1\right]
    &= 
    2 (n-1)^2(n-2)\big(\mathbb{E}\left[A_{13}A_{34}\mid X_1\right]
    - \mathbb{E}\left[A_{13}A_{14}\mid X_1\right]\big)\mathbb{E}\left[A_{12}\mid X_1\right],\\
    \mathbb{E}\left[d_1 P_1 \mid X_1\right]
    &= 
    (n-1)(n-2)(n-3) \big(\mathbb{E}\left[A_{13}A_{34}\mid X_1\right]-\mathbb{E}\left[A_{13}A_{14}\mid X_1\right]\big)\mathbb{E}\left[A_{12}\mid X_1\right]\\
    &+
    (n-1)(n-2)\big(\mathbb{E}\left[A_{13}A_{34}\mid X_1\right]-\mathbb{E}\left[A_{13}A_{14}\mid X_1\right]\big)
    \\
    &- (n-1)(n-2)\big( \mathbb{E}\left[A_{13}A_{14}\mid X_1\right]-\mathbb{E}\left[A_{13}A_{14}A_{12}\mid X_1\right]
    \big).
\end{align*}
Then the conditional expectation $\mathbb{E}\left[(2\mu_1 P_1- d_1 P_1)|X_1\right]$ is equal to
\begin{align*} 
    &\mathbb{E}\left[(2\mu_1 P_1- d_1 P_1)|X_1\right]\\
    &= (n+1)(n-1)(n-2) \big(\mathbb{E}\left[A_{13}A_{34}\mid X_1\right]-\mathbb{E}\left[A_{13}A_{14}\mid X_1\right]\big)\mathbb{E}\left[A_{12}\mid X_1\right]\\
    &-
    (n-1)(n-2)\big(\mathbb{E}\left[A_{13}A_{34}\mid X_1\right]-\mathbb{E}\left[A_{13}A_{14}\mid X_1\right]\big)
    \\
    &+ (n-1)(n-2)\big( \mathbb{E}\left[A_{13}A_{14}\mid X_1\right]-\mathbb{E}\left[A_{13}A_{14}A_{12}\mid X_1\right]
    \big).
\end{align*}
By equation (\ref{prope1}) and equation (\ref{prope01}) of Lemma \ref{propmain}, the second term of the foregoing expression is  equal to
\begin{align*}
(n-1)(n-2)\big(\mathbb{E}\left[A_{13}A_{34}\mid X_1\right]-\mathbb{E}\left[A_{13}A_{14}\mid X_1\right]\big)&=O(n^2r_n^4),
\end{align*}
where the upper bound $O(n^2r_n^4)$ does not depend on $X_1$. According to (\ref{prope0}) of Lemma \ref{propmain}, \(\mu _{1}=2(n-1)r_{n}f(X_{1})+O(r_{n}^{3})\). Recall that \(r_{n}=o(1)\) and \(f(x)\) is bounded   away from zero. It then follows that
\begin{align*}
    \frac{\mathbb{E}\left[(2\mu_1 P_1 - d_1P_1)\big|X_1\right]}{\mu_1^2}
    &=n\frac{
    \mathbb{E}\left[A_{13}A_{34}\mid X_1\right]
    -\mathbb{E}\left[A_{13}A_{14}\mid X_1\right]}
    {\mathbb{E}\left[A_{12}\mid X_1\right]}
    \\
    &\quad +\frac{
    \mathbb{E}\left[A_{13}A_{14}\mid X_1\right]
    -\mathbb{E}\left[A_{12}A_{23}A_{31}\mid X_1\right]}
    {\big(\mathbb{E}\left[A_{12}\mid X_1\right]\big)^2}
    +o(1).
\end{align*}
In view of (\ref{lem4poofeq2}), we have
   \begin{align}\nonumber
       \mathbb{E}\left[\Delta_1\right]
    &=n\mathbb{E}\left[\frac{\mathbb{E}[A_{12}A_{23}|X_1]-\mathbb{E}[A_{12}A_{13}|X_1]}{\mathbb{E}[A_{12}|X_1]}\right]\\ \label{lem4poofeq3}
    &\quad+\mathbb{E}\left[\frac{\mathbb{E}[A_{12}A_{13}|X_1]-\mathbb{E}[A_{12}A_{13}A_{23}|X_1]}{\big(\mathbb{E}[A_{12}|X_1]\big)^2}\right]+o(1).
\end{align} 

Next, we will evaluate the two expectations in (\ref{lem4poofeq3}) below. By Assumption \ref{assumptiona}, $r_n=o(1)$, and \(f(x)\) is bounded away from zero and has a bounded fourth derivative. Thus, \(f(x)\ge c>0\) for some constant \(c\), while $f(X_1)$, \((f^{\prime }(X_{1}))^{2}\) and \(f(X_{1})f^{\prime \prime }(X_{1})\) are  bounded. It then follows from equation (\ref{prope0}), equation (\ref{prope1}), equation (\ref{prope01}) and  equation (\ref{prope3}) of Lemma \ref{propmain} that
\begin{align}\nonumber
\mathbb{E}\left[
    \frac{
    \mathbb{E}\left[A_{13}A_{14}\mid X_1\right]
    -\mathbb{E}\left[A_{12}A_{23}A_{31}\mid X_1\right]}
    {\big(\mathbb{E}\left[A_{12}\mid X_1\right]\big)^2}\right]
    &= \int_0^1
    \frac
    {r_n^2f^2(x_1)+O(r_n^4)}
    {4r_n^2f^2(x_1)+O(r_n^4)}f(x_1)dx_1
    +o(1)\\ \label{lem4poofeq4}
    &= \frac{1}{4} + o(1),
\end{align}
and
\begin{align}\nonumber
    \mathbb{E}\left[\frac{\mathbb{E}[A_{12}A_{23}|X_1]-\mathbb{E}[A_{12}A_{13}|X_1]}{\mathbb{E}[A_{12}|X_1]}\right]&=\int_0^1\frac{\frac{r_n^4}{3}\big[4(f^{\prime}(x))^2+2f(x)f^{\prime\prime}(x)\big]+O(r_n^6)}{2r_nf(x)+O(r_n^3)}f(x)dx\\\label{lem4poofeq5}
    &=\frac{r_n^3}{3}\int_0^1  \big[2(f^{\prime}(x))^2+f(x)f^{\prime\prime}(x)\big] dx+o(1).
\end{align}
By integration by part, we get
\begin{align*}
 \int_0^1f(x)f^{\prime\prime}(x)dx
    =f(x)f^{\prime}(x)\big|_0^1-\int_0^1\left(f^{\prime}(x)\right)^2dx.
\end{align*}
By Assumption \ref{assumptiona}, $f(x)=g(x)I[0\leq x\leq 1]$ with \(g(x)\)  satisfying \(g(x+1)=g(x)=g(x-1)\) for all \(x\in \mathbb{R}\). Note that \(g'(x+1)=g'(x)=g'(x-1)\). Then $f(1)=g(1)=g(0)=f(0)$ and $f'(1)=g'(1)=g'(0)=f'(0)$. Then $f(x)f^{\prime}(x)\big|_0^1=0$, and
\begin{align*}
 \int_0^1f(x)f^{\prime\prime}(x)dx
    =-\int_0^1\left(f^{\prime}(x)\right)^2dx.
\end{align*}
It then follows from (\ref{lem4poofeq5}) that
\begin{align}\label{lem4poofeq6}
    \mathbb{E}\left[\frac{\mathbb{E}[A_{12}A_{23}|X_1]-\mathbb{E}[A_{12}A_{13}|X_1]}{\mathbb{E}[A_{12}|X_1]}\right]=\frac{r_n^3}{3}\int_0^1  (f^{\prime}(x))^2dx+o(1).
\end{align}

Combining (\ref{lem4poofeq3}) through (\ref{lem4poofeq6}) completes the proof of Proposition \ref{lem4}.

\qed

\section{The  variance of the friendship paradox }\label{fdvar}

In this section, we compute the second moment of \(\Delta _{1}\) and the cross-moment \(\mathbb{E}[\Delta _{1}\Delta _{2}]\) to characterize their asymptotic behavior. These quantities are essential for establishing the asymptotic order of the variance for the friendship paradox. This derivation is computationally intensive and requires careful treatment, as demonstrated in the proofs of Lemmas \ref{lem1} through \ref{lemnew3}.

\begin{Proposition}\label{lem5}
   Under Assumption \ref{assumptiona}, we have
\begin{align*}    \mathbb{E}\left[\Delta_1\Delta_2\right]+\frac{\mathbb{E}\left[\Delta_1^2\right]-\mathbb{E}\left[\Delta_1\Delta_2\right]}{n}=\big(\mathbb{E}\left[\Delta_1\right]\big)^2(1+o(1)).
\end{align*}
\end{Proposition}
Proposition \ref{lem5} provides the asymptotic order of the second moment of the friendship paradox. Combined with Proposition \ref{lem4}, it establishes the weak law of large numbers, as demonstrated in the subsequent section.

{\bf Proof of Proposition \ref{lem5}:} 
We prove Proposition \ref{lem5} in two steps: first, we bound \(\mathbb{E}[\Delta _{1}^{2}]\), and then we compute \(\mathbb{E}[\Delta _{1}\Delta _{2}]\).

{ \bf Step 1: bound \(\mathbb{E}[\Delta _{1}^{2}]\).}  We provide an upper bound for \(\mathbb{E}[\Delta _{1}^{2}]\). By Definition \ref{def0}, if $d_1=0$, then $\Delta_1=0$. 
Then $\Delta_1^2=\Delta_1^2I[d_1\geq1]$. We will bound $\mathbb{E}[\Delta_1^2I[d_1\geq1]]$.
In view of (\ref{revneweq1}) and (\ref{revneweq2}),
it is easy to verify that
\begin{align}\nonumber
     2\mu_1 P_1-d_1P_1
    &= (n+1)\mathbb{E}\left[A_{12}\mid X_1\right]\sum_{j\ne k \ne 1}
  T_{1jk}
    +\sum_{j\ne k \ne 1} S_{1jk}-\sum_{j\ne k \ne 1}T_{1jk}-\sum_{j \ne k \ne l \ne 1}
    T_{1jk}\bar{A}_{1l}\\ \label{lem5eqpoof1}
    &=(n+1)\mathbb{E}\left[A_{12}\mid X_1\right]P_1+Q_1-P_1-R_1.
\end{align}
Then, we have
\begin{align}\nonumber
&\mathbb{E}\left[\left(\frac{2\mu_1P_1-d_1 P_1}{\mu_1^2}\right)^2\Bigg|X_1\right]\\ \label{lem5eqpoof2}
&\leq 4\frac{(n+1)^2\big(\mathbb{E}\left[A_{12}\mid X_1\right]\big)^2\mathbb{E}[P_1^2|X_1]+\mathbb{E}[Q_1^2|X_1]+\mathbb{E}[P_1^2|X_1]+\mathbb{E}[R_1^2|X_1]}{\mu_1^4}.
\end{align}
When $d_1=0$, $A_{1j}=0$ for all $j\in\{1,2,\dots,n\}$. In this case, $P_1=d_1=0$.
By Lemma \ref{lem1}, Lemma \ref{addlem1}, and Lemma \ref{adlemnew1}, it follows that
\begin{align}\label{tueseq1}
\mathbb{E}\left[\left(\frac{2\mu_1P_1-d_1 P_1}{\mu_1^2}\right)^2I[d_1\geq1]\Bigg|X_1\right]=\mathbb{E}\left[\left(\frac{2\mu_1P_1-d_1 P_1}{\mu_1^2}\right)^2\Bigg|X_1\right]=O(1+nr_n^3+n^2r_n^6).
\end{align}
Combining (\ref{lem4poofeq1}), (\ref{tueseq1}) and Lemma \ref{lem3} yields
\begin{align*}
  \mathbb{E}[\Delta_1^2]&=\mathbb{E}\big[\mathbb{E}[\Delta_1^2|X_1]\big]\\
  &\leq 2\mathbb{E}\left[\mathbb{E}\left[\left(\frac{2\mu_1P_1-d_1 P_1}{\mu_1^2}\right)^2I[d_1\geq1]\Bigg|X_1\right]\right]+2\mathbb{E}\left[\left(\frac{P_1 (d_1-\mu_1)^2}{d_1\mu_1^2}\right)^2I[d_1\geq1]\right]\\
&=O(1+nr_n^3+n^2r_n^6).
\end{align*}
By Proposition \ref{lem4}, $(\mathbb{E}[\mathcal{F}_n])^2=\Theta(1+nr_n^3+n^2r_n^6)$. 
It then follows that
\begin{align}\label{lem5eqpoof13}
  \frac{\mathbb{E}[\Delta_1^2]}{n}=O\left(\frac{1}{n}+r_n^3+nr_n^6\right)=o(  (\mathbb{E}[\mathcal{F}_n])^2).
 \end{align}

{\bf Step 2: find $\mathbb{E}\left[\Delta_1\Delta_2\right]$.} 
Next, we find an asymptotic expression for \(\mathbb{E}[\Delta _{1}\Delta _{2}]\). By Definition \ref{def0}, if $d_i=0$, then $\Delta_i=0$. 
Hence, $\Delta_1\Delta_2=\Delta_1\Delta_2I[d_1\geq1]I[d_2\geq1]$. In addition,
if $d_i=0$, then $A_{ij}=0$ for all $j\in\{1,2,\dots,n\}$. In this case, $P_i=d_i=0$.
By (\ref{lem4poofeq1}),  \(\Delta _{1}\Delta _{2}\) is equal to
\begin{align*}
    \Delta_1\Delta_2&=\frac{2\mu_1P_1-d_1 P_1}{\mu_1^2}\frac{2\mu_2P_2-d_2 P_2}{\mu_2^2}
    +\frac{P_1 (d_1-\mu_1)^2}{d_1\mu_1^2}\frac{P_2 (d_2-\mu_2)^2}{d_2\mu_2^2}I[d_1\geq1]I[d_2\geq1]\\
    &\quad +\frac{P_1 (d_1-\mu_1)^2}{d_1\mu_1^2}\frac{2\mu_2P_2-d_2 P_2}{\mu_2^2}I[d_1\geq1]+\frac{2\mu_1P_1-d_1 P_1}{\mu_1^2}\frac{P_2 (d_2-\mu_2)^2}{d_2\mu_2^2}I[d_2\geq1].
\end{align*}
By the Cauchy-Schwarz inequality, equation (\ref{tueseq1}) and Lemma \ref{lem3}, we have
\begin{align*}
&\left|\mathbb{E}\left[\frac{P_1 (d_1-\mu_1)^2}{d_1\mu_1^2}\frac{P_2 (d_2-\mu_2)^2}{d_2\mu_1^2}I[d_1\geq1]I[d_2\geq1]\right]\right|\\
&\leq \sqrt{\mathbb{E}\left[\left(\frac{P_1 (d_1-\mu_1)^2}{d_1\mu_1^2}\right)^2I[d_1\geq1]\right]\mathbb{E}\left[\left(\frac{P_2 (d_2-\mu_2)^2}{d_2\mu_1^2}\right)^2I[d_2\geq1]\right]}\\
&=o(1),
\end{align*}
\begin{align*}
&\left|\mathbb{E}\left[\frac{P_1 (d_1-\mu_1)^2}{d_1\mu_1^2}\frac{2\mu_2P_2-d_2 P_2}{\mu_2^2}I[d_1\geq1]\right]\right|\\
&\leq \sqrt{\mathbb{E}\left[\left(\frac{P_1 (d_1-\mu_1)^2}{d_1\mu_1^2}\right)^2I[d_1\geq1]\right]\mathbb{E}\left[\left(\frac{2\mu_2P_2-d_2 P_2}{\mu_2^2}\right)^2\right]}\\
&=o(\sqrt{1+nr_n^3+n^2r_n^6}),
\end{align*}
\begin{align*}
&\left|\mathbb{E}\left[\frac{2\mu_1P_1-d_1 P_1}{\mu_1^2}\frac{P_2 (d_2-\mu_2)^2}{d_2\mu_2^2}I[d_2\geq1]\right]\right|\\
&\leq \sqrt{\mathbb{E}\left[\left(\frac{2\mu_1P_1-d_1 P_1}{\mu_1^2}\right)^2\right]\mathbb{E}\left[\left(\frac{P_2 (d_2-\mu_2)^2}{d_2\mu_2^2}\right)^2I[d_2\geq1]\right]}\\
&=o(\sqrt{1+nr_n^3+n^2r_n^6}).
\end{align*}
Then it follows that
\begin{align} \label{tueseq2}   \mathbb{E}\left[\Delta_1\Delta_2\right]&=\mathbb{E}\left[\frac{2\mu_1P_1-d_1 P_1}{\mu_1^2}\frac{2\mu_2P_2-d_2 P_2}{\mu_2^2}\right]
    +o(\sqrt{1+nr_n^3+n^2r_n^6}).
\end{align}

Next, we find the leading term of the expectation in (\ref{tueseq2}). Recall in (\ref{lem5eqpoof1}) that
\begin{align*}
     2\mu_t P_t-d_tP_t&=(n+1)\mathbb{E}\left[A_{12}\mid X_t\right]P_t+Q_t-P_t-R_t,\ \ \ \ \ \ t\in\{1,2\}.
\end{align*}
Then straightforward calculation yields
\begin{align}\nonumber
&\frac{2\mu_1P_1-d_1 P_1}{\mu_1^2}\frac{2\mu_2P_2-d_2 P_2}{\mu_2^2}\\ \nonumber
&=\frac{(n+1)^2\mathbb{E}\left[A_{12}\mid X_1\right]\mathbb{E}\left[A_{12}\mid X_2\right]P_1P_2}{\mu_1^2\mu_2^2}+\frac{Q_1Q_2}{\mu_1^2\mu_2^2}\\ \nonumber
&\quad+\frac{(n+1)\mathbb{E}\left[A_{12}\mid X_1\right]P_1Q_2}{\mu_1^2\mu_2^2}+\frac{(n+1)\mathbb{E}\left[A_{12}\mid X_2\right]P_2Q_1}{\mu_1^2\mu_2^2}\\ \nonumber
&\quad-\frac{(n+1)\mathbb{E}\left[A_{12}\mid X_1\right]}{\mu_1^2\mu_2^2}\big(P_1P_2+P_1R_2\big)-\frac{(n+1)\mathbb{E}\left[A_{12}\mid X_2\right]}{\mu_1^2\mu_2^2}\big(P_1P_2+P_2R_1\big)\\  \label{tueseq3} 
&\quad-\frac{Q_1P_2+Q_1R_2+Q_2P_1+Q_2R_1}{\mu_1^2\mu_2^2} +\frac{P_1P_2+P_1R_2+P_2R_1+R_2R_1}{\mu_1^2\mu_2^2}.
\end{align}
We first demonstrate that the last four terms of (\ref{tueseq3}) are of negligible order, then derive the asymptotic expressions for the remaining terms. By Lemma \ref{lemmapp},
 for $t\in\{1,2\}$, we have
\begin{align}\nonumber
\mathbb{E}\left[\frac{(n+1)\mathbb{E}\left[A_{12}\mid X_t\right]}{\mu_1^2\mu_2^2} P_1P_2\right]&=O\left(r_n^2+\frac{1}{nr_n}+nr_n^5\right),
\end{align}
\begin{align}\label{tueseq04}
\mathbb{E}\left[\frac{P_1P_2}{\mu_1^2\mu_2^2} \right]=O\left(\frac{1}{nr_n}+r_n\right).
\end{align}
 By Lemma \ref{lem1}, Lemma \ref{addlem1} and the Cauchy–Schwarz inequality and the assumption that $r_n=o(1)$ and $nr_n=\omega(1)$, we get
\begin{align}\label{frifineq1}
\mathbb{E}\left[\left|\frac{P_2Q_1}{\mu_1^2\mu_2^2} \right|\right]=O\left(\frac{1}{n^4r_n^4}\right)\mathbb{E}\left[|P_2Q_1|\right]=O\left(\frac{\sqrt{n^3r_n^5+n^2r_n^2+n^4r_n^8}}{n^2r_n^2}\right)=o(1),
\end{align}
\begin{align}\label{frifineq2}
\mathbb{E}\left[\left|\frac{P_1Q_2}{\mu_1^2\mu_2^2} \right|\right]=O\left(\frac{1}{n^4r_n^4}\right)\mathbb{E}\left[|P_1Q_2|\right]=O\left(\frac{\sqrt{n^3r_n^5+n^2r_n^2+n^4r_n^8}}{n^2r_n^2}\right)=o(1).
\end{align}
By Lemma \ref{lemnew1} and Lemma \ref{lemnew3},  we have
\begin{align}\label{frifineq03}
    \mathbb{E}\left[\frac{(n+1)\mathbb{E}\left[A_{12}\mid X_1\right]}{\mu_1^2\mu_2^2}P_1R_2\right]=    \mathbb{E}\left[\frac{(n+1)\mathbb{E}\left[A_{12}\mid X_2\right]}{\mu_1^2\mu_2^2}P_2R_1\right]=O\left(nr_n^4+r_n+\frac{1}{nr_n}\right),
\end{align}
\begin{align}\label{frifineq3}
\mathbb{E}\left[\frac{P_1P_2+P_1R_2+P_2R_1+R_2R_1}{\mu_1^2\mu_2^2}\right]=O\left(nr_n^5+r_n+\frac{1}{nr_n}\right).
\end{align}
\begin{align}\label{frifineq03}
\mathbb{E}\left[\frac{Q_1P_2+Q_1R_2+Q_2P_1+Q_2R_1}{\mu_1^2\mu_2^2} \right]=O\left(r_n+\frac{1}{nr_n}\right).
\end{align}

Accordint to Proposition \ref{lem4}, $(\mathbb{E}[\mathcal{F}_n])^2=\Theta(1+nr_n^3+n^2r_n^6)$. 
Combining (\ref{tueseq2})-(\ref{frifineq03}) yields
\begin{align*}   \mathbb{E}\left[\Delta_1\Delta_2\right]&=\mathbb{E}\left[\frac{(n+1)^2\mathbb{E}\left[A_{12}\mid X_1\right]\mathbb{E}\left[A_{12}\mid X_2\right]P_1P_2}{\mu_1^2\mu_2^2}\right]+\mathbb{E}\left[\frac{(n+1)\mathbb{E}\left[A_{12}\mid X_1\right]P_1Q_2}{\mu_1^2\mu_2^2}\right]\\
    &\quad+\mathbb{E}\left[\frac{Q_1Q_2}{\mu_1^2\mu_2^2}\right]+\mathbb{E}\left[\frac{(n+1)\mathbb{E}\left[A_{12}\mid X_2\right]P_2Q_1}{\mu_1^2\mu_2^2}\right]
    +o\big((\mathbb{E}[\mathcal{F}_n])^2\big).
\end{align*}
In view of Lemma \ref{lemnew3}, Lemma \ref{lemmapp} and Lemma \ref{lemqq}, we have
\begin{align*}   \mathbb{E}\left[\Delta_1\Delta_2\right]&=\left(n\mathbb{E}\left[\frac{\mathbb{E}\left[(A_{13}A_{34}
    -A_{13}A_{14})|X_1\right]}{\mathbb{E}\left[A_{12}\mid X_1\right]}\right]\right)^2+O\left(r_n+nr_n^4+\frac{1}{nr_n}\right)\\
    &+2n\mathbb{E}\left[\frac{\mathbb{E}[(A_{13}A_{34}-A_{13}A_{14})|X_1]}{\mathbb{E}\left[A_{12}\mid X_1\right]}\right]\mathbb{E}\left[\frac{\mathbb{E}[(A_{25}A_{26}-A_{25}A_{26}A_{56})|X_2]}{(\mathbb{E}\left[A_{12}\mid X_2\right])^2}\right]\\
    &+\left(\mathbb{E}\left[\frac{\mathbb{E}\left[\big(A_{12}A_{13}
    -A_{12}A_{13}A_{23}\big)|X_1\right]}{\big(\mathbb{E}[A_{12}|X_1]\big)^2}\right]\right)^2+o(\sqrt{1+nr_n^3+n^2r_n^6})\\
    &=\big(\mathbb{E}\left[\Delta_1\right]\big)^2+O\left(r_n+nr_n^4+\frac{1}{nr_n}\right)+o(\sqrt{1+nr_n^3+n^2r_n^6}).
\end{align*}

By (\ref{lem5eqpoof13}) and the fact that  $(\mathbb{E}[\Delta_1])^2=(\mathbb{E}[\mathcal{F}_n])^2=\Theta(1+nr_n^3+n^2r_n^6)$, we conclude that
\begin{align*}    \mathbb{E}\left[\Delta_1\Delta_2\right]+\frac{\mathbb{E}\left[\Delta_1^2\right]-\mathbb{E}\left[\Delta_1\Delta_2\right]}{n}=\big(\mathbb{E}\left[\Delta_1\right]\big)^2(1+o(1)).
\end{align*}

Then the proof of Proposition \ref{lem5} is complete.

\qed

\section{The WLLN for the friendship paradox }\label{wllnfp}
In this section, we build on the results of Proposition \ref{lem4} and Proposition \ref{lem5} to establish the weak law of large numbers for the friendship paradox via Markov’s inequality.

 \begin{Theorem}\label{mainthm}
Suppose Assumption \ref{assumptiona} holds. 
Then, for any positive constant $\epsilon$, we have
     \begin{equation}        \lim_{n\rightarrow\infty}\mathbb{P}\left(\left|\frac{\mathcal{F}_n}{\mathbb{E}[\mathcal{F}_n]}-1\right|>\epsilon\right)=0,
     \end{equation}
    where the expectation $\mathbb{E}[\mathcal{F}_n]$ has the following asymptotic expression:
     \begin{equation}\label{f0rrgg}
    \mathbb{E}[\mathcal{F}_n] = 
    \frac{nr_n^3}{3}\int_0^1  [f^{\prime}(x)]^2dx+\frac{1}{4}+o(1).
     \end{equation}
 \end{Theorem}

 {\bf Proof of Theorem \ref{mainthm}: }
Note that the second moment of $\mathcal{F}_n$ can be expressed as follows:
\begin{align*}   
\mathbb{E}\left[\mathcal{F}_n^2\right]&=\frac{1}{n^2}\sum_{i\neq j}\mathbb{E}\left[\Delta_i\Delta_j\right]+\frac{1}{n^2}\sum_{i}\mathbb{E}\left[\Delta_i^2\right]\\
&=\mathbb{E}\left[\Delta_1\Delta_2\right]+\frac{\mathbb{E}\left[\Delta_1^2\right]-\mathbb{E}\left[\Delta_1\Delta_2\right]}{n}.
\end{align*}
By Proposition \ref{lem5}, the variance of \(\mathcal{F}_{n}\) satisfies:
\begin{align*}   
\mathbb{E}[(\mathcal{F}_n-\mathbb{E}\left[\mathcal{F}_n\right])^2]=\mathbb{E}\left[\mathcal{F}_n^2\right]-\big(\mathbb{E}\left[\mathcal{F}_n\right]\big)^2=o\big(\big(\mathbb{E}\left[\mathcal{F}_n\right]\big)^2\big).
\end{align*}
Then, for any positive constant \(\epsilon \), it follows from Markov’s inequality that
\begin{align*}
\mathbb{P}\left(\left|\frac{\mathcal{F}_n}{\mathbb{E}[\mathcal{F}_n]}-1\right|>\epsilon\right)\leq \frac{\mathbb{E}[(\mathcal{F}_n-\mathbb{E}\left[\mathcal{F}_n\right])^2]}{\epsilon^2\big(\mathbb{E}\left[\mathcal{F}_n\right]\big)^2}=o(1).
\end{align*}
Proposition \ref{lem4} yields the expression in (\ref{f0rrgg}). Then the proof of  Theorem \ref{mainthm} is complete. 

\qed

Based on Theorem \ref{mainthm}, the friendship paradox $\mathcal{F}_n$ converges in probability to its expectation. Theorem \ref{mainthm} illustrates the notable difference in properties between uniform and non-uniform RGGs. For the uniform RGGs $\mathcal{G}_{n}(f,r_n)$ where $f(x)$ is equal to the constant 1, the  friendship paradox is asymptotically equal to $\frac{1}{4}$, that is,
\[\mathcal{F}_n=\frac{1}{4}+o_P(1).\]
In contrast, 
for the nonuniform RGGs $\mathcal{G}_{n}(f,r_n)$, the friendship paradox $\mathcal{F}_n$ exhibits distinct behaviors that depend on the sparsity of the graph. In the relatively sparse regime where $nr_n^3=o(1)$, the friendship paradox remains asymptotically equivalent to $\frac{1}{4}$, mirroring the behavior of the uniform case. In the intermediate sparse regime where $nr_n^3\rightarrow\lambda$ for some constant $\lambda>0$, the friendship paradox  converges to $\frac{1}{4}$ plus a constant that depends on the derivative of $f(x)$ as follows:
    \begin{equation*} 
    \mathcal{F}_n = 
    \frac{\lambda}{3}\int_0^1\left[f^{\prime}(x)\right]^2dx+\frac{1}{4}+o_P(1).
     \end{equation*}
If $f(x)$ is not a constant, this limit differs from  $\frac{1}{4}$, the value of the relatively sparse case or uniform case. Within the relatively dense regime  $nr_n^3=\omega(1)$, the friendship paradox is no longer bounded, but instead exhibits growth of order  $nr_n^3$.

Assumption \ref{assumptiona} is non-restrictive; they are satisfied by the uniform density and the widely-used von Mises distribution. Recall that the von Mises densities  on $[0,1]$ is given by
    \begin{equation}\label{vonMises}
    f(x)=g(x)I[0\leq x\leq 1],\hskip 1cm g(x)=\frac{e^{\kappa \cos(2\pi x-\mu)}}{\mathcal{I}_0},
\end{equation}
    where $\kappa\geq0$, $\mu\in[0,1]$, and
    \[\mathcal{I}_0=\frac{1}{2\pi}\int_0^{2\pi}e^{\kappa \cos(x)}dx.\]
    When $\kappa=0$, the density reduces to the uniform distribution on $[0,1]$. As $\kappa$ increases, the distribution becomes increasingly concentrated around the mean $\mu$.   For large $\kappa$, it asymptotically approaches the normal distribution with mean $\mu$ and variance $\frac{1}{\kappa}$.

The function $g(x)$ in (\ref{vonMises}) is smooth.  Its derivatives are bounded and continuous to any arbitrary order.  Furthermore, it is strictly lower-bounded by the constant $(e^{\kappa}I_0(\kappa))^{-1}$,  guaranteeing that the density remains positive across the domain.  Given that the cosine function is periodic with period \(2\pi \), the function $g(x)$ satisfies the boundary conditions $g(x+1)=g(x-1)=g(x)$ for all $x\in\mathbb{R}$.   The von Mises distribution satisfies all the conditions in Assumption \ref{assumptiona}.

    By specializing Theorem \ref{mainthm} to the von Mises distribution, we obtain the following corollary:

\begin{Corollary}
    Suppose $f(x)$ is the von Mises density given in (\ref{vonMises}), and Assumption \ref{assumptiona} holds. Then
         \begin{equation*}
   \mathcal{F}_n = 
    nr_n^3\tau_f+\frac{1}{4}+o_P(1).
     \end{equation*}
     where 
     \begin{equation}\label{cfrrgg}
     \tau_f=\frac{4\pi^2\kappa^2}{3\mathcal{I}_0^2}\int_0^1e^{2\kappa \cos(2\pi x-\mu)}\sin^2(2\pi x-\mu)dx.
     \end{equation}
\end{Corollary}
The integral in (\ref{cfrrgg}) is analytically intractable. To illustrate sensitivity of the limiting value  of $\mathcal{F}_n$ to the distribution's parameters, we compute numerical values of $\tau_f$
across a range of parameters $\kappa$ and $\mu$ in Table \ref{tab1}. The value of $\tau_f$  remains invariant  with respect to the mean $\mu$. This result is intuitively consistent with the underlying geometry of RGGs and the assumption that \(f(x)\) is a density on the circle. Because the distance metric is translation-invariant and $g(x)$ is periodic, shifting the mean direction \(\mu \) does not alter the joint distribution of edges. However, the value of $\tau_f$ is highly sensitive to the concentration parameter \(\kappa \). Recall that when $\kappa=0$, the von Mises distribution is the uniform distribution. For small values of $\kappa$, $\tau_f$ is small; the limiting value of the friendship index gets closer to \(1/4\), the value for the uniform case. For large $\kappa$, $\tau_f$  increases significantly. 

\begin{table}[!h]
\begin{center}
\caption{Numeric values of $\tau_f$. }
\label{tab1}
\begin{tabular}{ |c| c| c|c| c| c| } 
\hline
$(\kappa,\ \mu)$ & (0.1,\ 0.1) & (0.5,\ 0.1) & (1.0,\ 0.1) & (5,\ 0.1) & (10,\ 0.1) \\
 \hline
$\tau_f$ & 0.0657 & 1.6439& 6.5293 & 118.4242 &  352.3377
 \\ 
 \hline  
\hline
$(\kappa,\ \mu)$ & (0.1,\ 0.3) & (0.5,\ 0.3) & (1.0,\ 0.3) & (5,\ 0.3) & (10,\ 0.3) \\
 \hline
$\tau_f$ & 0.0657 & 1.6439 & 6.5293 &  118.4242    & 352.3377 \\ 
 \hline  
 \hline
$(\kappa,\ \mu)$ & (0.1,\ 0.5) & (0.5,\ 0.5) & (1.0,\ 0.5) & (5,\ 0.5) & (10,\ 0.5) \\
 \hline
$\tau_f$ &0.0657 & 1.6439  &6.5293 & 118.4242  &352.3377
\\ 
 \hline
\end{tabular}
\end{center}
\end{table}

\section{Proof of lemmas}\label{fplemprf}
In this section, we provide detailed proofs of the lemmas.  Recall the notations: $\mu_i = \mathbb{E}[d_i \mid X_i]$,  
 $\bar{A}_{tl}=A_{tl}-\mathbb{E}[A_{tl}|X_t]$, $P_t = \sum_{\substack{j \ne k\neq t } }T_{tjk}$, $Q_t=\sum_{j\ne k \ne t}S_{tjk}$, $R_t=\sum_{j \ne k \ne l \ne t}
    T_{tjk}\bar{A}_{tl}$, where $T_{tjk}=A_{tj}A_{jk}
    -A_{tj}A_{tk}$ and $S_{tjk}=A_{tj}A_{tk}
    -A_{tj}A_{jk}A_{kt}$ and $t\in\{1,2\}$.

\subsection{Proof of Lemma \ref{propmain}} Equations (\ref{prope0})-(\ref{prope3}) are proved in \cite{YS26}.  We only need to prove (\ref{1prope3}) and (\ref{2prope3}).

{\bf Proof of (\ref{1prope3}):} Firstly, we consider the first equality in (\ref{1prope3}). By (\ref{prope1}) and properties of conditional expectation, we have
\begin{align}\nonumber
    \mathbb{E}[A_{12}A_{23}A_{24}|X_1]
    &=\mathbb{E}[A_{12}\mathbb{E}[A_{23}A_{24}|X_1,X_2]|X_1]\\ \label{lem1peq1}
    &=(2r_n)^2\mathbb{E}[A_{12}f^2(X_2)|X_1]+\frac{4r_n^4 }{3}\mathbb{E}[A_{12}f(X_2) f''(X_2)|X_1]+ O(r_n^6).
\end{align}
By assumption, $f(x)$ has bounded fourth derivative. Since the interval \([0,1]\) is compact, it follows that both \(f\) and its second derivative \(f^{\prime \prime }\)  are uniformly bounded on this domain. Then
\begin{align}\label{lem1peq2}
\left|\frac{4r_n^4 }{3}\mathbb{E}[A_{12}f(X_2) f''(X_2)|X_1]\right|\leq\frac{4r_n^4 }{3}\mathbb{E}[A_{12}|f(X_2) f''(X_2)||X_1]=O(r_n^4)\mathbb{E}[A_{12}|X_1]=O(r_n^5).
\end{align}

Moreover, the conditional expectation in the first term of (\ref{lem1peq1}) can be expressed as
\begin{align}\nonumber
\mathbb{E}[A_{12}f^2(X_2)|X_1]&=I[r_n\leq X_1\leq 1-r_n]\int_{X_1-r_n}^{X_1+r_n}f^3(x)dx\\\nonumber
&+I[0\leq X_1\leq r_n]\left(\int_{0}^{X_1+r_n}f^3(x)dx+\int_{1-r_n+X_1}^{1}f^3(x)dx\right)\\ \label{lem1peq3}
&+I[1-r_n\leq X_1\leq 1]\left(\int_{X_1-r_n}^1f^3(x)dx+\int_{0}^{r_n+X_1-1}f^3(x)dx\right).
\end{align}
For $X_1-r_n\leq x\leq X_1+r_n$, we have $|x-X_1|\leq r_n$. By the Taylor expansion and the assumption that $f(x)$ has bounded fourth derivative, we have
\begin{align}\nonumber
\int_{X_1-r_n}^{X_1+r_n}f^3(x)dx&=\int_{X_1-r_n}^{X_1+r_n}\big[f^3(X_1)+3f^2(X_1)f'(X_1)(x-X_1)+O(r_n^2)\big]dx\\\label{lem1peq4}
&=2r_nf^3(X_1)+O(r_n^3).
\end{align}

For $0\leq X_1\leq r_n=o(1)$, we have $-1\leq X_1-r_n\leq 0$.
By the change of variable in calculus and the assumption that $f(x)=g(x)I[0\leq x\leq 1]$ with \(g(x)\)  satisfying \(g(x+1)=g(x)=g(x-1)\) for all \(x\in \mathbb{R}\), we have
\begin{align}\nonumber
\int_{1-r_n+X_1}^{1}f^3(x)dx&=\int_{1-r_n+X_1}^{1}g^3(x)I[0\leq x\leq 1]dx\\ \nonumber
&=\int_{X_1-r_n}^{0}g^3(x+1)I[0\leq x+1\leq 1]dx\\ \label{lem1peq5}
&=\int_{X_1-r_n}^{0}g^3(x)dx.
\end{align}
Note that $X_1\in[0,r_n]$ and for $x\in[0,2r_n]\subset[0,1]$, $g(X_1)=f(X_1)$. Then by (\ref{lem1peq5}) and a similar argument as in (\ref{lem1peq4}), we get
\begin{align}\nonumber
\int_{0}^{X_1+r_n}f^3(x)dx+\int_{1-r_n+X_1}^{1}f^3(x)dx&=\int_{X_1-r_n}^{X_1+r_n}g^3(x)dx\\ \nonumber
&=2r_ng^3(X_1)+O(r_n^3)\\ \label{lem1peq6}
&=2r_nf^3(X_1)+O(r_n^3).
\end{align}
Similarly, for $1-r_n\leq X_1\leq 1$, we have
\begin{align}\nonumber
&\int_{X_1-r_n}^1f^3(x)dx+\int_{0}^{r_n+X_1-1}f^3(x)dx\\\nonumber
&=\int_{X_1-r_n}^1g^3(x)dx+\int_{1}^{r_n+X_1}g^3(x-1)I[0\leq x-1\leq1]dx\\\nonumber
&=\int_{X_1-r_n}^1g^3(x)dx+\int_{1}^{r_n+X_1}g^3(x)dx\\\nonumber
&=\int_{X_1-r_n}^{r_n+X_1}g^3(x)dx\\\nonumber
&=2r_ng^3(X_1)+O(r_n^3)\\\label{lem1peq7}
&=2r_nf^3(X_1)+O(r_n^3).
\end{align}
Combining (\ref{lem1peq1})-(\ref{lem1peq7}) yields the asymptotic expression of the first conditional expectation in (\ref{1prope3}). 

We now consider the second conditional expectation in (\ref{1prope3}). Applying the properties of conditional expectation alongside Lemma \ref{propmain}, one may verify that
\begin{align*}\nonumber
\mathbb{E}[A_{12}A_{13}A_{34}|X_1]&=\mathbb{E}[\mathbb{E}[A_{12}A_{13}A_{34}|X_1,X_2,X_3]|X_1]\\
&=\mathbb{E}[A_{12}A_{13}\mathbb{E}[A_{34}|X_3]|X_1]\\
&=\mathbb{E}[A_{12}A_{13}(2r_nf(X_3)+O(r_n^3))|X_1]\\
&=2r_n\mathbb{E}[A_{12}A_{13}f(X_3)|X_1]+O(r_n^5)\\
&=2r_n\mathbb{E}[A_{13}f(X_3)\mathbb{E}[A_{12}|X_1,X_3]|X_1]+O(r_n^5)\\
&=(2r_n)^2\mathbb{E}[A_{13}f(X_1)f(X_3)|X_1]+O(r_n^5)\\
&=(2r_n)^3f^3(X_1)+O(r_n^5).
\end{align*}
Here, the last equality follows from arguments analogous to those used in (\ref{lem1peq3})-(\ref{lem1peq7}).

Finally, we evaluate the third conditional expectation in (\ref{1prope3}). It is straightforward to see that
\begin{align}\nonumber
\mathbb{E}[A_{12}A_{13}A_{14}|X_1]=\left(\mathbb{E}[A_{12}|X_1]\right)^3=(2r_n)^3f^3(X_1)+O(r_n^5).
\end{align}

{\bf Proof of (\ref{2prope3}):} In view of (\ref{prope0}) and (\ref{prope3}), the second conditional expectation of (\ref{2prope3}) is readily obtained as follows:
\begin{align*}
    \mathbb{E}[A_{12}A_{13}A_{23}A_{14}|X_1]&=\mathbb{E}[A_{12}A_{13}A_{23}|X_1]\mathbb{E}[A_{14}|X_1]\\
    &=\left(3r_n^2f^2(X_1) + O(r_n^4)\right)\left(2r_nf(X_1) + O(r_n^3)\right)\\
    &=6r_n^3f^3(X_1)+O(r_n^5).
\end{align*}

Now we consider the first conditional expectation in (\ref{2prope3}). It is easy to verify that
\begin{align}\nonumber
    \mathbb{E}[A_{12}A_{13}A_{23}A_{34}|X_1]&= \mathbb{E}[\mathbb{E}[A_{12}A_{13}A_{23}A_{34}|X_1,X_2,X_3]|X_1]\\ \nonumber
    &=\mathbb{E}[A_{12}A_{13}A_{23}\mathbb{E}[A_{34}|X_3]|X_1]\\ \label{finmodeq1}
    &=2r_n \mathbb{E}[A_{12}A_{13}A_{23}f(X_3)|X_1]+O(r_n^5).
\end{align}
We evaluate the expectation in (\ref{finmodeq1}) by considering three distinct cases: \(X_{1}\in (r_{n},1-r_{n})\), \(X_{1}\in (0,r_{n})\), and \(X_{1}\in (1-r_{n},1)\).

Suppose $X_1\in(r_n,1-r_n)$. Under this condition, the event \(\{A_{12}A_{13}A_{23}=1\}\) occurs in two distinct configurations: (a) $X_2\in(X_1,X_1+r_n)$ and $X_3\in(X_1,X_1+r_n)$ or $X_3\in(X_2-r_n,X_1)$; (b) $X_2\in(X_1-r_n,X_1)$ and $X_3\in(X_1-r_n,X_1)$ or $X_3\in(X_1,X_2+r_n)$. Hence,  we have
\begin{align}\nonumber
    \mathbb{E}[A_{12}A_{13}A_{23}f(X_3)|X_1] &=\int_{X_1}^{X_1+r_n}f(x_2)dx_2\left(\int_{X_1}^{X_1+r_n}f^2(x_3)dx_3+\int_{x_2-r_n}^{X_1}f^2(x_3)dx_3\right)\\ \label{trianeq1}
    &\quad+\int_{X_1-r_n}^{X_1}f(x_2)dx_2\left(\int_{X_1-r_n}^{X_1}f^2(x_3)dx_3+\int_{X_1}^{x_2+r_n}f^2(x_3)dx_3\right).
\end{align}

By employing a fourth-order Taylor expansion of $f(x_2)$ and $f^2(x_3)$ centered at $X_1$, it is straightforward to show that
\begin{align} \label{trianeq2}
\int_{X_1}^{X_1 + r_n} f(x_2) dx_2=f(X_1) r_n + \frac{f'(X_1)}{2} r_n^2 + \frac{f''(X_1)}{6} r_n^3  +O(r_n^4),
\end{align}
\begin{align} \label{trianeq3}
\int_{X_1}^{X_1 + r_n} f^2(x_3) dx_3&= f^2(X_1) r_n +  f(X_1) f'(X_1) r_n^2  + O(r_n^3).
\end{align}

When $X_2\in(X_1,X_1+r_n)$ and  $x_3\in(X_2-r_n,X_1)$, we have $|X_1-x_3|\leq r_n$ and $|X_1-X_2+r_n|\leq 2r_n$. Then
\begin{align*}
\int_{X_2-r_n}^{X_1}f^2(x_3)dx_3=f^2(X_1)(X_1-X_2+r_n)-f(X_1)f'(X_1)(X_1-X_2+r_n)^2+O(r_n^3).
\end{align*}
This implies that
\begin{align*}\nonumber
&\int_{X_1}^{X_1+r_n}f(x_2)dx_2\int_{X_2-r_n}^{X_1}f^2(x_3)dx_3\\ \nonumber
&=f^2(X_1)(X_1+r_n)\int_{X_1}^{X_1+r_n}f(x_2)dx_2-f^2(X_1)\int_{X_1}^{X_1+r_n}x_2f(x_2)dx_2\\  
&\quad-f(X_1)f'(X_1)\int_{X_1}^{X_1+r_n}f(x_2)(X_1-x_2+r_n)^2dx_2.
\end{align*}
Note that
\begin{align*}\nonumber
&f^2(X_1)(X_1+r_n)\int_{X_1}^{X_1+r_n}f(x_2)dx_2\\  
&=f^3(X_1)(X_1r_n+r_n^2)+\frac{f^2(X_1)f'(X_1)}{2} r_n^2(X_1+r_n) + \frac{f''(X_1)f^2(X_1)(X_1+r_n)}{6} r_n^3 +O(r_n^4),  
\end{align*}
\begin{align*}\nonumber
&f^2(X_1)\int_{X_1}^{X_1+r_n}x_2f(x_2)dx_2\\ \nonumber
&=f^3(X_1)X_1r_n+\frac{f^2(X_1)f'(X_1)X_1+f^3(X_1)}{2} r_n^2\\  
&\quad +\frac{2f^2(X_1)f'(X_1)+X_1f^2(X_1)f''(X_1)}{6} r_n^3 +O(r_n^4).
\end{align*}
and
\begin{align*} 
f(X_1)f'(X_1)\int_{X_1}^{X_1+r_n}f(x_2)(X_1-x_2+r_n)^2dx_2=O(r_n^4).
\end{align*}
Then
\begin{align} \label{trianeq8}
&\int_{X_1}^{X_1+r_n}f(x_2)dx_2\int_{X_2-r_n}^{X_1}f^2(x_3)dx_3=\frac{1}{2}f^3(X_1)r_n^2-\frac{f^2(X_1)f'(X_1)}{3} r_n^3+O(r_n^4).
\end{align}
It then follows from (\ref{trianeq2}), (\ref{trianeq3}) and (\ref{trianeq8}) that
\begin{align*}
&\int_{X_1}^{X_1+r_n}f(x_2)dx_2\left(\int_{X_1}^{X_1+r_n}f^2(x_3)dx_3+\int_{X_2-r_n}^{X_1}f^2(x_3)dx_3\right)\\
  &=\frac{3}{2}f^3(X_1)r_n^2+\frac{7f^2(X_1)f'(X_1)}{6} r_n^3+O(r_n^4).
\end{align*}

Similarly, we can get
\begin{align*}
&\int_{X_1-r_n}^{X_1}f(x_2)dx_2\left(\int_{X_1-r_n}^{X_1}f^2(x_3)dx_3+\int_{X_1}^{X_2+r_n}f^2(x_3)dx_3\right)\\
  &=\frac{3}{2}f^3(X_1)r_n^2-\frac{7f^2(X_1)f'(X_1)}{6} r_n^3+O(r_n^4).
\end{align*}
In view of (\ref{trianeq1}), for \(r_{n}\le X_{1}\le 1-r_{n}\), we have
\begin{eqnarray}\label{trianeq9}
\mathbb{E}[A_{12}A_{13}A_{23}f(X_3)|X_1]= 3r_n^3f^2(X_1)+O(r_n^4),
\end{eqnarray}
which, together with (\ref{finmodeq1}), implies that
\begin{align}\label{trianeq09}
    \mathbb{E}[A_{12}A_{13}A_{23}A_{34}|X_1]&=6r_n^3f^2(X_1)+O(r_n^5).
\end{align}

Suppose $X_1\in[0,r_n)$. If $X_2\in[0,X_1]$, then $X_3\in(X_1,X_2+r_n)$ or $X_3\in(0,X_1)$ or $X_3\in(1+X_1-r_n,1)$. If $X_2\in(1+X_1-r_n,1)$, then $X_3\in[0,X_1]$ or $X_3\in[1+X_1-r_n,1]$ or $X_3\in[X_1,X_2+r_n-1]$. If $X_2\in[X_1,r_n]$, then $X_3\in[0,X_1+r_n]$ or $X_3\in[X_2+1-r_n]$. If $X_2\in[r_n,X_1+r_n]$, then $X_3\in(X_1,X_1+r_n)$ or $X_3\in(X_2-r_n,X_1)$.  Hence, we have
\begin{align*}
\mathbb{E}[A_{12}A_{13}A_{23}f(X_3)|X_1]  
&=\int_{0}^{X_1}f(x_2)dx_2\left(\int_{0}^{x_2+r_n}f^2(x_3)dx_3+\int_{1+X_1-r_n}^{1}f^2(x_3)dx_3\right)\\
&\quad+\int_{1+X_1-r_n}^{1}f(x_2)dx_2\left(\int_{0}^{x_2+r_n-1}f^2(x_3)dx_3+\int_{1+X_1-r_n}^{1}f^2(x_3)dx_3\right)\\  \nonumber
    &\quad+\int_{X_1}^{r_n}f(x_2)dx_2\left(\int_0^{X_1+r_n}f^2(x_3)dX_3+\int_{X_2+1-r_n}^1f^2(x_3)dX_3\right)\\  \nonumber
    &\quad+\int_{r_n}^{X_1+r_n}f(x_2)dx_2\int_{x_2-r_n}^{X_1+r_n}f^2(x_3)dx_3.
\end{align*}

By Assumption \ref{assumptiona}, $f(x)=g(x)I[0\leq x\leq 1]$ with \(g(x)\)  satisfying \(g(x+1)=g(x)=g(x-1)\) for all \(x\in \mathbb{R}\). By a change of variables in the definite integral, we find that
\begin{align*}
\int_{1+X_1-r_n}^{1}f^2(x_3)dx_3&=\int_{X_1-r_n}^{0}g^2(y+1)dy=\int_{X_1-r_n}^{0}g^2(y)dy\\
\int_{x_2+1-r_n}^1f^2(x_3)dx_3&=\int_{x_2-r_n}^{0}g^2(y+1)dy=\int_{x_2-r_n}^{0}g^2(y)dy,
\end{align*}
\begin{align*}
\int_{1+X_1-r_n}^{1}f(x_2) \int_{X_1-r_n}^{x_2+r_n-1}f^2(x_3)dx_3dx_2&=\int_{X_1-r_n}^{0}f(x_2+1) \int_{X_1-r_n}^{x_2+r_n}f^2(x_3)dx_3dx_2\\
&=\int_{X_1-r_n}^{0}f(x_2) \int_{X_1-r_n}^{x_2+r_n}f^2(x_3)dx_3dx_2.
\end{align*}

Then
\begin{align*}
\mathbb{E}[A_{12}A_{13}A_{23}f(X_3)|X_1]  
&=\int_{0}^{X_1}f(x_2)dx_2\int_{X_1-r_n}^{x_2+r_n}f^2(x_3)dx_3+\\
&\quad+\int_{X_1-r_n}^{0}f(x_2) \int_{X_1-r_n}^{x_2+r_n}f^2(x_3)dx_3dx_2\\  \nonumber
    &\quad+\int_{X_1}^{X_1+r_n}f(x_2)dx_2\int_{x_2-r_n}^{X_1+r_n}f^2(x_3)dx_3\\
    &=\int_{X_1-r_n}^{X_1}f(x_2)dx_2\int_{X_1-r_n}^{x_2+r_n}f^2(x_3)dx_3\\\nonumber
    &\quad+\int_{X_1}^{X_1+r_n}f(x_2)dx_2\int_{x_2-r_n}^{X_1+r_n}f^2(x_3)dx_3,
\end{align*}
which is the same as (\ref{trianeq1}). Therefore, we conclude (\ref{trianeq09}) holds for $X_1\in[0,r_n)$.

Similarly, it is easy to show that (\ref{trianeq09}) holds for $X_1\in[1-r_n,1]$.
Then the proof is complete.

\subsection{ Proof of Lemma \ref{lem2} }

Given $X_1$,  
$d_1$ follows the  Binomial distribution $B(n-1, p_n)$, where $p_n=\mathbb{E}[A_{12}|X_1]$. From Lemma \ref{propmain}, we have \(p_{n}=2r_{n}f(X_{1})+O(r_{n}^{3})\); furthermore, Assumption \ref{assumptiona} guarantees that \(f(x)\ge \lambda >0\). It follows from  the Chernoff bound that
\begin{eqnarray}\nonumber
\mathbb{P}(d_1\leq \delta \lambda nr_n|X_1)\leq \mathbb{P}(d_1\leq \delta \mu_1|X_1)
&\leq&e^{-\frac{(1-\delta)^2\mu_1}{2}}\leq e^{-(1-\delta)^2r_n\lambda (1+o(1))}.
\end{eqnarray}
Then the proof is complete.

\qed

\subsection{ Proof of Lemma \ref{lnem3}}
Note that $\mathbb{E}\left[(d_1-\mu_1)^8\right]=\mathbb{E}[\mathbb{E}[(d_1 - \mu_1)^8|X_1]]$ and
\begin{align*}
\mathbb{E}[(d_1 - \mu_1)^8|X_1]=\sum_{j_1,j_2,j_3,j_4,j_5,j_6,j_7,j_8\geq1}\mathbb{E}\big[\bar{A}_{1j_1}\bar{A}_{1j_2}\bar{A}_{1j_3}\bar{A}_{1j_4}\bar{A}_{1j_5}\bar{A}_{1j_6}\bar{A}_{1j_7}\bar{A}_{1j_8}|X_1\big].
\end{align*}
If $j_{1}\not\in\{j_2,j_3,j_4,j_5,j_6,j_7,j_8\}$, then, given $X_1$, $A_{1j_1}$ is independent of $A_{1j_t}$ $(2\leq t\leq 8)$. In this case, we have
\[\mathbb{E}\left[\prod_{t=1}^8\bar{A}_{1j_t}\Big|X_1\right]=\mathbb{E}\left[\bar{A}_{1j_1}|X_1\right]\mathbb{E}\left[\prod_{t=2}^8\bar{A}_{1j_t}\Big|X_1\right]=0.\]
Hence,  $j_1\in\{j_2,j_3,j_4,j_5,j_6,j_7,j_8\}$. Similar result holds for each $j_t$. Each index $j_t$ must equal another index $j_s$ ($s\neq t$ $s,t\in\{1,2,\dots,8\}$). Then  $|\{j_1,j_2,j_3,j_4,j_5,j_6,j_7,j_8\}|\leq 4$. Note that $A_{1j}$ and $A_{1k}$ ($k\neq s$) are conditionally independent given $X_1$ and $\mathbb{E}[\bar{A}_{1j}^m|X_1]=O(r_n)$ for all positive integer $m$ (Lemma \ref{propmain}). Then
\begin{align*}
\mathbb{E}[(d_1 - \mu_1)^8|X_1]&=C_1\sum_{j_1\neq j_2\neq j_3\neq j_4\geq1}\mathbb{E}\big[\bar{A}_{1j_1}^2\bar{A}_{1j_2}^2\bar{A}_{1j_3}^2\bar{A}_{1j_4}^2|X_1\big]+\sum_{j_1\geq1}\mathbb{E}\big[\bar{A}_{1j_1}^8|X_1\big]\\
&+C_2\sum_{j_1\neq j_2\neq j_3\geq1}\mathbb{E}\big[\bar{A}_{1j_1}^2\bar{A}_{1j_2}^2\bar{A}_{1j_3}^4|X_1\big]+C_3\sum_{j_1\neq j_2\neq j_3\geq1}\mathbb{E}\big[\bar{A}_{1j_1}^2\bar{A}_{1j_2}^3\bar{A}_{1j_3}^3|X_1\big]\\
&+\sum_{j_1\neq j_2\geq1}\Big(C_4\mathbb{E}\big[\bar{A}_{1j_1}^2\bar{A}_{1j_2}^6|X_1\big]+C_5\mathbb{E}\big[\bar{A}_{1j_1}^3\bar{A}_{1j_2}^5|X_1\big]+C_6\mathbb{E}\big[\bar{A}_{1j_1}^4\bar{A}_{1j_2}^4|X_1\big]\Big)\\
&=O(n^4r_n^4),
\end{align*}
where $C_1, C_2, C_3, C_4, C_5,C_6$ are positive constants. Then the proof is complete.
 
\qed

\subsection{ Proof of Lemma \ref{lem1}: }

 Firstly, we prove (\ref{fieq1}). The second moment of $P_1$ can be expressed as
\begin{align}\nonumber
\mathbb{E}[P_1^2]
&=
\sum_{\substack{j \ne k \ne 1\\j_1 \ne k_1 \ne 1}}
\mathbb{E}\big[T_{1jk}T_{1j_1k_1}\big]\\ \label{p1eq0}
&=\sum_{\substack{j \ne k \ne 1\\j_1 \ne k_1 \ne 1\\ \{j,k\} \cap \{j_1,k_1\} = \emptyset}}
\mathbb{E}\big[T_{1jk}T_{1j_1k_1}\big]+\sum_{\substack{j \ne k \ne 1\\j_1 \ne k_1 \ne 1\\ |\{j,k\}\cap\{j_1,k_1\}| = 1}}
\mathbb{E}\big[T_{1jk}T_{1j_1k_1}\big]+\sum_{\substack{j \ne k \ne 1\\j_1 \ne k_1 \ne 1\\ \{j,k\}=\{j_1,k_1\}}}
\mathbb{E}\big[T_{1jk}T_{1j_1k_1}\big]
\end{align}

Suppose $\{j,k\} \cap \{j_1,k_1\} = \emptyset$.There are   \((n-1)(n-2)(n-3)(n-4)\) quadruplets of indices \((j,k,j_{1},k_{1})\) such that \(j\ne k\ne 1\), \(j_{1}\ne k_{1}\ne 1\), and \(\{j,k\}\cap \{j_{1},k_{1}\}=\emptyset \).  Given $X_1$,  $T_{1jk}$ is independent of $T_{1j_1k_1}$. Since \(f(x)\) is assumed to have a bounded fourth derivative, it follows that \((f^{\prime }(X_{1}))^{2}\) and \(f(X_{1})f^{\prime \prime }(X_{1})\) are also bounded. By (\ref{prope1}) and (\ref{prope01}) of Lemma \ref{propmain}, the first sum in (\ref{p1eq0}) is equal to
\begin{align}\nonumber
&(n-1)(n-2)(n-3)(n-4)\mathbb{E}\big[\mathbb{E}\big[T_{123}T_{145}|X_1\big]\big]\\ \label{p1eq1}
&=
(n-1)(n-2)(n-3)(n-4)\mathbb{E}\big[\mathbb{E}[T_{123}| X_1]\big]\big[\mathbb{E}[T_{145}]| X_1\big]\big] 
= O(n^4r_n^8).
\end{align}

Suppose $\left|\{j,k\} \cap \{j_1,k_1\}\right| =1$. Similarly, there are at most \(n^{3}\) such index quadruplets \((j,k,j_{1},k_{1})\). If $j=j_1$, by (\ref{1prope3}) of Lemma \ref{propmain}, we have
\begin{align}\nonumber
    \mathbb{E}\big[\mathbb{E}\big[
    T_{1jk}T_{1j_1k_1}|X_1
    \big]\big]
    &=\mathbb{E}\big[\mathbb{E}\big[A_{1j}A_{jk}A_{jk_1}|X_1\big]\big]
-\mathbb{E}\big[\mathbb{E}\big[A_{1j}A_{jk}A_{1k_1}|X_1\big]\big]\\\nonumber
    &\quad-\mathbb{E}\big[\mathbb{E}\big[A_{1j}A_{1k}A_{jk_1}|X_1\big]\big]
+\mathbb{E}\big[\mathbb{E}\big[A_{1j}A_{1k}A_{1k_1}|X_1\big]\big]\\ \label{p1eq2}
    &=O(r_n^5).
\end{align}
If $j =k_1$, by (\ref{1prope3}) and (\ref{2prope3}) of Lemma \ref{propmain}, then 
\begin{align}\nonumber
     \mathbb{E}\big[\mathbb{E}\big[
    T_{1jk}T_{1j_1k_1}|X_1
    \big]\big]
    &=  \mathbb{E}\big[\mathbb{E}\big[A_{1j}A_{jk}A_{1j_1}A_{j_1j}|X_1\big]\big]
    -\mathbb{E}\big[\mathbb{E}\big[A_{1j}A_{jk}A_{1j_1}|X_1\big]\big]\\ \nonumber
    &\quad
    -\mathbb{E}\big[\mathbb{E}\big[A_{1j}A_{jj_1}A_{1j_1}A_{1k}|X_1\big]\big]
    +\mathbb{E}\big[\mathbb{E}\big[A_{1j}A_{1k}A_{1j_1}|X_1\big]\big] \\ \label{p1eq3}
    &=O(r_n^5).
\end{align}
The case $k=j_1$ can be similarly bounded as in  (\ref{p1eq3}).  Let  $k=k_1$. Note that $\mathbb{E}\big[A_{1j}A_{jk}A_{1j_1}A_{j_1k}|X_1\big]\leq \mathbb{E}\big[A_{1j}A_{jk}A_{1j_1}|X_1\big]=O(r_n^3)$. By (\ref{1prope3}) and (\ref{2prope3}) of Lemma \ref{propmain}, then 
\begin{align}\nonumber
     \mathbb{E}\big[\mathbb{E}\big[
    T_{1jk}T_{1j_1k_1}|X_1
    \big]\big]
    &=  \mathbb{E}\big[\mathbb{E}\big[A_{1j}A_{jk}A_{1j_1}A_{j_1k}|X_1\big]\big]
    -\mathbb{E}\big[\mathbb{E}\big[A_{1j}A_{jk}A_{1j_1}A_{1k}|X_1\big]\big]\\ \nonumber
    &\quad
    -\mathbb{E}\big[\mathbb{E}\big[A_{1j}A_{1k}A_{1j_1}A_{j_1k}|X_1\big]\big]
    +\mathbb{E}\big[\mathbb{E}\big[A_{1j}A_{1k}A_{1j_1}|X_1\big]\big] \\  \label{00p1eq2}
    &=O(r_n^3).
\end{align}
Then the second sum in (\ref{p1eq0}) is equal to
\begin{align}\label{psqreq1}
\sum_{\substack{j \ne k \ne 1\\j_1 \ne k_1 \ne 1\\ \{j,k\}\cap\{j_1,k_1\}| = 1}}
\mathbb{E}\big[T_{1jk}T_{1j_1k_1}\big]=O(n^3r_n^3).
\end{align}

Suppose $\{j,k\} = \{j_1,k_1\}$. There are at most $n^2$ such index quadruplets \((j,k,j_{1},k_{1})\). Note that $|T_{1jk}|\leq A_{1j}A_{jk}+A_{1j}A_{1k}$ and $|T_{1j_1k_1}|\leq2$.  By (\ref{prope1}) and (\ref{prope01}) of Lemma \ref{propmain}, one has
\begin{align}\label{p1eq4}
    \sum_{\substack{j \ne k \ne 1\\j_1 \ne k_1 \ne 1\\ \{j,k\}=\{j_1,k_1\}}}
\mathbb{E}\big[|T_{1jk}T_{1j_1k_1}|\big]
    \le
    2n^2\mathbb{E}[A_{1j}A_{jk}+A_{1j}A_{1k}] 
    =O(n^2r_n^2).
\end{align}

Combining (\ref{p1eq0}), (\ref{p1eq1}), (\ref{psqreq1}) and (\ref{p1eq4}), we get that equation (\ref{fieq1}).

Next, we prove equation (\ref{fieq2}). The fourth moment of $P_1$ can be written as
\begin{align}\label{fieq3}
\mathbb{E}\left[P_1^4\right]=\sum_{s=2}^8\sum_{\substack{j_1\neq k_1\neq 1,j_2\neq k_2\neq 1\\j_3\neq k_3\neq 1,j_4\neq k_4\neq 1\\ |\{j_1,k_1,j_2,k_2,j_3,k_3,j_4,k_4\}|=s}}\mathbb{E}[T_{1j_1k_1}T_{1j_2k_2}T_{1j_3k_3}T_{1j_4k_4}].
\end{align}

If \(|\{j_{1},k_{1},j_{2},k_{2},j_{3},k_{3},j_{4},k_{4}\}|=8\), there are at most \(n^{8}\) such 8-tuples. In this case, the variables \(T_{1j_{1}k_{1}},T_{1j_{2}k_{2}},T_{1j_{3}k_{3}},\) and \(T_{1j_{4}k_{4}}\) are conditionally independent given \(X_{1}\). By Lemma \ref{propmain}, one has
\begin{align*}\nonumber
\mathbb{E}[T_{1j_1k_1}T_{1j_2k_2}T_{1j_3k_3}T_{1j_4k_4}]&=\mathbb{E}[\mathbb{E}[T_{1j_1k_1}T_{1j_2k_2}T_{1j_3k_3}T_{1j_4k_4}|X_1]]\\
&=\mathbb{E}\big[\mathbb{E}[T_{1j_1k_1}|X_1]\mathbb{E}[T_{1j_2k_2}|X_1]\mathbb{E}[T_{1j_3k_3}|X_1]\mathbb{E}[T_{1j_4k_4}|X_1]\big]=O(r_n^{16}).
\end{align*}
The sum corresponding to $s=8$ in (\ref{fieq3}) is equal to $O(n^8r_n^{16})$.

If $|\{j_1,k_1,j_2,k_2,j_3,k_3,j_4,k_4\}|=7$,  there are at most \(n^{7}\) such 8-tuples. In this case, there exist mutually distinct indices $j_{t_1},k_{t_1},j_{t_2},k_{t_2}$ ($t_1\neq t_2$, $t_1,t_2\in\{1,2,3,4\}$) such that they are distinct from the remaining indices. Without loss of generality, let $t_1=1$ and $t_2=2$. Then $T_{1j_1k_1}$ and $T_{1j_2k_2}$ are conditionally independent given $X_1$, and $T_{1j_1k_1}T_{1j_2k_2}$ is conditionally independent of $T_{1j_3k_3}T_{1j_4k_4}$ given $X_1$. Note that $|T_{1j_3k_3}T_{1j_4k_4}|\leq4$.  By Lemma \ref{propmain}, we have
\begin{align}\nonumber
\mathbb{E}[\mathbb{E}[T_{1j_1k_1}T_{1j_2k_2}T_{1j_3k_3}T_{1j_4k_4}|X_1]]=\mathbb{E}\big[\mathbb{E}[T_{1j_1k_1}|X_1]\mathbb{E}[T_{1j_2k_2}|X_1]\mathbb{E}[T_{1j_3k_3}T_{1j_4k_4}|X_1]\big]=O(r_n^8).
\end{align}
The sum corresponding to $s=7$ in (\ref{fieq3}) is equal to $O(n^7r_n^{8})$.

Suppose $|\{j_1,k_1,j_2,k_2,j_3,k_3,j_4,k_4\}|=6$. There are at most \(n^{6}\) such 8-tuples. Note that $|T_{1jk}|\leq A_{1j}A_{jk}+A_{1j}A_{1k}$. Then
\begin{align}\nonumber
|\mathbb{E}[T_{1j_1k_1}T_{1j_2k_2}T_{1j_3k_3}T_{1j_4k_4}]|
&\leq \mathbb{E}[|T_{1j_1k_1}||T_{1j_2k_2}||T_{1j_3k_3}||T_{1j_4k_4}|]\\ \nonumber
&\leq \mathbb{E}[(A_{1j_1}A_{j_1k_1}+A_{1j_1}A_{1k_1})(A_{1j_2}A_{j_2k_2}+A_{1j_2}A_{1k_2})\\ \label{fieq4}
&\quad \times (A_{1j_3}A_{j_3k_3}+A_{1j_3}A_{1k_3})(A_{1j_4}A_{j_4k_4}+A_{1j_4}A_{1k_4})].
\end{align}
We will show that the expectation of each term in the expansion of the product in  (\ref{fieq4}) is $O(r_n^6)$.
The expectation of the first term of the expansion of the product in (\ref{fieq4}) is
\[\mathbb{E}[A_{1j_1}A_{j_1k_1}A_{1j_2}A_{j_2k_2}A_{1j_3}A_{j_3k_3}A_{1j_4}A_{j_4k_4}].\]
When \[A_{1j_1}A_{j_1k_1}A_{1j_2}A_{j_2k_2}A_{1j_3}A_{j_3k_3}A_{1j_4}A_{j_4k_4}=1,\]
the vertex set  $\{1,j_1,k_1,j_2,k_2,j_3,k_3,j_4,k_4\}$ and the edges $A_{1j_1},A_{j_1k_1}$,$A_{1j_2},A_{j_2k_2}$,$A_{1j_3},A_{j_3k_3}$,$A_{1j_4},A_{j_4k_4}$ form a connected graph, denoted as $G_1$.  If $|\{j_1,k_1,j_2,k_2,j_3,k_3,j_4,k_4\}|=6$, then $G_1$ has 7 nodes. Therefore, $G_1$ has a spanning tree that contains exactly 6 edges, denoted as $T$. Let $\mathcal{E}[T]$ denote the edge set of $T$. Then
 \[\mathbb{E}[A_{1j_1}A_{j_1k_1}A_{1j_2}A_{j_2k_2}A_{1j_3}A_{j_3k_3}A_{1j_4}A_{j_4k_4}]\leq \mathbb{E}[\prod_{e\in\mathcal{E}[T]}A_{e}]=O(r_n^6),\]
where the last term is obtained by repeatedly applying (\ref{prope0}) of Lemma \ref{propmain} six times. The expectation of the remaining terms in  the expansion of the product in (\ref{fieq4}) can be similarly bounded. Then the sum corresponding to $s=6$ in (\ref{fieq3}) is $O(n^6r_n^{6})$.

Similarly, the sum corresponding to $s$ in (\ref{fieq3}) is $O(n^sr_n^{s})$ for $s\in\{2,3,4,5\}$.  In summary, we obtain (\ref{fieq2}). The proof of Lemma \ref{lem1} is complete.

\qed

\subsection{Proof of Lemma \ref{lem3}}
Note that (\ref{fieq6}) follows from (\ref{fieq7}) by the Cauchy–Schwarz inequality. Then we only need to prove (\ref{fieq7}). It is easy to verify by the Cauchy–Schwarz inequality that
\begin{align}\label{fieq8}
    \mathbb{E}\left[\frac{P_1^2(d_1-\mu_1)^4}{d_1^2\mu_1^4} I[d_1\geq1]\right]
    &\leq \sqrt{  \mathbb{E}\left[\frac{P_1^4}{\mu_1^7} \right]  \mathbb{E}\left[  \frac{(d_1-\mu_1)^8}{d_1^4\mu_1}  I[d_1\geq1]\right]}.
\end{align}
By Lemma \ref{lem1} and the fact that $\mu_1=\Theta(nr_n)$, we have
\begin{align}\label{fieq9}
\mathbb{E}\left[\frac{P_1^4}{\mu_1^7}\right]=O\left(\frac{1}{(nr_n)^7}\right)\mathbb{E}\left[P_1^4\right]=O\left(nr_n^{9}+r_n+\frac{1}{nr_n}\right).
\end{align}
Let $\epsilon\in(0,1)$ be  a small constant.   Then
\begin{align}\label{fieq10}
\mathbb{E}\left[  \frac{(d_1-\mu_1)^8}{d_1^4\mu_1}  I[d_1\geq1]\right] 
    &=  \mathbb{E}\left[ \frac{(d_1-\mu_1)^8}{d_1^4\mu_1}  I[1\leq d_1 \le \epsilon n r_n]\right]+\mathbb{E}\left[ \frac{(d_1-\mu_1)^8}{d_1^4\mu_1} I[d_1 > \epsilon n r_n]\right].
\end{align}
Note that $|d_1-\mu_1|\leq 2\mu_1$ when $1\leq d_1 \le \epsilon n r_n$. In light of Lemma \ref{lem2},
it is easy to verify that
\begin{align}\nonumber
\mathbb{E}\left[ \frac{(d_1-\mu_1)^8}{d_1^4\mu_1}  I[1\leq d_1 \le \epsilon n r_n]\right]&\leq \mathbb{E}\left[ \frac{2^8\mu_1^8}{\mu_1}  I[1\leq d_1 \le \epsilon n r_n]\right]\\ \nonumber
&=O(n^7r_n^7)\mathbb{P}\left(1\leq d_1 \le \epsilon n r_n\right)\\ \nonumber
&=O(n^7r_n^7)e^{-cnr_n(1+o(1))}\\ \label{fieq11}
&=e^{-cnr_n(1+o(1))}.
\end{align}
In addition, it follows from Lemma \ref{lnem3} that
\begin{align}\label{fieq12}
\mathbb{E}\left[ \frac{(d_1-\mu_1)^8}{d_1^4\mu_1} I[d_1 > \epsilon n r_n]\right]=O\left(\frac{1}{n^5r_n^5}\right)\mathbb{E}\left[(d_1-\mu_1)^8\right]=O\left(\frac{1}{nr_n}\right).
\end{align}

Combining the assumption \(r_{n}=o(1)\) with (\ref{fieq8})--(\ref{fieq12}), we conclude the proof of Lemma \ref{lem3}.

\qed

\subsection{Proof of Lemma \ref{addlem1}}

Note that
\begin{align} \label{lem5eqpoof3}
\mathbb{E}[Q_1^2|X_1]=\sum_{\substack{j\ne k \ne 1\\ j_1\ne k_1\ne 1}}\mathbb{E}\big[S_{1jk}
   S_{1j_1k_1}|X_1\big].
\end{align}
Suppose $\{j,k\}\cap\{j_1,k_1\}=\emptyset$. There are at most $n^4$ such index quadruplets \((j,k,j_{1},k_{1})\). Moreover, $S_{1jk}$ and $S_{1j_1k_1}$ are conditionally independent given $X_1$. By Lemma \ref{propmain}, we have
\begin{align}\label{lem5eqpoof4}
\mathbb{E}\big[S_{1jk} S_{1j_1k_1}|X_1\big]=\mathbb{E}\big[S_{1jk}|X_1\big]\mathbb{E}\big[S_{1j_1k_1}|X_1\big]=r_n^4f^4(X_1)+O(r_n^6).
\end{align}
Suppose $|\{j,k\}\cap\{j_1,k_1\}|=1$. There are at most $n^3$ such index quadruplets \((j,k,j_{1},k_{1})\). Without loss of generality, let $j=j_1$. Note that $A_{1j}A_{1k}
    \geq A_{1j}A_{jk}A_{k1}$. Then
\begin{align}\label{lem5eqpoof5}
\big|\mathbb{E}\big[S_{1jk} S_{1j_1k_1}|X_1\big]\big|&\leq 4\mathbb{E}\big[A_{1j}A_{1k}A_{1k_1}|X_1\big]=O(r_n^3).
\end{align}
Suppose $\{j,k\}=\{j_1,k_1\}$. There are at most $n^2$ such index quadruplets \((j,k,j_{1},k_{1})\). Then
\begin{align}\label{lem5eqpoof6}
\big|\mathbb{E}\big[S_{1jk} S_{1j_1k_1}|X_1\big]\big|\leq \mathbb{E}\big[A_{1j}A_{1k}|X_1\big]=O(r_n^2).
\end{align}
Combining (\ref{lem5eqpoof3})-(\ref{lem5eqpoof6}) yields (\ref{9lem5eqpoof7}). Then the proof is complete.

\qed

\subsection{Proof of Lemma \ref{adlemnew1}}

Note that
\begin{align}\label{leqpoo0f8}
\mathbb{E}[R_1^2|X_1]&=\sum_{\substack{j \ne k \ne l \ne 1\\ j_1 \ne k_1 \ne l_1 \ne 1}}
    \mathbb{E}[T_{1jk}T_{1j_1k_1}\bar{A}_{1l}\bar{A}_{1l_1}\big|X_1].
\end{align}
If $l\not\in\{j_1,k_1,l_1\}$, then $\bar{A}_{1l}$ is conditionally independent of $T_{1jk}T_{1j_1k_1}\bar{A}_{1l_1}$ given $X_1$. In this case,
\[\mathbb{E}[T_{1jk}T_{1j_1k_1}\bar{A}_{1l}\bar{A}_{1l_1}\big|X_1]=\mathbb{E}[\bar{A}_{1l}\big|X_1]\mathbb{E}[T_{1jk}T_{1j_1k_1}\bar{A}_{1l_1}\big|X_1]=0.\]
Then $l\in\{j_1,k_1,l_1\}$. Similarly, we have $l_1\in\{j,k,l\}$. 

Suppose $l=l_1$. Note that the summation in (\ref{leqpoo0f8}) is taken over all indices $j \ne k \ne l \ne 1$ and $ j_1 \ne k_1 \ne l_1 \ne 1$. Hence $l=l_1\not\in\{j,k,j_1,k_1\}$. If $\{j,k\}\cap\{j_1,k_1\}=\emptyset$, there are at most $n^5$ such index 6-tuples \((j,k,l,j_{1},k_{1},l_1)\). Moreover, $T_{1jk}$,  $T_{1j_1k_1}$ and $\bar{A}_{1l}$  are conditionally independent given $X_1$.  It follows from Lemma \ref{propmain} that
\begin{align*} \mathbb{E}\big[T_{1jk}T_{1j_1k_1}\bar{A}_{1l}^2|X_1\big]=\mathbb{E}\big[T_{1jk}\big|X_1\big]\mathbb{E}\big[T_{1j_1k_1}\big|X_1\big]\mathbb{E}\big[\bar{A}_{1l}^2|X_1\big]=O(r_n^9).
\end{align*}
If $|\{j,k\}\cap\{j_1,k_1\}|=1$, there are at most $n^4$ such index 6-tuples \((j,k,l,j_{1},k_{1},l_1)\). By (\ref{p1eq2})-(\ref{00p1eq2}), we have  
\begin{align*} 
\mathbb{E}\Big[T_{1jk}T_{1j_1k_1}\bar{A}_{1l}^2|X_1\Big]
=\mathbb{E}\big[T_{1jk}T_{1j_1k_1}|X_1\big]\mathbb{E}\big[\bar{A}_{1l}^2|X_1\big]=O(r_n^4).
\end{align*}
If $\{j,k\}=\{j_1,k_1\}$, there are at most $n^3$ such index 6-tuples \((j,k,l,j_{1},k_{1},l_1)\). Then we have  
\begin{align*} 
|\mathbb{E}\Big[T_{1jk}T_{1j_1k_1}\bar{A}_{1l}^2|X_1\Big]|
\leq2\mathbb{E}\big[|T_{1jk}||X_1\big]\mathbb{E}\big[\bar{A}_{1l}^2|X_1\big]=O(r_n^3).
\end{align*}
Therefore, for the sum taken over $l=l_1$ in (\ref{leqpoo0f8}), we have
\begin{align}\label{leqpoof08}
\sum_{\substack{j \ne k \ne l \ne 1\\ j_1 \ne k_1 \ne l \ne 1}}
    \mathbb{E}[T_{1jk}T_{1j_1k_1}\bar{A}_{1l}^2\big|X_1]=O(n^5r_n^9+n^4r_n^4+n^3r_n^3).
\end{align}

Suppose $l\neq l_1$ and $|\{l,j,k\}\cap\{l_1,j_1,k_1\}|=2$. There are at most $n^4$ such index 6-tuples \((j,k,l,j_{1},k_{1},l_1)\). Recall that $l\in\{j_1,k_1,l_1\}$ and $l_1\in\{j,k,l\}$. Otherwise, the expectation in (\ref{leqpoo0f8}) vanishes. Then there are 4 cases: (a) $l=j_1$ and $l_1=j$; (b) $l=j_1$ and $l_1=k$; (c) $l=k_1$ and $l_1=j$; (d) $l=k_1$ and $l_1=k$. We evaluate the expectation in (77) by considering the four cases. Note that $\{j,k\}\cap\{j_1,k_1\}=\emptyset$ and $A_{1l}\bar{A}_{1l}=A_{1l}(1-\mathbb{E}[A_{1l}|X_1])$. If $l=j_1$ and $l_1=j$, then 
\begin{align*}\nonumber
\mathbb{E}\big[T_{1jk}T_{1j_1k_1}\bar{A}_{1j_1}\bar{A}_{1j}|X_1\big]
&=\mathbb{E}\big[A_{1j}(A_{jk}-A_{1k})A_{1j_1}(A_{j_1k_1}-A_{1k_1})\bar{A}_{1j_1}\bar{A}_{1j}|X_1\big]\\ \nonumber
&=(1-\mathbb{E}[A_{1l}|X_1])^2\mathbb{E}\big[T_{1jk}T_{1j_1k_1}|X_1\big]\\  
&=O(r_n^8).
\end{align*}
If $l=j_1$ and $l_1=k$, then
\begin{align*}\nonumber
\mathbb{E}\big[T_{1jk}T_{1j_1k_1}\bar{A}_{1j_1}\bar{A}_{1k}|X_1\big]
&=\mathbb{E}\big[A_{1j}(A_{jk}-A_{1k})A_{1j_1}(A_{j_1k_1}-A_{1k_1})\bar{A}_{1j_1}\bar{A}_{1k}|X_1\big]\\ \nonumber
&=(1-\mathbb{E}[A_{1j_1}|X_1])\mathbb{E}\big[T_{1jk}\bar{A}_{1k}|X_1\big]\mathbb{E}\big[T_{1j_1k_1}|X_1\big]\\  
&=O(r_n^6).
\end{align*}
If $l=k_1$ and $l_1=j$, then
\begin{align*}\nonumber
\mathbb{E}\big[T_{1jk}T_{1j_1k_1}\bar{A}_{1j}\bar{A}_{1k_1}|X_1\big]
&=\mathbb{E}\big[A_{1j}(A_{jk}-A_{1k})A_{1j_1}(A_{j_1k_1}-A_{1k_1})\bar{A}_{1j}\bar{A}_{1k_1}|X_1\big]\\ \nonumber
&=(1-\mathbb{E}[A_{1j}|X_1])\mathbb{E}\big[T_{1jk}|X_1\big]\mathbb{E}\big[T_{1j_1k_1}\bar{A}_{1k_1}|X_1\big]\\ 
&=O(r_n^6).
\end{align*}
If $l=k_1$ and $l_1=k$, then
\begin{align*}\nonumber
\mathbb{E}\big[T_{1jk}T_{1j_1k_1}\bar{A}_{1k}\bar{A}_{1k_1}|X_1\big]
=\mathbb{E}\big[T_{1jk}\bar{A}_{1k}|X_1\big]\mathbb{E}\big[T_{1j_1k_1}\bar{A}_{1k_1}|X_1\big]=O(r_n^4).
\end{align*}
Therefore, for the sum taken over $l\neq l_1$ and $|\{l,j,k\}\cap\{l_1,j_1,k_1\}|=2$ in (\ref{leqpoo0f8}), we have
\begin{align}\label{leqpoof008}
\sum_{\substack{j \ne k \ne l \ne 1\\ j_1 \ne k_1 \ne l_1 \ne 1\\
|\{l,j,k\}\cap\{l_1,j_1,k_1\}|=2\\ l\neq l_1}}
    \mathbb{E}[T_{1jk}T_{1j_1k_1}\bar{A}_{1l}\bar{A}_{1l_1}\big|X_1]=O(n^4r_n^4).
\end{align}

Suppose $\{j,k,l\}=\{j_1,k_1,l_1\}$. There are at most $n^3$ such index 6-tuples \((j,k,l,j_{1},k_{1},l_1)\).  Note that $|\bar{A}_{1l_1}|\leq 2$ and $|T_{1j_1k_1}|\leq2$. The sum taken over  $\{l,j,k\}=\{l_1,j_1,k_1\}$ in (\ref{leqpoo0f8}) is bounded by 
\begin{align}\nonumber
\sum_{\substack{j \ne k \ne l \ne 1\\ j_1 \ne k_1 \ne l_1 \ne 1\\
\{l,j,k\}=\{l_1,j_1,k_1\}
}}\big|\mathbb{E}\big[\big(T_{1jk}T_{1j_1k_1}\bar{A}_{1l}\bar{A}_{1l_1}\big|X_1\big]\big|
&\leq
4\mathbb{E}\big[\big(A_{1j}A_{jk}+A_{1j}A_{1k}\big)\big(A_{1l}+\mathbb{E}[A_{1l}|X_1]\big)|X_1\big]\\ \label{lem5eqpoof11}
&=O(n^3r_n^3).
\end{align}

Combining (\ref{leqpoo0f8})-(\ref{lem5eqpoof11}) yields (\ref{lemnewww2}). Then the proof is complete.

\qed

\subsection{Proof of Lemma \ref{lemnew1} }

 We provide detailed proof of (\ref{lemneweq1}). The proofs of (\ref{lemneweq01}) and (\ref{lem2neweq2}) are analogous. We omit them for simplicity.  Note that
\begin{align}\label{frifinropeq1}
     \mathbb{E}\left[\frac{P_1R_2}{\mu_1^2\mu_2^2}\right]=\mathbb{E}\left[\frac{\mathbb{E}[P_1R_2|X_1,X_2]}{\mu_1^2\mu_2^2}\right]=\sum_{\substack{j\ne k \ne 1\\ j_1\ne k_1\neq l \ne 2}} \mathbb{E}\left[\frac{1}{\mu_1^2\mu_2^2}\mathbb{E}[T_{1jk}T_{2j_1k_1}\bar{A}_{2l}|X_1,X_2]\right].
\end{align}
We will derive an upper bound for the conditional expectation \(\mathbb{E}[T_{1jk}T_{2j_{1}k_{1}}{A}_{2l}|X_{1},X_{2}]\) by considering the cases where \(l\notin \{1,j,k\}\) and \(l\in \{1,j,k\}\). 

If \(l\notin \{1,j,k\}\), then \(T_{1jk}T_{2j_{1}k_{1}}\) and \(\bar{A}_{2l}\) are conditionally independent given \(X_{1}\) and \(X_{2}\); consequently, we have
\begin{align*}
\mathbb{E}[T_{1jk}T_{2j_1k_1}\bar{A}_{2l}|X_1,X_2]=\mathbb{E}[T_{1jk}T_{2j_1k_1}|X_1,X_2]\mathbb{E}[\bar{A}_{2l}|X_1,X_2]=0.
\end{align*}
Then the sum over indices satisfying \(j\ne k\ne 1\) and \(j_{1}\ne k_{1}\ne l\ne 2\) with  \(l\notin \{1,j,k\}\) in (\ref{frifinropeq1}) vanishes.

If $l=1$, then $j_1,k_1\geq 3$. Suppose $\{j,k\}\cap\{j_1,k_1\}=\emptyset$ and $j,k\geq3$. There are at most $n^4$ such index tuples $\{j,k,j_1,k_1,l\}$.  In addition,  \(T_{1jk}\), \(T_{2j_{1}k_{1}}\) and \(\bar{A}_{21}\) are conditionally independent given \(X_{1}\) and \(X_{2}\). It follows from Lemma \ref{propmain} that
\begin{align*}
\mathbb{E}[T_{1jk}T_{2j_1k_1}\bar{A}_{21}|X_1,X_2]=\bar{A}_{21}\mathbb{E}[T_{1jk}|X_1,X_2]\mathbb{E}[T_{2j_1k_1}|X_1,X_2]=O(r_n^8).
\end{align*}
Consequently, the sum in (\ref{frifinropeq1}) over the indices satisfying \(j\ne k\ne 1\) and \(j_{1}\ne k_{1}\ne l\ne 2\), restricted to the case \(l=1\) and \(\{j,k\}\cap \{j_{1},k_{1}\}=\emptyset \) and $j,k\geq3$, is of order \(O(r_{n}^{4})\).

If $l=1$, and $|\{j,k\}\cap\{j_1,k_1\}|=3$, there are at most $n^3$ such index tuples $\{j,k,j_1,k_1,l\}$. Note that $|\bar{A}_{21}|\leq2$, $\mu_1=\Theta(nr_n)$ and $\mu_2=\Theta(nr_n)$. If $j=j_1$, the expectation in    (\ref{frifinropeq1}) can be bounded as follows:
\begin{align*}
&\mathbb{E}\left[\frac{1}{\mu_1^2\mu_2^2}|\mathbb{E}[T_{1jk}T_{2jk_1}\bar{A}_{21}|X_1,X_2]|\right]\\
&\leq 2\mathbb{E}\left[\frac{1}{\mu_1^2\mu_2^2}\mathbb{E}[A_{1j}A_{jk}A_{2j}A_{jk_1}|X_1,X_2]\right]+2\mathbb{E}\left[\frac{1}{\mu_1^2\mu_2^2}\mathbb{E}[A_{1j}A_{jk}A_{2j}A_{2k_1}|X_1,X_2]\right]\\
&\quad +2\mathbb{E}\left[\frac{1}{\mu_1^2\mu_2^2}\mathbb{E}[A_{1j}A_{kk}A_{2j}A_{jk_1}|X_1,X_2]\right]+2\mathbb{E}\left[\frac{1}{\mu_1^2\mu_2^2}\mathbb{E}[A_{1j}A_{1k}A_{2j}A_{2k_1}|X_1,X_2]\right]\\
&=O\left(\frac{1}{n^4r_n^4}\right)\Big(\mathbb{E}\left[A_{1j}A_{jk}A_{2j}A_{jk_1}\right]+\mathbb{E}[A_{1j}A_{jk}A_{2j}A_{2k_1}]\\
&\quad +\mathbb{E}[A_{1j}A_{kk}A_{2j}A_{jk_1}]+\mathbb{E}[A_{1j}A_{1k}A_{2j}A_{2k_1}]\Big)\\
&=O\left(\frac{r_n^4}{n^4r_n^4}\right).
\end{align*}
Here, the last equality follows from repeatedly applying the first equation in Lemma \ref{propmain} four times. The cases $j=k_1$ or $k=j_1$ or $k=k_1$ can be similarly bounded.
Consequently, the sum in (\ref{frifinropeq1}) over the indices satisfying \(j\ne k\ne 1\) and \(j_{1}\ne k_{1}\ne l\ne 2\), restricted to the case \(l=1\) and $|\{j,k\}\cap\{j_1,k_1\}|=3$, is of order \(O\left(\frac{1}{n}\right)\).

If $l=1$, and $\{j,k\}=\{j_1,k_1\}$, there are at most $n^2$ such index tuples $\{j,k,j_1,k_1,l\}$. Note that $|\bar{A}_{21}|\leq2$, $|T_{2j_1k_1}|\leq2$ and $|T_{1jk}|\leq A_{1j}(A_{jk}+A_{1k})$. In this case, we have
\begin{align*}
|\mathbb{E}[T_{1jk}T_{2j_1k_1}\bar{A}_{21}|X_1,X_2]|\leq4\mathbb{E}[|T_{1jk}||X_1,X_2]=O(r_n^2).
\end{align*}
Consequently, the sum in (\ref{frifinropeq1}) over the indices satisfying \(j\ne k\ne 1\) and \(j_{1}\ne k_{1}\ne l\ne 2\), restricted to the case \(l=1\) and $\{j,k\}=\{j_1,k_1\}$, is of order \(O\left(\frac{1}{n^{2}r_{n}^{2}}\right)\).

If $l=1$, $j=2$ and $k\neq j_1\neq k_1$,  there are at most $n^3$ such index tuples $\{j,k,j_1,k_1,l\}$. In this case, we have
\begin{align*}
\mathbb{E}[T_{12k}T_{2j_1k_1}\bar{A}_{21}|X_1,X_2]= \bar{A}_{21}\mathbb{E}[T_{12k}|X_1,X_2]\mathbb{E}[T_{2j_1k_1}|X_1,X_2]=O(r_n^4).
\end{align*}
Consequently, the sum in (\ref{frifinropeq1}) over these indices is of order  \(O\left(\frac{1}{n}\right)\).

If $l=1$, $j=2$, $k=j_1$ or  $k=k_1$, there are at most $n^2$ such index tuples $\{j,k,j_1,k_1,l\}$. In this case, we have
\begin{align*}
\mathbb{E}[|T_{12k}T_{2j_1k_1}\bar{A}_{21}|]\leq 4\mathbb{E}[A_{12}(A_{2k}+A_{1k})]=O(r_n^2).
\end{align*}
Thus, the summation in (\ref{frifinropeq1}) restricted to these indices is bounded by \(O\left(\frac{1}{n^{2}r_{n}^{2}}\right)\).

If $l=1$, $k=2$, and $j\neq j_1\neq k_1$, there are at most $n^3$ such index tuples $\{j,k,j_1,k_1,l\}$. In this case, we have
\begin{align*}
\mathbb{E}[T_{1j2}T_{2j_1k_1}\bar{A}_{21}|X_1,X_2]=\bar{A}_{21}\mathbb{E}[T_{1j2}|X_1,X_2]\mathbb{E}[T_{2j_1k_1}|X_1,X_2]=O(r_n^4).
\end{align*}
Consequently, the sum in (\ref{frifinropeq1}) over these indices is of order  
\(O\left(\frac{1}{n}\right)\).

If $l=1$, $k=2$, $j=j_1$ or  $k=k_1$, there are at most $n^2$ such index tuples $\{j,k,j_1,k_1,l\}$. In this case, we have
\begin{align*}
\mathbb{E}[|T_{1j2}T_{2j_1k_1}\bar{A}_{21}|]\leq 4\mathbb{E}[(A_{1j_1}A_{j_1k_1}+A_{1j_1}A_{1k_1})|X_1,X_2]=O(r_n^2).
\end{align*}
Thus, the summation in (\ref{frifinropeq1}) restricted to these indices is bounded by \(O\left(\frac{1}{n^{2}r_{n}^{2}}\right)\).

If $l=j$, then $l=j\geq3$, and $l=j\not\in\{k,j_1,k_1\}$. Suppose $k\not\in\{j_1,k_1\}$, $k, j_1,k_1\geq3$. There are at most $n^3$ such index tuples $\{j,k,j_1,k_1,l\}$. In this case, we have
\begin{align*}
\mathbb{E}[T_{1jk}T_{2j_1k_1}\bar{A}_{2j}|X_1,X_2]=\mathbb{E}[T_{1jk}\bar{A}_{2j}|X_1,X_2]\mathbb{E}[T_{2j_1k_1}|X_1,X_2]=O(r_n^7).
\end{align*}
Consequently, the sum in (\ref{frifinropeq1}) over these indices is \(O\left(\frac{1}{n}\right)\).

 Suppose $l=j$. If $k=j_1$ or $k=k_1$, and $k, j_1,k_1\geq3$, there are at most $n^2$ such index tuples $\{j,k,j_1,k_1,l\}$. In this case, we have
\begin{align*}
|\mathbb{E}[T_{1jk}T_{2j_1k_1}\bar{A}_{2j}|X_1,X_2]|\leq 4\mathbb{E}[|T_{1jk}| |X_1,X_2]=O(r_n^2).
\end{align*}
Consequently, the sum in (\ref{frifinropeq1}) over these indices is  \(O\left(\frac{1}{n^{2}r_{n}^{2}}\right)\). Similar results hold for $k=2$ or $j_1=1$ or $k_1=1$.

If $l=k$, then $k=l\geq3$, and $l=k\not\in\{j,j_1,k_1\}$. Suppose $j\neq j_1\neq k_1\geq3$. There are at most $n^4$ such index tuples $\{j,k,j_1,k_1,l\}$. In this case, we have
\begin{align*}
\mathbb{E}[T_{1jk}T_{2j_1k_1}\bar{A}_{2k}|X_1,X_2]=\mathbb{E}[T_{1jk}\bar{A}_{2k}|X_1,X_2]\mathbb{E}[T_{2j_1k_1}|X_1,X_2]=O(r_n^7).
\end{align*}
Consequently, the sum in (\ref{frifinropeq1}) over these indices is of order \(O(r_{n}^{3})\).

Suppose $l=k$.  If $j=j_1\neq k_1\geq3$, there are at most $n^3$ such index tuples $\{j,k,j_1,k_1,l\}$. In this case, we have
\begin{align*}
&\mathbb{E}\left[\frac{1}{\mu_1^2\mu_2^2}\mathbb{E}[|T_{1jk}T_{2jk_1}\bar{A}_{2k}||X_1,X_2]\right]\\
&\leq O\left(\frac{1}{n^4r_n^4}\right)\mathbb{E}\left[\mathbb{E}[ (A_{1j}A_{jk}A_{jk_1}+A_{1j}A_{1k}A_{2k_1})(A_{2k}+\mathbb{E}[A_{2k}|X_1])|X_1,X_2]\right]\\
&=O\left(\frac{1}{n^4}\right).
\end{align*}
Consequently, the sum in (\ref{frifinropeq1}) over these indices is of order \(O\left(\frac{1}{n}\right)\).

Suppose $l=k$.  If $j=k_1\neq j_1\geq3$, there are at most $n^3$ such index tuples $\{j,k,j_1,k_1,l\}$. In this case, we have
\begin{align*}
&\mathbb{E}\left[\frac{1}{\mu_1^2\mu_2^2}\mathbb{E}[|T_{1jk}T_{2j_1j}\bar{A}_{2k}||X_1,X_2]\right]\\
&\leq O\left(\frac{1}{n^4r_n^4}\right)\mathbb{E}\left[\mathbb{E}[ (A_{1j}A_{jk}A_{2j_1}+A_{1j}A_{1k}A_{2j_1})(A_{2k}+\mathbb{E}[A_{2k}|X_1])|X_1,X_2]\right]\\
&=O\left(\frac{1}{n^4}\right).
\end{align*}
Consequently, the sum in (\ref{frifinropeq1}) over these indices is of order \(O\left(\frac{1}{n}\right)\).

If $l=k$, $j=2$ and $j_1\neq k_1\geq3$, there are at most $n^3$ such index tuples $\{j,k,j_1,k_1,l\}$. In this case, we have
\begin{align*}
&\mathbb{E}\left[\frac{1}{\mu_1^2\mu_2^2}\mathbb{E}[|T_{12k}T_{2j_1k_1}\bar{A}_{2k}||X_1,X_2]\right]\\
&\leq O\left(\frac{1}{n^4r_n^4}\right)\mathbb{E}\left[ (A_{12}A_{2k}+A_{12}A_{1k})(A_{2j_1}A_{j_1k_1}+A_{2j_1}A_{2k_1})\right] \\
&=O\left(\frac{1}{n^4}\right).
\end{align*}
Consequently, the sum in (\ref{frifinropeq1}) over these indices is of order \(O\left(\frac{1}{n}\right)\).

If $l=k$  and $j=2$ and $j_1=1$ or $k_1=1$, there are at most $n^2$ such index tuples $\{j,k,j_1,k_1,l\}$. In this case, we have
\begin{align*}
\mathbb{E}[|T_{12k}T_{2j_1k_1}\bar{A}_{2k}||X_1,X_2]\leq 4\mathbb{E}[ (A_{2k}A_{2j_1}A_{j_1k_1}+A_{2k}A_{2j_1}A_{2k_1})|X_1,X_2] =O(r_n^2).
\end{align*}
Consequently, the sum in (\ref{frifinropeq1}) over these indices is of order \(O\left(\frac{1}{n^{2}r_{n}^{2}}\right)\).

In summary, we obtain (\ref{lemneweq1}).  The proofs of (\ref{lemneweq01}) and (\ref{lem2neweq2}) are analogous. We omit them for simplicity.

\qed

\subsection{Proof of Lemma \ref{lemnew3} }

   Note that
\begin{align}\label{lem5eqpoof8}
\mathbb{E}\left[\frac{R_1R_2}{\mu_1^2\mu_2^2}\right]=\mathbb{E}\left[\frac{\mathbb{E}[R_1R_2|X_1,X_2]}{\mu_1^2\mu_2^2}\right]&=\sum_{\substack{j \ne k \ne l \ne 1\\ j_1 \ne k_1 \ne l_1 \ne 2}}
    \mathbb{E}\left[\frac{\mathbb{E}[T_{1jk}T_{2j_1k_1}\bar{A}_{1l}\bar{A}_{2l_1}\big|X_1,X_2]}{\mu_1^2\mu_2^2}\right].
\end{align}
If $l\not\in\{2,j_1,k_1,l_1\}$, then $\bar{A}_{1l}$ is conditionally independent of $T_{1jk}T_{2j_1k_1}\bar{A}_{2l_1}$ given $X_1,X_2$. Then
\[\mathbb{E}\Big[T_{1jk}T_{2j_1k_1}\bar{A}_{1l}\bar{A}_{2l_1}\big|X_1,X_2\Big]=\mathbb{E}\big[T_{1jk}T_{2j_1k_1}\bar{A}_{2l_1}\big|X_1,X_2\big]\mathbb{E}\big[\bar{A}_{1l}\big|X_1,X_2\big]=0.\]
Hence, $l\in\{2, j_1,k_1,l_1\}$. Similarly, we have $l_1\in\{1,j,k,l\}$. Next, we derive upper bounds for the sums in (\ref{lem5eqpoof8}) corresponding to the cases \(l=l_{1}\), \(l=2\), \(l=j_{1}\), and \(l=k_{1}\), respectively.

(c1). If $l=l_1$, then $l=l_1\geq 3$ and $l=l_1\not\in\{j,k,j_1,k_1\}$. For $\{j,k\}\cap\{j_1,k_1\}=\emptyset$  and $j,k,j_1,k_1\geq3$, it follows from Lemma \ref{propmain} that
\begin{align}\nonumber
\mathbb{E}\big[T_{1jk}T_{2j_1k_1}\bar{A}_{1l}^2\big|X_1,X_2\big]=\mathbb{E}\big[T_{1jk}\big|X_1,X_2\big]\mathbb{E}\big[T_{2j_1k_1}\big|X_1,X_2\big]\mathbb{E}\big[\bar{A}_{1l}^2\big|X_1,X_2\big]
=O(r_n^9).
\end{align}
Moreover, there are at most $n^5$ such index tuples $(j,k,l,j_1,k_1,l_1)$. 
Consequently, the sum in (\ref{lem5eqpoof8}) over these indices is of order \(O(nr_{n}^{5})\).

(c2). If $l=l_1$ and $j=j_1\neq k\neq k_1\geq3$, it follows from Lemma \ref{propmain} that
\begin{align}\nonumber
&\mathbb{E}\big[\left|\mathbb{E}\big[T_{1jk}T_{2jk_1}\bar{A}_{1l}^2\big|X_1,X_2\big]\right|\big]\\ \nonumber
&=\mathbb{E}\big[\left|\mathbb{E}\big[(A_{1j}A_{jk}-A_{1j}A_{1k})(A_{2j}A_{jk_1}-A_{2j}A_{2k_1})\big|X_1,X_2\big]\mathbb{E}\big[\bar{A}_{1l}^2\big|X_1,X_2\big]\right|\big]\\ \nonumber
&=O(r_n)\mathbb{E}\big[\mathbb{E}\big[\left|(A_{1j}A_{jk}-A_{1j}A_{1k})(A_{2j}A_{jk_1}-A_{2j}A_{2k_1})\right|\big|X_1,X_2\big]\big]\\ \nonumber
&\leq O(r_n)\mathbb{E}\big[(A_{1j}A_{jk}+A_{1j}A_{1k})(A_{2j}A_{jk_1}+A_{2j}A_{2k_1})\big]\\ \nonumber
&=O(r_n^5).
\end{align}
In this case, there are at most $n^4$ such index tuples $(j,k,l,j_1,k_1,l_1)$. 
Consequently, the sum in (\ref{lem5eqpoof8}) over these indices is   \(O( r_{n} )\). The same bound applies to the case $l=l_1$  and $j=k_1\neq j_1\neq k_1\geq3$.

(c3). If $l=l_1$ and $k=j_1\neq j\neq k_1\geq3$, it follows from Lemma \ref{propmain} that
\begin{align*}\nonumber
&\mathbb{E}\big[\left|\mathbb{E}\big[T_{1jk}T_{2kk_1}\bar{A}_{1l}^2\big|X_1,X_2\big]\right|\big]\\
&=\mathbb{E}\big[\left|\mathbb{E}\big[(A_{1j}A_{jk}-A_{1j}A_{1k})(A_{2k}A_{kk_1}-A_{2k}A_{2k_1})\big|X_1,X_2\big]\mathbb{E}\big[\bar{A}_{1l}^2\big|X_1,X_2\big]\right|\big]\\
&=O(r_n)\mathbb{E}\big[\mathbb{E}\big[\left|(A_{1j}A_{jk}-A_{1j}A_{1k})(A_{2k}A_{kk_1}-A_{2k}A_{2k_1})\right|\big|X_1,X_2\big]\big]\\
&\leq O(r_n)\mathbb{E}\big[(A_{1j}A_{jk}+A_{1j}A_{1k})(A_{2k}A_{kk_1}+A_{2k}A_{2k_1})\big]\\
&=O(r_n^5).
\end{align*}
In this case, there are at most $n^4$ such index tuples $(j,k,l,j_1,k_1,l_1)$. 
Consequently, the sum in (\ref{lem5eqpoof8}) over these indices is   \(O(r_{n})\). The same bound applies to the case $l=l_1$  and $k=k_1\neq j\neq j_1\geq3$.

(c4). If $l=l_1$ and $\{j,k\}=\{j_1,k_1\}$, it follows from Lemma \ref{propmain} that
\begin{align*}\nonumber
\mathbb{E}\big[\left|\mathbb{E}\big[T_{1jk}T_{2kk_1}\bar{A}_{1l}^2\big|X_1,X_2\big]\right|\big]\leq\mathbb{E}\big[|T_{1jk}|]|X_1,X_2\big] \mathbb{E}[\bar{A}_{1l}^2\big|X_1,X_2\big]=O(r_n^3).
\end{align*}
In this case, there are at most $n^3$ such index tuples $(j,k,l,j_1,k_1,l_1)$. 
Consequently, the sum in (\ref{lem5eqpoof8}) over these indices is   \(O\left(\frac{1}{nr_{n}}\right)\).

(c5). If $l=l_1$, $j=2$ and $k\neq j_1\neq k_1\geq3$,
\begin{align}\nonumber
\mathbb{E}\big[T_{12k}T_{2j_1k_1}\bar{A}_{1l}^2\big|X_1,X_2\big]
&=\mathbb{E}\big[T_{12k}\big|X_1,X_2\big]\mathbb{E}\big[T_{2j_1k_1}\big|X_1,X_2\big]\mathbb{E}\big[\bar{A}_{1l}^2\big|X_1,X_2\big]\\ \nonumber
&=\mathbb{E}\big[(A_{2k}+A_{1k})\big|X_1,X_2\big]\mathbb{E}\big[T_{2j_1k_1}\big|X_1,X_2\big]\mathbb{E}\big[\bar{A}_{1l}^2\big|X_1,X_2\big] \\
&=O(r_n^5).
\end{align}
In this case, there are at most $n^4$ such index tuples $(j,k,l,j_1,k_1,l_1)$. 
Thus, the sum in (\ref{lem5eqpoof8}) over these indices is   \(O(r_{n})\).

(c6). If $l=l_1$, $j=2$ and $k=j_1\neq k_1\geq3$,
\begin{align}\nonumber
\big|\mathbb{E}\big[T_{12k}T_{2kk_1}\bar{A}_{1l}^2\big|X_1,X_2\big]\big|
&\leq 2\mathbb{E}\big[|T_{2kk_1}|\big|X_1,X_2\big]\mathbb{E}\big[\bar{A}_{1l}^2\big|X_1,X_2\big]\\ \nonumber
&=O(r_n^3).
\end{align}
In this case, there are at most $n^3$ such index tuples $(j,k,l,j_1,k_1,l_1)$. 
Thus, the sum in (\ref{lem5eqpoof8}) over these indices is   \(O(\frac{1}{nr_{n}})\). The same bound applies to the case $l=l_1$, $j=2$ and $k=k_1\neq j_1\geq3$.

(c7). If $l=l_1$, $j=2$ and $j_1=1$, then $k,k_1\geq3$. For $k\neq k_1$, we have
\begin{align}\nonumber
\big|\mathbb{E}\big[T_{12k}T_{21k_1}\bar{A}_{1l}^2\big|X_1,X_2\big]\big|
\leq 2\mathbb{E}\big[(A_{2k}+A_{1k})\big|X_1,X_2\big]\mathbb{E}\big[\bar{A}_{1l}^2\big|X_1,X_2\big]=O(r_n^2).
\end{align}
In this case, there are at most $n^2$ such index tuples $(j,k,l,j_1,k_1,l_1)$. 
Thus, the sum in (\ref{lem5eqpoof8}) over these indices is   \(O(\frac{1}{n^2r_{n}^2})\). The same bound applies to the case $l=l_1$, $j=2$ and $k_1=1$, $k\neq j_1$.

(c8). For $l=l_1$, $j=2$ and $j_1=1$, and $k=k_1$, we have
\begin{align}\nonumber
\big|\mathbb{E}\big[T_{12k}T_{21k}\bar{A}_{1l}^2\big|X_1,X_2\big]\big|
\leq 4\mathbb{E}[A_{2k}+A_{1k}]\mathbb{E}\big[\bar{A}_{1l}^2\big|X_1,X_2\big]=O(r_n^2).
\end{align}
In this case, there are at most $n^2$ such index tuples $(j,k,l,j_1,k_1,l_1)$. 
Thus, the sum in (\ref{lem5eqpoof8}) over these indices is   \(O(\frac{1}{n^2r_{n}^2})\). The same bound applies to the case $l=l_1$, $j=2$ and $k_1=1$, $j_1=k$.

(c9). If $l=l_1$, $k=2$ and $j\neq j_1\neq k_1\geq3$, then
\begin{align}\nonumber
\mathbb{E}\big[T_{1j2}T_{2j_1k_1}\bar{A}_{1l}^2\big|X_1,X_2\big]
&=\mathbb{E}\big[T_{1j2}\big|X_1,X_2\big]\mathbb{E}\big[T_{2j_1k_1}\big|X_1,X_2\big]\mathbb{E}\big[\bar{A}_{1l}^2\big|X_1,X_2\big]\\ \nonumber
&= \mathbb{E}\big[(A_{j2}+A_{1j})\big|X_1,X_2\big]\mathbb{E}\big[T_{2j_1k_1}\big|X_1,X_2\big]\mathbb{E}\big[\bar{A}_{1l}^2\big|X_1,X_2\big] \\
&=O(r_n^5).
\end{align}
In this case, there are at most $n^4$ such index tuples $(j,k,l,j_1,k_1,l_1)$. 
Thus, the sum in (\ref{lem5eqpoof8}) over these indices is \(O(r_{n})\).

(c10). If $l=l_1$, $k=2$ and $j=j_1\neq k_1\geq3$, then
\begin{align}\nonumber
|\mathbb{E}\big[T_{1j2}T_{2jk_1}\bar{A}_{1l}^2\big|X_1,X_2\big]|
&\leq\mathbb{E}\big[|T_{2j_1k_1}|\big|X_1,X_2\big]\mathbb{E}\big[\bar{A}_{1l}^2\big|X_1,X_2\big]=O(r_n^3).
\end{align}
In this case, there are at most $n^3$ such index tuples $(j,k,l,j_1,k_1,l_1)$. 
Thus, the sum in (\ref{lem5eqpoof8}) over these indices is \(O\left(\frac{1}{nr_{n}}\right)\). The same bound applies to the case $l=l_1$, $k=2$ and $j=k_1\neq j_1\geq3$.

(c11). If $l=l_1$, $k=2$, $j_1=1$, $j\neq k_1\geq3$, then
\begin{align}\nonumber
|\mathbb{E}\big[T_{1j2}T_{21k_1}\bar{A}_{1l}^2\big|X_1,X_2\big]|
&\leq4\mathbb{E}\big[A_{1j}(A_{1k_1}+A_{2k_1})\big|X_1,X_2\big]\mathbb{E}\big[\bar{A}_{1l}^2\big|X_1,X_2\big]=O(r_n^3).
\end{align}
In this case, there are at most $n^3$ such index tuples $(j,k,l,j_1,k_1,l_1)$. 
Thus, the sum in (\ref{lem5eqpoof8}) over these indices is \(O\left(\frac{1}{nr_{n}}\right)\). The same bound applies to the case $l=l_1$, $k=2$, $k_1=1$, $j_1\neq k\geq3$.

(c12). If $l=l_1$, $k=2$, $j_1=1$, $j=k_1\geq3$, then
\begin{align}\nonumber
|\mathbb{E}\big[T_{1j2}T_{21j}\bar{A}_{1l}^2\big|X_1,X_2\big]|
&\leq4\mathbb{E}\big[A_{1j}\big|X_1,X_2\big]\mathbb{E}\big[\bar{A}_{1l}^2\big|X_1,X_2\big]=O(r_n^2).
\end{align}
In this case, there are at most $n^2$ such index tuples $(j,k,l,j_1,k_1,l_1)$. 
Thus, the sum in (\ref{lem5eqpoof8}) over these indices is \(O\left(\frac{1}{n^2r_{n}^2}\right)\). The same bound applies to the case $l=l_1$, $k=2$, $k_1=1$, $j_1=k\geq3$.

(c13). If $l=2$, then $j,k\geq3$. If $l_1=1$, then  $j_1, k_1\geq3$. When $\{j,k\}\cap\{j_1,k_1\}=\emptyset$, we have
\begin{align}\nonumber
\mathbb{E}\big[T_{1jk}T_{2j_1k_1}\bar{A}_{12}^2\big|X_1,X_2\big]=\bar{A}_{12}^2\mathbb{E}\big[T_{1jk}\big|X_1,X_2\big]\mathbb{E}\big[T_{1j_1k_1}\big|X_1,X_2\big]=O(r_n^8).
\end{align}
In this case, there are at most $n^4$ such index tuples $(j,k,l,j_1,k_1,l_1)$. 
Thus, the sum in (\ref{lem5eqpoof8}) over these indices is \(O\left(r_n^4\right)\).

(c14). If $l=2$, then $j,k\geq3$. If $l_1=1$, then  $j_1, k_1\geq3$. When $\{j,k\}=\{j_1,k_1\}$, we have
\begin{align}\nonumber
|\mathbb{E}\big[T_{1jk}T_{2j_1k_1}\bar{A}_{12}^2\big|X_1,X_2\big]|\leq \bar{A}_{12}^2\mathbb{E}\big[|T_{1jk}|\big|X_1,X_2\big]=O(r_n^2).
\end{align}
In this case, there are at most $n^2$ such index tuples $(j,k,l,j_1,k_1,l_1)$. 
Thus, the sum in (\ref{lem5eqpoof8}) over these indices is \(O\left(\frac{1}{n^2r_{n}^2}\right)\).

(c15). If $l=2$, $l_1=1$ and $j=j_1$, we have
\begin{align}\nonumber
|\mathbb{E}\big[T_{1jk}T_{2jk_1}\bar{A}_{12}^2\big|X_1,X_2\big]|\leq \bar{A}_{12}^2\mathbb{E}\big[|T_{1jk}|A_{jk_1}\big|X_1,X_2\big]=O(r_n^3).
\end{align}
In this case, there are at most $n^3$ such index tuples $(j,k,l,j_1,k_1,l_1)$. 
Thus, the sum in (\ref{lem5eqpoof8}) over these indices is \(O\left(\frac{1}{nr_{n}}\right)\). The same bound holds for the cases (a) $l=2$, $l_1=1$ and $j=k_1$; (b) $l=2$, $l_1=1$ and $k=j_1$;  (c) $l=2$, $l_1=1$ and $k=k_1$.

(c16). If $l=2$ and $l_1=j$, then $j,k\geq3$. When  $j_1\neq k\neq k_1\geq3$, it follows from Lemma \ref{propmain} that
\begin{align}\nonumber
\mathbb{E}\big[T_{1jk}T_{2j_1k_1}\bar{A}_{12}\bar{A}_{2j}\big|X_1,X_2\big]=\mathbb{E}\big[T_{1jk}\bar{A}_{2j}\big|X_1,X_2\big]\mathbb{E}\big[T_{2j_1k_1}\big|X_1,X_2\big]\bar{A}_{12}=O(r_n^5).
\end{align}
In this case, there are at most $n^4$ such index tuples $(j,k,l,j_1,k_1,l_1)$. 
Thus, the sum in (\ref{lem5eqpoof8}) over these indices is \(O\left(r_n\right)\).

(c17). If $l=2$ and $l_1=j$, then $j,k\geq3$. When  $j_1=k\neq k_1\geq3$, it follows from Lemma \ref{propmain} that
\begin{align}\nonumber
|\mathbb{E}\big[T_{1jk}T_{2kk_1}\bar{A}_{12}\bar{A}_{2j}\big|X_1,X_2\big]|\leq4\mathbb{E}\big[|T_{1jk}|(A_{kk_1}+A_{2k_1})\big|X_1,X_2\big]=O(r_n^3).
\end{align}
In this case, there are at most $n^3$ such index tuples $(j,k,l,j_1,k_1,l_1)$. 
Thus, the sum in (\ref{lem5eqpoof8}) over these indices is \(O\left(\frac{1}{nr_{n}}\right)\). The same bound applies to the case $l=2$ and $l_1=j$ and $j_1\neq k=k_1\geq3$.

(c18). If $l=2$ and $l_1=j$, then $j,k\geq3$. When  $j_1=1$ and $k\neq k_1$, we have
\begin{align}\nonumber
|\mathbb{E}\big[T_{1jk}T_{21k_1}\bar{A}_{12}\bar{A}_{2j}\big|X_1,X_2\big]|\leq4\mathbb{E}\big[|T_{1jk}|(A_{1k_1}+A_{2k_1})\big|X_1,X_2\big]=O(r_n^3).
\end{align}
In this case, there are at most $n^3$ such index tuples $(j,k,l,j_1,k_1,l_1)$. 
Thus, the sum in (\ref{lem5eqpoof8}) over these indices is \(O\left(\frac{1}{nr_{n}}\right)\). The same bound applies to the case $l=2$, $l_1=j$, $k_1=1$ and $k\neq j_1$.

(c19). If $l=2$ and $l_1=j$, then $j,k\geq3$. When  $j_1=1$ and $k=k_1$, we have
\begin{align}\nonumber
|\mathbb{E}\big[T_{1jk}T_{21k_1}\bar{A}_{12}\bar{A}_{2j}\big|X_1,X_2\big]|\leq4\mathbb{E}\big[|T_{1jk}|\big|X_1,X_2\big]=O(r_n^2).
\end{align}
In this case, there are at most $n^2$ such index tuples $(j,k,l,j_1,k_1,l_1)$. 
Thus, the sum in (\ref{lem5eqpoof8}) over these indices is \(O\left(\frac{1}{n^2r_{n}^2}\right)\). The same bound applies to the case $l=2$, $l_1=j$, $k_1=1$ and $k=j_1$.

(c20). If $l=2$ and $l_1=k$, then $j,k\geq3$. When $j\neq j_1\neq k_1\geq3$, by an argument similar to that for (c16),  the sum in (\ref{lem5eqpoof8}) over these indices is \(O\left(r_{n}^2\right)\).

(c21). If $l=2$ and $l_1=k$, then $j,k\geq3$. When $j=j_1\neq k_1\geq3$ or $j=k_1\neq j_1$, by an argument similar to that for (c17),  the sum in (\ref{lem5eqpoof8}) over these indices is  \(O\left(\frac{1}{nr_{n}}\right)\).

(c22). If $l=2$,  $l_1=k$, $j_1=1$ and $j\neq k_1\geq3$, then, by an argument similar to that for (c18),  the sum in (\ref{lem5eqpoof8}) over these indices is  \(O\left(\frac{1}{nr_{n}}\right)\). The same bound holds for the case $l=2$,  $l_1=k$, $k_1=1$ and $j_1\neq k\geq3$.

(c23). If $l=2$,  $l_1=k$, $j_1=1$ and $j=k_1\geq3$, then, by an argument similar to that for (c19),  the sum in (\ref{lem5eqpoof8}) over these indices is  \(O\left(\frac{1}{n^2r_{n}^2}\right)\). The same bound holds for the case $l=2$,  $l_1=k$, $k_1=1$ and $j_1=k\geq3$.

(c24). The case \(l=j_{1}\), \(l_{1}=1\) is analogous to the case \(l=2,l_{1}=j\). By a similar arguments in (c16)-(c19), the sum in (\ref{lem5eqpoof8}) over these indices is  \(O\left(\frac{1}{nr_{n}}+r_n\right)\).

(c25).  If $l=j_1$ and $l_1=j$, then $k\not\in\{j_1,l_1\}$ and $k_1\not\in\{j,l\}$.  When  $k\neq k_1\geq3$, then
\begin{align}\nonumber
\mathbb{E}\big[\big|\mathbb{E}\big[T_{1jk}T_{2j_1k_1}\bar{A}_{1j_1}\bar{A}_{2j}\big|X_1,X_2\big]|\big]=\mathbb{E}\big[|T_{1jk}\bar{A}_{2j}T_{2j_1k_1}\bar{A}_{1j_1}|\big]=O(r_n^6).
\end{align}
In this case, there are at most $n^4$ such index tuples $(j,k,l,j_1,k_1,l_1)$. 
Thus, the sum in (\ref{lem5eqpoof8}) over these indices is \(O\left(r_n^2\right)\).

(c26).  If $l=j_1$ and $l_1=j$, then $k\not\in\{j_1,l_1\}$ and $k_1\not\in\{j,l\}$.  When  $k=k_1\geq3$, then
\begin{align}\nonumber
\mathbb{E}\big[T_{1jk}T_{2j_1k}\bar{A}_{1j_1}\bar{A}_{2j}\big|X_1,X_2\big]&=\mathbb{E}\big[T_{1jk}\bar{A}_{2j}\big|X_1,X_2\big]\mathbb{E}\big[T_{2j_1k}\bar{A}_{1j_1}\big|X_1,X_2\big]=O(r_n^4).
\end{align}
In this case, there are at most $n^3$ such index tuples $(j,k,l,j_1,k_1,l_1)$. 
Thus, the sum in (\ref{lem5eqpoof8}) over these indices is \(O\left(r_n\right)\).

(c27).  If $l=j_1$ and $l_1=j$, then $k\not\in\{j_1,l_1\}$ and $k_1\not\in\{j,l\}$.  When  $k=2$ and $k_1\geq3$, then
\begin{align}\nonumber
|\mathbb{E}\big[T_{1j2}T_{2j_1k_1}\bar{A}_{1j_1}\bar{A}_{2j}\big|X_1,X_2\big]|&\leq4\mathbb{E}\big[A_{1j}|T_{2j_1k_1}|\big|X_1,X_2\big]=O(r_n^3).
\end{align}
In this case, there are at most $n^3$ such index tuples $(j,k,l,j_1,k_1,l_1)$. 
Thus, the sum in (\ref{lem5eqpoof8}) over these indices is \(O\left(\frac{1}{nr_n}\right)\).  The same bound applies to the case $l=j_1$, $l_1=j$, $k_1=1$ and $k\geq3$.

(c28).  If $l=j_1$ and $l_1=j$, then $k\not\in\{j_1,l_1\}$ and $k_1\not\in\{j,l\}$.  When  $k=2$ and $k_1=1$, then
\begin{align}\nonumber
|\mathbb{E}\big[T_{1j2}T_{2j_11}\bar{A}_{1j_1}\bar{A}_{2j}\big|X_1,X_2\big]|&\leq4\mathbb{E}\big[|\bar{A}_{1j_1}\bar{A}_{2j}|\big|X_1,X_2\big]=O(r_n^2).
\end{align}
In this case, there are at most $n^2$ such index tuples $(j,k,l,j_1,k_1,l_1)$. 
Thus, the sum in (\ref{lem5eqpoof8}) over these indices is \(O\left(\frac{1}{nr_n}\right)\).

(c29).  If $l=j_1$ and $l_1=k$, then $j\not\in\{j_1,l_1\}$ and $k_1\not\in\{k,l\}$.  When  $j\neq k_1\geq3$, the bound in (c25) holds by a similar argument.

(c30).  If $l=j_1$ and $l_1=k$, then $j\not\in\{j_1,l_1\}$ and $k_1\not\in\{k,l\}$.  When  $j=k_1\geq3$, the bound in (c26) holds by a similar argument.

(c31).  If $l=j_1$ and $l_1=k$, then $j\not\in\{j_1,l_1\}$ and $k_1\not\in\{k,l\}$.  When  $j=2$ or $k_1=1$, the bounds in (c27) or (c28) hold by a similar argument.

(c32). The case $l=k_1$ and $l_1=1$ is analogous to the case $l=2$ and $l_1=k$. The bounds in (c20)-(c23) hold.

(c33). The case $l=k_1$ and $l_1=j$ is analogous to the case $l=j_1$ and $l_1=k$. The bounds in (c29)-(c31) hold.

(c34). If $l=k_1$ and $l_1=k$, then $j\not\in\{k_1,l_1\}$ and $j_1\not\in\{k,l\}$. When $j\neq j_1\geq3$, the bound in (c25) holds by a similar argument.

(c35). If $l=k_1$ and $l_1=k$, then $j\not\in\{k_1,l_1\}$ and $j_1\not\in\{k,l\}$. When $j=j_1\geq3$, the bound in (c26) holds by a similar argument.

(c36). If $l=k_1$ and $l_1=k$, then $j\not\in\{k_1,l_1\}$ and $j_1\not\in\{k,l\}$. When $j=2$ and $j_1\geq3$ or $j_1=1$ and $j\geq3$, the bound in (c27) holds by a similar argument.

(c37). If $l=k_1$ and $l_1=k$, then $j\not\in\{k_1,l_1\}$ and $j_1\not\in\{k,l\}$. When $j=2$ and $j_1=1$, the bound in (c28) holds by a similar argument.

In summary, we have
\begin{align*} 
\mathbb{E}\left[\frac{R_1R_2}{\mu_1^2\mu_2^2}\right]=\mathbb{E}\left[\frac{\mathbb{E}[R_1R_2|X_1,X_2]}{\mu_1^2\mu_2^2}\right]=O\left(nr_n^5+r_n+\frac{1}{nr_n}\right).
\end{align*}
Then the proof is complete.

\qed

\subsection{Proof of Lemma \ref{lemmapp} }
Recall that $\mu_t=(n-1)\mathbb{E}[A_{12}|X_t]$ for $t\in\{1,2\}$. Then
\begin{align} \nonumber
&\mathbb{E}\left[\frac{(n+1)^2\mathbb{E}\left[A_{12}\mid X_1\right]\mathbb{E}\left[A_{12}\mid X_2\right]}{\mu_1^2\mu_2^2}\mathbb{E}\left[P_1P_2|X_1,X_2\right]\right]\\ \label{lemmapepeq1}
&=\mathbb{E}\left[\frac{(n+1)^2}{(n-1)^4\mathbb{E}\left[A_{12}\mid X_1\right]\mathbb{E}\left[A_{12}\mid X_2\right]}\sum_{\substack{j\ne k \ne 1\\ j_1\ne k_1\neq 2}}\mathbb{E}\left[T_{1jk}T_{2j_1k_1}|X_1,X_2\right]\right].
\end{align}
Next, we shall derive the leading term of (\ref{lemmapepeq1}). When
 $j\ne k \ne j_1\ne k_1\geq3$,  $T_{1jk}$ and $T_{2j_1k_1}$ are conditionally independent given $X_1$ and $X_2$. Then  we have
\begin{align*}
\mathbb{E}\left[T_{1jk}T_{2j_1k_1}|X_1,X_2\right]&=\mathbb{E}\left[T_{1jk}|X_1\right]\mathbb{E}\left[T_{2j_1k_1}|X_2\right]\\
&=\mathbb{E}\left[(A_{13}A_{34}
    -A_{13}A_{14})|X_1\right]\mathbb{E}\left[(A_{25}A_{56}
    -A_{25}A_{26})|X_2\right].
\end{align*}
Lemma \ref{propmain} implies that \(\mathbb{E}[A_{12}|X_{t}]=\Theta (r_{n})\) and \(\mathbb{E}[T_{123}|X_{t}]=O(r_{n}^{4})\). Therefore, the sum over indices  $j\ne k \ne j_1\ne k_1\geq3$ in (\ref{lemmapepeq1}) is equal to 
\begin{align}\nonumber
&\mathbb{E}\left[\frac{(n+1)^2}{(n-1)^4\mathbb{E}\left[A_{12}\mid X_1\right]\mathbb{E}\left[A_{12}\mid X_2\right]}\sum_{\substack{j\ne k \ne j_1\ne k_1\geq 3}}T_{1jk}T_{2j_1k_1}\right]\\ \nonumber
&=\mathbb{E}\left[\frac{(n+1)^2}{(n-1)^4\mathbb{E}\left[A_{12}\mid X_1\right]\mathbb{E}\left[A_{12}\mid X_2\right]}\sum_{\substack{j\ne k \ne j_1\ne k_1\geq 3}}\mathbb{E}\left[T_{1jk}T_{2j_1k_1}|X_1,X_2\right]\right]\\ \nonumber
&= \frac{(n+1)^2(n-3)(n-4)(n-5)(n-6)}{(n-1)^4}\left(\mathbb{E}\left[\frac{\mathbb{E}\left[(A_{13}A_{34}
    -A_{13}A_{14})|X_1\right]}{\mathbb{E}\left[A_{12}\mid X_1\right]}\right]\right)^2\\ \label{lemmapepeq2}
&=\left(n\mathbb{E}\left[\frac{\mathbb{E}\left[(A_{13}A_{34}
    -A_{13}A_{14})|X_1\right]}{\mathbb{E}\left[A_{12}\mid X_1\right]}\right]\right)^2+O(nr_n^6).
\end{align}

Suppose $\{j, k\}=\{j_1, k_1\}$. In this case, $j,k,j_1,k_1\geq3$. If $j=j_1$ and $k=k_1$, then
\begin{align*}
|T_{1jk}T_{2jk}|&\leq A_{1j}A_{2j}A_{jk}+A_{1j}A_{2j}A_{jk}A_{2k}+A_{1j}A_{1k}A_{2j}A_{jk}+A_{1j}A_{2j}A_{1k}A_{2k}\\
&\leq 3A_{1j}A_{2j}A_{jk}+A_{1j}A_{2j}A_{2k}.
\end{align*}
By Lemma \ref{propmain}, it is easy to obtain
\begin{align*}
\mathbb{E}[A_{1j}A_{2j}A_{jk}]=O(r_n^3),\ \ \ \ \ \mathbb{E}[A_{1j}A_{2j}A_{2k}]=O(r_n^3).
\end{align*}
Therefore, we have
\begin{align}\nonumber
&\mathbb{E}\left[\frac{(n+1)^2}{(n-1)^4\mathbb{E}\left[A_{12}\mid X_1\right]\mathbb{E}\left[A_{12}\mid X_2\right]}\left|\sum_{\substack{j\ne k\geq3 }}T_{1jk}T_{2jk}\right|\right]\\ \nonumber
&\leq\mathbb{E}\left[\frac{(n+1)^2}{(n-1)^4\mathbb{E}\left[A_{12}\mid X_1\right]\mathbb{E}\left[A_{12}\mid X_2\right]}\sum_{\substack{j\ne k\geq3}}\mathbb{E}\left[|T_{1jk}T_{2jk}||X_1,X_2\right]\right]\\ \nonumber
&= O\left(\frac{1}{r_n^2}\right)\mathbb{E}\left[\mathbb{E}\left[|T_{1jk}T_{2jk}||X_1,X_2\right]\right]\\ \nonumber
&= O\left(\frac{1}{r_n^2}\right)\mathbb{E}\left[|T_{1jk}T_{2jk}|\right]\\ \label{lemmapepeq3}
&=O(r_n).
\end{align}
If $j=k_1$ and $k=j_1$, then
\begin{align*}
|T_{1jk}T_{2kj}|&\leq A_{1j}A_{jk}A_{2k}+A_{1j}A_{jk}A_{2k}A_{2j}+A_{1j}A_{1k}A_{2k}A_{kj}+A_{1j}A_{1k}A_{2k}A_{2j}\\
&\leq A_{1j}A_{jk}A_{2k}+A_{1j}A_{jk}A_{2j}+A_{1k}A_{2k}A_{kj}+A_{1j}A_{1k}A_{2j}
\end{align*}
Hence, $\mathbb{E}\left[|T_{1jk}T_{2kj}|\right]=O(r_n^3)$. Similar to (\ref{lemmapepeq3}), we have
\begin{align}\label{lemmapepeq4}
\mathbb{E}\left[\frac{(n+1)^2}{(n-1)^4\mathbb{E}\left[A_{12}\mid X_1\right]\mathbb{E}\left[A_{12}\mid X_2\right]}\left|\sum_{\substack{j\ne k\geq3 }}T_{1jk}T_{2kj}\right|\right]=O(r_n).
\end{align}

Suppose $|\{j, k\}\cap\{j_1, k_1\}|=1$ and $j,k,j_1,k_1\geq3$. If $j=j_1$ and $k\neq k_1$, by Lemma \ref{propmain}, we get
\begin{align}\nonumber
&\frac{1}{\mathbb{E}\left[A_{12}\mid X_1\right]\mathbb{E}\left[A_{12}\mid X_2\right]}\mathbb{E}[T_{1jk}T_{2jk_1}|X_1,X_2]\\ \nonumber
&= \frac{\mathbb{E}\big[(A_{1j}A_{2j}A_{jk}A_{jk_1}-A_{1j}A_{2j}A_{jk}A_{2k_1}-A_{1j}A_{1k}A_{2j}A_{jk_1}+A_{1j}A_{2j}A_{1k}A_{2k_1})|X_1,X_2\big]}{\mathbb{E}\left[A_{12}\mid X_1\right]\mathbb{E}\left[A_{12}\mid X_2\right]}\\ \nonumber
&=\frac{1}{\mathbb{E}\left[A_{12}\mid X_1\right]\mathbb{E}\left[A_{12}\mid X_2\right]}\Big((2r_n)^2\mathbb{E}\big[A_{1j}A_{2j}f^2(X_j)|X_1,X_2\big]\\ \nonumber
&\quad-(2r_n)^2\mathbb{E}\big[A_{1j}A_{2j}f(X_j)f(X_2)|X_1,X_2\big]-(2r_n)^2\mathbb{E}\big[A_{1j}A_{2j}f(X_j)f(X_1)|X_1,X_2\big]\\  \nonumber
&\quad+(2r_n)^2\mathbb{E}\big[A_{1j}A_{2j}f(X_1)f(X_2)|X_1,X_2\big]+O(r_n^4)\mathbb{E}\big[A_{1j}A_{2j}|X_1,X_2\big]\\ \label{lemmapepeq5}
&\quad+O(r_n^4)\mathbb{E}\big[A_{1j}A_{2j}|X_1,X_2\big]\Big)+O(r_n^4)\mathbb{E}\big[A_{1j}A_{2j}|X_1,X_2\big]+O(r_n^4)\mathbb{E}\big[A_{1j}A_{2j}|X_1,X_2\big].
\end{align}
We will derive asymptotic expressions or upper bounds for the expectation of each term in (\ref{lemmapepeq5}).

By Lemma \ref{propmain}, $\mathbb{E}\big[A_{12}|X_1\big]=2r_nf(X_1)+O(r_n^3)$. It follows from Assumption \ref{assumptiona} that $f(x)$ is bounded away from zero and  $f(x)^2$ is bounded.  Then the expectation of the first term in (\ref{lemmapepeq5}) is given by
\begin{align*}\nonumber
(2r_n)^2\mathbb{E}\left[\frac{\mathbb{E}\big[A_{1j}A_{2j}f^2(X_j)|X_1,X_2\big]}{\mathbb{E}\big[A_{13}|X_1\big]\mathbb{E}\big[A_{24}|X_2\big]}\right]&=\mathbb{E}\left[\frac{\mathbb{E}\big[A_{1j}A_{2j}f^2(X_j)|X_1,X_2\big]}{f(X_1)f(X_2)}\right]+O(r_n^2)\mathbb{E}[A_{1j}A_{2j}]\\ \nonumber
&=\mathbb{E}\left[\frac{A_{1j}A_{2j}f^2(X_j)}{f(X_1)f(X_2)}\right]+O(r_n^4)\\ 
&=\mathbb{E}\left[f^2(X_j)\mathbb{E}\left[\frac{A_{1j}A_{2j}}{f(X_1)f(X_2)}\Bigg|X_j\right]\right]+O(r_n^4).
\end{align*}
Note that
\begin{align*}
&\mathbb{E}\left[\frac{A_{1j}A_{2j}}{f(X_1)f(X_2)}\Bigg|X_j\right]\\
&=I[r_n\leq X_j\leq 1-r_n]\int_{X_j-r_n}^{X_j+r_n}dx_1\int_{X_j-r_n}^{X_j+r_n}dx_2\\
&\quad+I[0\leq X_j\leq r_n]\left(\int_{0}^{X_j+r_n}dx_1+\int_{1-r_n+X_j}^{1}dx_1\right)\left(\int_{0}^{X_j+r_n}dx_2+\int_{1-r_n+X_j}^{1}dx_2\right)\\
&\quad+I[1-r_n\leq X_j\leq 1]\left(\int_{X_j-r_n}^1dx_1+\int_0^{r_n+X_j-1}dx_1\right)\left(\int_{X_j-r_n}^1dx_2+\int_0^{r_n+X_j-1}dx_2\right)\\
&=(2r_n)^2.
\end{align*}
It then follows that
\begin{align}\label{lemmapepeq6}
(2r_n)^2\mathbb{E}\left[\frac{\mathbb{E}\big[A_{1j}A_{2j}f^2(X_j)|X_1,X_2\big]}{\mathbb{E}\big[A_{13}|X_1\big]\mathbb{E}\big[A_{24}|X_2\big]}\right]=(2r_n)^2\mathbb{E}\left[f^2(X_j)\right]+O(r_n^4).
\end{align}

The expectation of the second term in (\ref{lemmapepeq5}) is given by
\begin{align}\nonumber
&(2r_n)^2\mathbb{E}\left[\frac{\mathbb{E}\big[A_{1j}A_{2j}f(X_j)f(X_2)|X_1,X_2\big]}{\mathbb{E}\big[A_{13}|X_1\big]\mathbb{E}\big[A_{24}|X_2\big]}\right]\\ \nonumber
&=\mathbb{E}\left[\frac{\mathbb{E}\big[A_{1j}A_{2j}f(X_j)f(X_2)|X_1,X_2\big]}{f(X_1)f(X_2)}\right]+O(r_n^2)\mathbb{E}[A_{1j}A_{2j}]\\ \nonumber
&=\mathbb{E}\left[\frac{A_{1j}A_{2j}f(X_j)}{f(X_1)}\right]+O(r_n^4)\\ \nonumber
&=\mathbb{E}\left[f(X_j)A_{2j}\mathbb{E}\left[\frac{A_{1j}}{f(X_1)}\Bigg|X_j,X_2\right]\right]+O(r_n^4)\\ \nonumber
&=2r_n\mathbb{E}\left[f(X_j)A_{2j}\right]+O(r_n^4)\\ \nonumber
&=2r_n\mathbb{E}\left[f(X_j)\mathbb{E}[A_{2j}|X_j]\right]+O(r_n^4)\\ \nonumber
&=2r_n\mathbb{E}\left[2r_nf^2(X_j)+\frac{f(X_j)f''(X_j)}{3} r_n^3 \right]+O(r_n^4)\\ \label{lemmapepeq7}
&=(2r_n)^2\mathbb{E}\left[f^2(X_j)\right]+O(r_n^4),
\end{align}
where the fourth equality follows from the following argument
\begin{align}\nonumber
\mathbb{E}\left[\frac{A_{1j}}{f(X_1)}\Bigg|X_j,X_2\right]
&=I[r_n\leq X_j\leq 1-r_n]\int_{X_j-r_n}^{X_j+r_n}dx_1\\ \nonumber
&\quad+I[0\leq X_j\leq r_n]\left(\int_{0}^{X_j+r_n}dx_1+\int_{1-r_n+X_j}^{1}dx_1\right)\\ \nonumber
&\quad+I[1-r_n\leq X_j\leq 1]\left(\int_{X_j-r_n}^1dx_1+\int_0^{r_n+X_j-1}dx_1\right)\\ \label{aera1}
&=(2r_n).
\end{align}

The expectation of the third term  in (\ref{lemmapepeq5}) is identical to the second term. The expectation of the fourth term in (\ref{lemmapepeq5}) is given by
\begin{align}\nonumber
&(2r_n)^2\mathbb{E}\left[\frac{\mathbb{E}\big[A_{1j}A_{2j}f(X_1)f(X_2)|X_1,X_2\big]}{\mathbb{E}\big[A_{13}|X_1\big]\mathbb{E}\big[A_{24}|X_2\big]}\right]\\ \nonumber
&=\mathbb{E}\left[\frac{\mathbb{E}\big[A_{1j}A_{2j}f(X_1)f(X_2)|X_1,X_2\big]}{f(X_1)f(X_2)}\right]+O(r_n^2)\mathbb{E}[A_{1j}A_{2j}]\\ \nonumber
&=\mathbb{E}[\mathbb{E}[A_{1j}|X_j]\mathbb{E}[A_{2j}|X_j]]+O(r_n^4)\\ \nonumber
&=\mathbb{E}\left[(2r_n)^2f^2(X_j) + \frac{r_n^4 }{3} f(X_j) f''(X_j) \right]+O(r_n^4)\\ \label{lemmapepeq8}
&=(2r_n)^2\mathbb{E}\left[f^2(X_j)\right]+O(r_n^4).
\end{align}

Combining (\ref{lemmapepeq5})-(\ref{lemmapepeq8}) yields
\begin{align*}
\mathbb{E}\left[\frac{\mathbb{E}\big[T_{1jk}T_{2jk_1}|X_1,X_2\big]}{\mathbb{E}\big[A_{13}|X_1\big]\mathbb{E}\big[A_{24}|X_2\big]}\right]=O\left(r_n^4\right).
\end{align*}
Therefore, we have
\begin{align} \label{lemmapepeq9}
\mathbb{E}\left[\frac{(n+1)^2}{(n-1)^4\mathbb{E}\left[A_{12}\mid X_1\right]\mathbb{E}\left[A_{12}\mid X_2\right]}\sum_{\substack{j\ne k\ne k_1\geq3 }}T_{1jk}T_{2jk_1}\right]=O(nr_n^4).
\end{align}

Suppose $j=k_1$, $k\neq j_1$, and $j,k,j_1,k_1\geq3$. Then
\begin{align}\nonumber
&\frac{1}{\mathbb{E}\left[A_{12}\mid X_1\right]\mathbb{E}\left[A_{12}\mid X_2\right]}\mathbb{E}[T_{1jk}T_{2j_1j}|X_1,X_2]\\ \nonumber
&= \frac{\mathbb{E}\big[(A_{1j}A_{jk}A_{2j_1}A_{j_1j}-A_{1j}A_{jk}A_{2j_1}A_{2j}-A_{1j}A_{1k}A_{2j_1}A_{j_1j}+A_{1j}A_{2j_1}A_{1k}A_{2j})|X_1,X_2\big]}{\mathbb{E}\left[A_{12}\mid X_1\right]\mathbb{E}\left[A_{12}\mid X_2\right]}\\ \nonumber
&=\frac{1}{\mathbb{E}\left[A_{12}\mid X_1\right]\mathbb{E}\left[A_{12}\mid X_2\right]}\Big((2r_n)\mathbb{E}\big[A_{1j}A_{2j_1}A_{j_1j}f(X_j)|X_1,X_2\big]\\ \nonumber
&\quad-(2r_n)\mathbb{E}\big[A_{1j}A_{2j_1}A_{2j}f(X_j)|X_1,X_2\big]-(2r_n)\mathbb{E}\big[A_{1j}A_{2j_1}A_{j_1j}f(X_1)|X_1,X_2\big]\\ \nonumber
&\quad+(2r_n)\mathbb{E}\big[A_{1j}A_{2j_1}A_{2j}f(X_1)|X_1,X_2\big]+O(r_n^3)\mathbb{E}\big[A_{1j}A_{2j_1}A_{j_1j}|X_1,X_2\big]\\ \nonumber
&\quad+O(r_n^3)\mathbb{E}\big[A_{1j}A_{2j_1}A_{j_1j}|X_1,X_2\big]+O(r_n^3)\mathbb{E}\big[A_{1j}A_{2j_1}A_{2j}|X_1,X_2\big]\\ \label{lemmapepeq10}
&\quad+O(r_n^3)\mathbb{E}\big[A_{1j}A_{2j_1}A_{2j}|X_1,X_2\big]\Big).
\end{align}
By (\ref{aera1}), the expectation of the first term of (\ref{lemmapepeq10}) is equal to
\begin{align}\nonumber
   &(2r_n)\mathbb{E}\left[ \frac{\mathbb{E}\big[A_{1j}A_{2j_1}A_{j_1j}f(X_j)|X_1,X_2\big]}{\mathbb{E}\left[A_{12}\mid X_1\right]\mathbb{E}\left[A_{12}\mid X_2\right]} \right]\\ \nonumber
   &=\frac{1}{2r_n}\mathbb{E}\left[ \frac{A_{1j}A_{2j_1}A_{j_1j}f(X_j)}{f(X_1)f(X_2)} \right]+O(r_n^3)\mathbb{E}\big[A_{1j}A_{2j_1}A_{j_1j}\big]\\ \nonumber
   &=\frac{1}{2r_n}\mathbb{E}\left[A_{j_1j}f(X_j) \mathbb{E}\left[ \frac{A_{1j}}{f(X_1)} \Big|X_j\right]\mathbb{E}\left[ \frac{A_{2j_1}}{f(X_2)} \Big|X_{j_1}\right]\right]+O(r_n^4)\\ \nonumber
   &=2r_n\mathbb{E}\left[A_{j_1j}f(X_j)\right]+O(r_n^4)\\ \label{00lemmapepeq10}
   &=(2r_n)^2\mathbb{E}\left[f^2(X_j)\right]+O(r_n^4).
\end{align}
By (\ref{aera1}), the expectation of the second term of (\ref{lemmapepeq10}) is equal to
\begin{align}\nonumber
   &(2r_n)\mathbb{E}\left[ \frac{\mathbb{E}\big[A_{1j}A_{2j_1}A_{2j}f(X_j)|X_1,X_2\big]}{\mathbb{E}\left[A_{12}\mid X_1\right]\mathbb{E}\left[A_{12}\mid X_2\right]} \right]\\ \nonumber
   &=\frac{1}{2r_n}\mathbb{E}\left[ \frac{A_{1j}A_{2j_1}A_{2j}f(X_j)}{f(X_1)f(X_2)} \right]+O(r_n^3)\mathbb{E}\big[A_{1j}A_{2j_1}A_{2j}\big]\\ \nonumber
   &=\frac{1}{2r_n}\mathbb{E}\left[A_{j2}f(X_j) \mathbb{E}\left[ \frac{A_{1j}}{f(X_1)} \Big|X_j\right]\mathbb{E}\left[ \frac{A_{2j_1}}{f(X_2)} \Big|X_{2}\right]\right]+O(r_n^4)\\ \nonumber
   &=2r_n\mathbb{E}\left[A_{j2}f(X_j)\right]+O(r_n^4)\\ \label{000lemmapepeq10}
   &=(2r_n)^2\mathbb{E}\left[f^2(X_j)\right]+O(r_n^4).
\end{align}
By (\ref{aera1}), the expectation of the third term of (\ref{lemmapepeq10}) is equal to
\begin{align}\nonumber
   &(2r_n)\mathbb{E}\left[ \frac{\mathbb{E}\big[A_{1j}A_{2j_1}A_{j_1j}f(X_1)|X_1,X_2\big]}{\mathbb{E}\left[A_{12}\mid X_1\right]\mathbb{E}\left[A_{12}\mid X_2\right]} \right]\\ \nonumber
   &=\frac{1}{2r_n}\mathbb{E}\left[ \frac{A_{1j}A_{2j_1}A_{j_1j}}{f(X_2)} \right]+O(r_n^3)\mathbb{E}\big[A_{1j}A_{2j_1}A_{j_1j}\big]\\ \nonumber
   &=\frac{1}{2r_n}\mathbb{E}\left[A_{jj_1} \mathbb{E}\left[ A_{1j} \big|X_j\right]\mathbb{E}\left[ \frac{A_{2j_1}}{f(X_2)} \Big|X_{j_1}\right]\right]+O(r_n^4)\\ \nonumber
   &=2r_n\mathbb{E}\left[A_{jj_1}f(X_j)\right]+O(r_n^4)\\ \label{40lemmapepeq10}
   &=(2r_n)^2\mathbb{E}\left[f^2(X_j)\right]+O(r_n^4).
\end{align}
The expectation of the fourth term of (\ref{lemmapepeq10}) is equal to
\begin{align}\nonumber
   &(2r_n)\mathbb{E}\left[ \frac{\mathbb{E}\big[A_{1j}A_{2j_1}A_{2j}f(X_1)|X_1,X_2\big]}{\mathbb{E}\left[A_{12}\mid X_1\right]\mathbb{E}\left[A_{12}\mid X_2\right]} \right]\\ \nonumber
   &=\frac{1}{2r_n}\mathbb{E}\left[ \frac{A_{1j}A_{2j_1}A_{2j}}{f(X_2)} \right]+O(r_n^3)\mathbb{E}\big[A_{1j}A_{2j_1}A_{2j}\big]\\ \nonumber
   &=\frac{1}{2r_n}\mathbb{E}\left[\frac{A_{2j}}{f(X_2)} \mathbb{E}\left[ A_{1j} \big|X_j\right]\mathbb{E}\left[ A_{2j_1} \big|X_{2}\right]\right]+O(r_n^4)\\ \nonumber
   &=2r_n\mathbb{E}\left[A_{2j}f(X_j)\right]+O(r_n^4)\\ \label{50lemmapepeq10}
   &=(2r_n)^2\mathbb{E}\left[f^2(X_j)\right]+O(r_n^4).
\end{align}
Note that $\mathbb{E}\big[A_{13}|X_1\big]=\Theta(r_n)$. By Lemma \ref{propmain}, it is easy to get
\begin{align}\label{lemmapepeq11}
\mathbb{E}\left[\frac{\mathbb{E}\big[A_{1j}A_{2j_1}A_{2j}|X_1,X_2\big]}{\mathbb{E}\big[A_{13}|X_1\big]\mathbb{E}\big[A_{24}|X_2\big]}\right]=O\left(\frac{1}{r_n^2}\right)\mathbb{E}\big[A_{1j}A_{2j_1}A_{2j}\big]=O(r_n),
\end{align}
and
\begin{align}\label{lemmapepeq12}
\mathbb{E}\left[\frac{\mathbb{E}\big[A_{1j}A_{2j_1}A_{j_1j}|X_1,X_2\big]}{\mathbb{E}\big[A_{13}|X_1\big]\mathbb{E}\big[A_{24}|X_2\big]}\right]=O\left(\frac{1}{r_n^2}\right)\mathbb{E}\big[A_{1j}A_{2j_1}A_{j_1j}\big]=O(r_n).
\end{align}

Combining (\ref{lemmapepeq10})-(\ref{lemmapepeq12}) yields
\begin{align}\label{lemmapepeq13}
\mathbb{E}\left[\frac{(n+1)^2}{(n-1)^4\mathbb{E}\left[A_{12}\mid X_1\right]\mathbb{E}\left[A_{12}\mid X_2\right]}\sum_{\substack{j\ne k\ne j_1\geq3 }}T_{1jk}T_{2j_1j}\right]=O(nr_n^4).
\end{align}

By (\ref{lemmapepeq2}), (\ref{lemmapepeq3}), (\ref{lemmapepeq9}) and (\ref{lemmapepeq13}), we have
\begin{align}\nonumber
&\mathbb{E}\left[\frac{(n+1)^2\mathbb{E}\left[A_{12}\mid X_1\right]\mathbb{E}\left[A_{12}\mid X_2\right]\sum_{\substack{j\ne k\geq3, j_1\ne k_1\geq 3}}T_{1jk}T_{2j_1k_1}}{\mu_1^2\mu_2^2}\right]\\ \label{lemmapepeq14}
&=\left(n\mathbb{E}\left[\frac{\mathbb{E}\left[(A_{13}A_{34}
    -A_{13}A_{14})|X_1\right]}{\mathbb{E}\left[A_{12}\mid X_1\right]}\right]\right)^2+O(nr_n^4+r_n).
\end{align}

Suppose $j=2$ and $k\neq j_1\neq k_1\geq3$. Then
\begin{align*}
&\mathbb{E}[T_{12k}T_{2j_1k_1}|X_1,X_2]\\
&= \mathbb{E}\big[(A_{12}A_{2k}A_{2j_1}A_{j_1k_1}-A_{12}A_{2k}A_{2j_1}A_{2k_1}-A_{12}A_{1k}A_{2j_1}A_{j_1k_1}+A_{12}A_{2j_1}A_{1k}A_{2k_1})|X_1,X_2\big]\\
&= A_{12}\mathbb{E}\big[A_{2k}|X_1,X_2\big]\mathbb{E}\big[A_{2j_1}A_{j_1k_1}|X_1,X_2\big]-A_{12}(\mathbb{E}\big[A_{2k}|X_1,X_2\big])^3\\
&\quad-A_{12}\mathbb{E}\big[A_{1k}|X_1,X_2\big]\mathbb{E}\big[A_{2j_1}A_{j_1k_1}\big]+A_{12}\mathbb{E}\big[A_{1k}|X_1,X_2\big](\mathbb{E}\big[A_{2k}|X_1,X_2\big])^2\\
&=(2r_n)^3\mathbb{E}\big[A_{12}f^3(X_2)|X_1,X_2\big]+O(r_n^5)\mathbb{E}\big[A_{12}|X_1,X_2\big]\\
&\quad-(2r_n)^3\mathbb{E}\big[A_{12}f^3(X_2)|X_1,X_2\big]+O(r_n^5)\mathbb{E}\big[A_{12}|X_1,X_2\big]\\
&\quad-(2r_n)^3\mathbb{E}\big[A_{12}f^2(X_{2})f(X_1)|X_1,X_2\big]+O(r_n^5)\mathbb{E}\big[A_{12}|X_1,X_2\big]\\
&\quad+(2r_n)^3\mathbb{E}\big[A_{12}f^2(X_{2})f(X_1)|X_1,X_2\big]+O(r_n^5)\mathbb{E}\big[A_{12}|X_1,X_2\big]
\end{align*}
Hence, we get
\begin{align*} 
\mathbb{E}\left[\frac{\mathbb{E}\big[T_{12k}T_{2j_1k_1}|X_1,X_2\big]}{\mathbb{E}\big[A_{13}|X_1\big]\mathbb{E}\big[A_{24}|X_2\big]}\right]&=O(r_n^5)\mathbb{E}\left[\frac{A_{12}}{(2r_n)^2f(X_1)f(X_2)}\right]=O(r_n^4).
\end{align*}
Therefore, we have
\begin{align}\label{lemmapepeq15}
\mathbb{E}\left[\frac{(n+1)^2\mathbb{E}\left[A_{12}\mid X_1\right]\mathbb{E}\left[A_{12}\mid X_2\right]\sum_{\substack{k\neq j_1\ne k_1\geq 3}}T_{12k}T_{2j_1k_1}}{\mu_1^2\mu_2^2}\right]=O(nr_n^4).
\end{align}

Suppose $j=2$ and $k=j_1\neq k_1\geq3$. Then
\begin{align*}
&\mathbb{E}[|T_{12k}T_{2kk_1}|]\\
&\leq \mathbb{E}\Big[(A_{12}A_{2k}A_{kk_1}+A_{12}A_{2k}A_{2k_1}+A_{12}A_{1k}A_{2k}A_{kk_1}+A_{12}A_{2k}A_{1k}A_{2k_1})\Big]=O(r_n^3).
\end{align*}
Therefore, we have
\begin{align}\nonumber
&\mathbb{E}\left[\left|\frac{(n+1)^2\mathbb{E}\left[A_{12}\mid X_1\right]\mathbb{E}\left[A_{12}\mid X_2\right]\sum_{\substack{k\ne k_1\geq 3}}T_{12k}T_{2kk_1}}{\mu_1^2\mu_2^2}\right|\right]\\ \nonumber
&=O\left(\frac{(n+1)^2(n-3)(n-4)}{(n-1)^4r_n^2}\right)\mathbb{E}[|T_{12k}T_{2kk_1}|]\\ \label{lemmapepeq16}
&=O(r_n).
\end{align}

Suppose $j=2$ and $k=k_1\neq j_1\geq3$. Then
\begin{align*}
&\mathbb{E}[|T_{12k}T_{2j_1k}|]\\
&\leq \mathbb{E}\Big[(A_{12}A_{2k}A_{2j_1}A_{j_1k}+A_{12}A_{2k}A_{2j_1}+A_{12}A_{1k}A_{2j_1}A_{j_1k}+A_{12}A_{2k}A_{1k}A_{2j_1})\Big]=O(r_n^3).
\end{align*}
Therefore, we have
\begin{align}\nonumber
&\mathbb{E}\left[\left|\frac{(n+1)^2\mathbb{E}\left[A_{12}\mid X_1\right]\mathbb{E}\left[A_{12}\mid X_2\right]\sum_{\substack{k\ne j_1\geq 3}}T_{12k}T_{2j_1k}}{\mu_1^2\mu_2^2}\right|\right]\\ \nonumber
&=O\left(\frac{(n+1)^2(n-3)(n-4)}{(n-1)^4r_n^2}\right)\mathbb{E}[|T_{12k}T_{2j_1k}|]\\ \label{lemmapepeq17}
&=O(r_n).
\end{align}

Suppose $j=2$ and $j_1=1$. In this case,  $k,k_1\geq3$ and  $T_{12k}T_{21k_1}$ can be expressed as 
\[T_{12k}T_{21k_1}= A_{12}(A_{1k_1}A_{2k}- A_{2k}A_{2k_1}-A_{1k}A_{1k_1}+A_{1k}A_{2k_1}).\]
If $k_1\neq k$, then
\begin{align*}
\mathbb{E}[T_{12k}T_{21k_1}]=O(r_n^3).
\end{align*}
If $k_1=k$, then
\begin{align*}
\mathbb{E}[T_{12k}T_{21k}]=O(r_n^2).
\end{align*}
Therefore, we have
\begin{align}\nonumber
&\mathbb{E}\left[\left|\frac{(n+1)^2\mathbb{E}\left[A_{12}\mid X_1\right]\mathbb{E}\left[A_{12}\mid X_2\right]\sum_{\substack{k, k_1\geq 3}}T_{12k}T_{21k_1}}{\mu_1^2\mu_2^2}\right|\right]\\ \nonumber
&=O\left(\frac{(n+1)^2(n-3)(n-4)r_n^3}{(n-1)^4r_n^2}\right)+O\left(\frac{(n+1)^2(n-3)r_n^2}{(n-1)^4r_n^2}\right) \\  \label{lemmapepeq18}
&=O\left(r_n+\frac{1}{n}\right).
\end{align}

Suppose $j=2$ and $k_1=1$.  In this case, $j_1, k\geq3$, and $T_{12k}T_{2j_11}$ can be expressed as
\begin{align*}
T_{12k}T_{2j_11}= A_{12}A_{1j_1}A_{2k}A_{2j_1}- A_{12}A_{2k}A_{2j_1}-A_{12}A_{1k}A_{1j_1}A_{2j_1}+A_{12}A_{1k}A_{2j_1}.
\end{align*}
If $j_1\neq k$, then
\begin{align*}
\mathbb{E}[|T_{12k}T_{2j_11}|]=O(r_n^3).
\end{align*}
If $j_1=k$, then
\begin{align*}
\mathbb{E}[|T_{12k}T_{2j_11}|]=O(r_n^2).
\end{align*}
Therefore, we have
\begin{align}\label{lemmapepeq19}
\mathbb{E}\left[\left|\frac{(n+1)^2\mathbb{E}\left[A_{12}\mid X_1\right]\mathbb{E}\left[A_{12}\mid X_2\right]\sum_{\substack{k, j_1\geq 3}}T_{12k}T_{2j_11}}{\mu_1^2\mu_2^2}\right|\right]=O\left(r_n+\frac{1}{n}\right).
\end{align}

Suppose $k=2$ and $j\neq j_1\neq k_1\geq3$. In this case, the conditional expectation of \(T_{1j2}T_{2j_{1}k_{1}}\) given $X_1,X_2$ is given by
\begin{align*}
&\mathbb{E}[T_{1j2}T_{2j_1k_1}|X_1,X_2]\\
&= \mathbb{E}\big[(A_{1j}A_{j2}A_{2j_1}A_{j_1k_1}-A_{1j}A_{j2}A_{2j_1}A_{2k_1}-A_{1j}A_{12}A_{2j_1}A_{j_1k_1}+A_{1j}A_{2j_1}A_{12}A_{2k_1})|X_1,X_2\big]\\
&=(2r_n)^2\mathbb{E}\big[A_{1j}A_{j2}f^2(X_2)|X_1,X_2\big]+O(r_n^4)\mathbb{E}\big[A_{1j}A_{j2}|X_1,X_2\big]\\
&\quad-(2r_n)^2\mathbb{E}\big[A_{1j}A_{j2}f^2(X_2)|X_1,X_2\big]+O(r_n^4)\mathbb{E}\big[A_{1j}A_{j2}|X_1,X_2\big]\\
&\quad-(2r_n)^2\mathbb{E}\big[A_{1j}A_{12}f^2(X_2)|X_1,X_2\big]+O(r_n^4)\mathbb{E}\big[A_{1j}A_{12}|X_1,X_2\big]\\
&\quad+(2r_n)^2\mathbb{E}\big[A_{1j}A_{12}f^2(X_2)|X_1,X_2\big]+O(r_n^4)\mathbb{E}\big[A_{1j}A_{12}|X_1,X_2\big]\\
&=O(r_n^4)\mathbb{E}\big[A_{1j}A_{j2}|X_1,X_2\big]+O(r_n^4)\mathbb{E}\big[A_{1j}A_{12}|X_1,X_2\big].
\end{align*}
Therefore, we have
\begin{align}\nonumber
&\mathbb{E}\left[\frac{(n+1)^2\mathbb{E}\left[A_{12}\mid X_1\right]\mathbb{E}\left[A_{12}\mid X_2\right]\sum_{\substack{j\neq j_1\ne k_1\geq 3}}T_{1j2}T_{2j_1k_1}}{\mu_1^2\mu_2^2}\right]\\ \nonumber
&=\frac{(n+1)^2(n-3)(n-4)(n-5)}{(n-1)^4}\mathbb{E}\left[\frac{\mathbb{E}\big[T_{1j2}T_{2j_1k_1}|X_1,X_2\big]}{\mathbb{E}\big[A_{13}|X_1\big]\mathbb{E}\big[A_{24}|X_2\big]}\right]\\ \nonumber
&=O(nr_n^4)\mathbb{E}\left[\frac{\mathbb{E}\big[A_{1j}A_{j2}|X_1,X_2\big]+\mathbb{E}\big[A_{1j}A_{12}|X_1,X_2\big]}{\mathbb{E}\big[A_{13}|X_1\big]\mathbb{E}\big[A_{24}|X_2\big]}\right]\\ \label{lemmapepeq20}
&=O(nr_n^4),
\end{align}
where the last equality follows by an argument similar to that used in (\ref{50lemmapepeq10}).

Suppose $k=2$ and $j=j_1\neq k_1\geq3$. Then
\begin{align*}
&\mathbb{E}[|T_{1j2}T_{2jk_1}|]\\
&\leq \mathbb{E}\big[A_{1j}A_{j2}A_{jk_1}+A_{1j}A_{j2}A_{2k_1}+A_{1j}A_{12}A_{2j}A_{jk_1}+A_{1j}A_{2j}A_{12}A_{2k_1}\big]=O(r_n^3).
\end{align*}
Therefore, we have
\begin{align}\label{lemmapepeq21}
\mathbb{E}\left[\left|\frac{(n+1)^2\mathbb{E}\left[A_{12}\mid X_1\right]\mathbb{E}\left[A_{12}\mid X_2\right]\sum_{\substack{j_1\ne k_1\geq 3}}T_{1j2}T_{2jk_1}}{\mu_1^2\mu_2^2}\right|\right]=O(r_n).
\end{align}

Suppose $k=2$ and $j=k_1\neq j_1\geq3$. Then
\begin{align*}
&\mathbb{E}[|T_{1j2}T_{2j_1j}|]\\
&\leq \mathbb{E}\big[A_{1j}A_{j2}A_{2j_1}A_{jj_1}+A_{1j}A_{j2}A_{2j_1}+A_{1j}A_{12}A_{2j}A_{jj_1}+A_{1j}A_{2j}A_{12}A_{2j_1})\big]\\
&=O(r_n^3).
\end{align*}
Therefore, we have
\begin{align}\label{lemmapepeq22}
\mathbb{E}\left[\left|\frac{(n+1)^2\mathbb{E}\left[A_{12}\mid X_1\right]\mathbb{E}\left[A_{12}\mid X_2\right]\sum_{\substack{j\ne j_1\geq 3}}T_{1j2}T_{2j_1j}}{\mu_1^2\mu_2^2}\right|\right]=O(r_n).
\end{align}

Suppose $k=2$ and $j_1=1$. In this case, $k_1, j\geq3$. Then
\begin{align*}
T_{1j2}T_{21k_1}= A_{1j}A_{j2}A_{21}A_{1k_1}- A_{1j}A_{j2}A_{21}A_{2k_1}-A_{12}A_{1j}A_{1k_1}+A_{1j}A_{12}A_{2k_1}
\end{align*}
If $j_1\neq k$, then
\begin{align*}
\mathbb{E}[|T_{1j2}T_{21k_1}|]=O(r_n^3).
\end{align*}
If $j_1=k$, then
\begin{align*}
\mathbb{E}[|T_{1j2}T_{21k_1}|]=O(r_n^2).
\end{align*}
Therefore, we have
\begin{align}\label{lemmapepeq23}
\mathbb{E}\left[\left|\frac{(n+1)^2\mathbb{E}\left[A_{12}\mid X_1\right]\mathbb{E}\left[A_{12}\mid X_2\right]\sum_{\substack{j, k_1\geq 3}}T_{1j2}T_{21k_1}}{\mu_1^2\mu_2^2}\right|\right]=O\left(r_n+\frac{1}{n}\right).
\end{align}

Suppose $k=2$ and $k_1=1$. In this case,  $j_1, j\geq3$. Then
\begin{align*}
T_{1j2}T_{2j_11}= A_{1j}A_{j2}A_{2j_1}A_{j_11}- A_{1j}A_{j2}A_{21}A_{2j_1}-A_{12}A_{1j}A_{2j_1}A_{j_11}+A_{1j}A_{12}A_{2j_1}.
\end{align*}
If $j_1\neq j$, then
\begin{align*}
\mathbb{E}[|T_{1j2}T_{2j_11}|]=O(r_n^3).
\end{align*}
If $j_1=j$, then
\begin{align*}
\mathbb{E}[|T_{1j2}T_{2j_11}|]=O(r_n^2).
\end{align*}
Therefore, we have
\begin{align}\label{lemmapepeq24}
\mathbb{E}\left[\left|\frac{(n+1)^2\mathbb{E}\left[A_{12}\mid X_1\right]\mathbb{E}\left[A_{12}\mid X_2\right]\sum_{\substack{j, j_1\geq 3}}T_{1j2}T_{2j_11}}{\mu_1^2\mu_2^2}\right|\right]=O\left(r_n+\frac{1}{n}\right).
\end{align}

By (\ref{lemmapepeq15})-(\ref{lemmapepeq24}), we have
\begin{align}\label{00lemmapepeq14}
\mathbb{E}\left[\frac{(n+1)^2\mathbb{E}\left[A_{12}\mid X_1\right]\mathbb{E}\left[A_{12}\mid X_2\right]\sum_{\substack{j\ne k\geq3, j_1\ne k_1\geq 3\\ j=2, or\ k=2}}T_{1jk}T_{2j_1k_1}}{\mu_1^2\mu_2^2}\right]  
=O\left(r_n+\frac{1}{n}+nr_n^4\right).
\end{align}

Combining (\ref{lemmapepeq14}) and (\ref{00lemmapepeq14}) yields (\ref{franeq3}).  Equations (\ref{franeq1}) and (\ref{franeq2}) follow easily from the proof of (\ref{franeq3}).

\qed

\subsection{Proof of Lemma \ref{lemqq} }
Note that
\begin{align}\label{aftereq1}
\mathbb{E}\left[\frac{Q_1Q_2}{\mu_1^2\mu_2^2}\right]=\sum_{\substack{j\ne k \ne 1\\ j_1\ne k_1\neq 2}}\mathbb{E}\left[\frac{S_{1jk}S_{2j_1k_1}}{\mu_1^2\mu_2^2}\right].
\end{align} 
    
If $j\ne k \ne j_1\ne k_1\geq3$, then  $S_{1jk}$ and $S_{2j_1k_1}$ are conditionally independent given $X_1,X_2$. Hence, we have
\begin{align*}
\mathbb{E}\left[\frac{S_{1jk}S_{2j_1k_1}}{\mu_1^2\mu_2^2}\right]
&=\mathbb{E}\left[\frac{\mathbb{E}\left[S_{1jk}S_{2j_1k_1}|X_1,X_2\right]}{\mu_1^2\mu_2^2}\right]\\
&=\mathbb{E}\left[\frac{\mathbb{E}\left[S_{1jk}|X_1\right]\mathbb{E}\left[S_{2j_1k_1}|X_2\right]}{\mu_1^2\mu_2^2}\right]\\
&=\mathbb{E}\left[\frac{\mathbb{E}\left[S_{1jk}|X_1\right]}{\mu_1^2}\right]\mathbb{E}\left[\frac{\mathbb{E}\left[S_{2j_1k_1}|X_2\right]}{\mu_2^2}\right]\\
&=\frac{1}{(n-1)^4}\left(\mathbb{E}\left[\frac{\mathbb{E}\left[\big(A_{12}A_{13}
    -A_{12}A_{13}A_{23}\big)|X_1\right]}{\big(\mathbb{E}[A_{12}|X_1]\big)^2}\right]\right)^2.
\end{align*}
It then follows that
\begin{align}\label{ffeqqeq7}
\sum_{\substack{j\ne k \ne j_1\ne k_1\geq3}}\mathbb{E}\left[\frac{S_{1jk}S_{2j_1k_1}}{\mu_1^2\mu_2^2}\right]=\left(\mathbb{E}\left[\frac{\mathbb{E}\left[\big(A_{12}A_{13}
    -A_{12}A_{13}A_{23}\big)|X_1\right]}{\big(\mathbb{E}[A_{12}|X_1]\big)^2}\right]\right)^2+O\left(\frac{1}{n}\right).
\end{align}

Suppose $\{j, k\}=\{j_1, k_1\}$. In this case, $j,k,j_1,k_1\geq3$. There are at most $n^2$ such index pairs. Note that $0\leq S_{1jk}=A_{1j}A_{1k}(1-A_{jk})\leq 2A_{1j}A_{1k}$ and $|S_{2j_1k_1}|=|A_{2j_1}A_{2k_1}(1-A_{j_1k_1})|\leq 2$. Then $|S_{1jk}S_{2j_1k_1}|\leq 4 A_{1j}A_{1k}$. In this case, 
\begin{align}\label{eqqeq1}
\mathbb{E}\left[|S_{1jk}S_{2j_1k_1}|\right]\leq 4 \mathbb{E}\left[A_{1j}A_{1k}\right]=4 \mathbb{E}\left[\mathbb{E}\left[A_{1j}A_{1k}|X_1\right]\right]=O(r_n^2).
\end{align}
Then
\begin{align}\label{eqqeq2}
\mathbb{E}\left[\left|\frac{\sum_{\substack{j\ne k\geq 3, j_1\ne k_1\geq3\\
\{j, k\}=\{j_1, k_1\}}}S_{1jk}S_{2j_1k_1}}{\mu_1^2\mu_2^2}\right|\right]=O\left(\frac{n^2r_n^2}{n^4r_n^4}\right) =O\left(\frac{1}{n^2r_n^2}\right) .
\end{align}

Suppose $|\{j, k\}\cap\{j_1, k_1\}|=1$ and $j,k,j_1,k_1\geq3$. There are at most $n^3$ such index tuples $\{j,k,j_1,k_1\}$. In this case, either $j_1\not\in\{j, k\}$ or $k_1\not\in\{j, k\}$. Without loss of generality, let $j_1\not\in\{j, k\}$. Then $|S_{1jk}S_{2j_1k_1}|\leq 4 A_{1j}A_{1k}A_{2j_1}$. In this case,  
\[\mathbb{E}\left[|S_{1jk}S_{2j_1k_1}|\right]\leq 4 \mathbb{E}\left[A_{1j}A_{1k}A_{2j_1}\right]=4 \mathbb{E}\left[\mathbb{E}\left[A_{1j}A_{1k}A_{2j_1}|X_1\right]\right]=O(r_n^3).\]
Then
\begin{align}\label{ffeqqeq3}
\mathbb{E}\left[\left|\frac{\sum_{\substack{j\ne k\geq 3, j_1\ne k_1\geq3\\
|\{j, k\}\cap\{j_1, k_1\}|=1}}S_{1jk}S_{2j_1k_1}}{\mu_1^2\mu_2^2}\right|\right]=O\left(\frac{n^3r_n^3}{n^4r_n^4}\right) =O\left(\frac{1}{nr_n}\right).
\end{align}

Suppose $j=2$ and $j_1=1$. Then $k\geq3$ and $k_1\geq3$. If $k\neq k_1$, there are at most $n^2$ such index tuples $\{j,k,j_1,k_1\}$. In this case,  we have
\begin{equation}\label{aftereqq1}
\mathbb{E}\left[|S_{12k}S_{21k_1}|\right]\leq 4 \mathbb{E}\left[A_{2k}A_{1k_1}\right]=O(r_n^2).
\end{equation}
If $k=k_1$, there are at most $n$ such index tuples $\{j,k,j_1,k_1\}$. In this case,  we have
\begin{equation}\label{aftereqq2}
\mathbb{E}\left[|S_{12k}S_{21k_1}|\right]\leq 4 \mathbb{E}\left[A_{2k}\right]=O(r_n).
\end{equation}
Similarly, we can get (\ref{aftereqq1}) and (\ref{aftereqq2}) for the case $\{j,k\}\cap\{2\}\neq\emptyset$ and $\{j_1,k_1\}\cap\{1\}\neq\emptyset$. Therefore, we have
\begin{align}\label{ffeqqeq4}
\mathbb{E}\left[\left|\frac{\sum_{\substack{j\ne k\neq 1, j_1\ne k_1\neq2\\
\{j,k\}\cap\{2\}\neq\emptyset,\{j_1,k_1\}\cap\{1\}\neq\emptyset}}S_{1jk}S_{2j_1k_1}}{\mu_1^2\mu_2^2}\right|\right]=O\left(\frac{n^2r_n^2+nr_n}{n^4r_n^4}\right)  
=O\left(\frac{1}{n^2r_n^2}\right).
\end{align}

Suppose $j=2$. If $k\neq j_1\neq k_1\geq3$, there are at most $n^3$ such index tuples $\{j,k,j_1,k_1\}$. Then 
\[\mathbb{E}\left[|S_{1jk}S_{2j_1k_1}|\right]\leq 4 \mathbb{E}\left[A_{1k}A_{2j_1}A_{2k_1}\right]=4 \mathbb{E}\left[\mathbb{E}\left[A_{1k}A_{2j_1}A_{2k_1}|X_1,X_2\right]\right]=O(r_n^3),\]
which implies that
\begin{align}  \label{ffeqqeq5}
\mathbb{E}\left[\left|\frac{\sum_{\substack{k\neq j_1\ne k_1\neq2}}S_{12k}S_{2j_1k_1}}{\mu_1^2\mu_2^2}\right|\right]=O\left(\frac{n^3r_n^3}{n^4r_n^4}\right)=O\left(\frac{1}{nr_n}\right).
\end{align}
If $k\in\{j_1,k_1\}$, then 
\[\mathbb{E}\left[|S_{1jk}S_{2j_1k_1}|\right]\leq 4 \mathbb{E}\left[A_{2j_1}A_{2k_1}\right]=4 \mathbb{E}\left[\mathbb{E}\left[A_{2j_1}A_{2k_1}|X_1,X_2\right]\right]=O(r_n^2).\]
It follows that
\begin{align}\label{ffeqqeq6}
\mathbb{E}\left[\left|\frac{\sum_{\substack{j_1\ne k_1\neq2\\ k\in\{j_1,k_1\}}}S_{12k}S_{2j_1k_1}}{\mu_1^2\mu_2^2}\right|\right]&=O\left(\frac{n^2r_n^2}{n^3r_n^3}\right) 
&=O\left(\frac{1}{n^2r_n^2}\right).
\end{align}

The case \(k=2\) is analogous to the case \(j=2\).  In summary, combining (\ref{ffeqqeq7}) through (\ref{ffeqqeq6}) establishes Lemma \ref{lemqq}.

\qed

\subsection{Proof of Lemma \ref{lempq} }

Note that
\begin{align}\label{214eq0}
\frac{(n+1)\mathbb{E}\left[A_{12}\mid X_1\right]}{\mu_1^2\mu_2^2}P_1Q_2=\frac{(n+1)\mathbb{E}\left[A_{12}\mid X_1\right]}{\mu_1^2\mu_2^2}\sum_{\substack{j\ne k \ne 1\\ j_1\ne k_1\neq 2}}T_{1jk}S_{2j_1k_1}.
\end{align}

Suppose $j\neq k\neq j_1\neq k_1\geq3$. Since \(T_{1jk}\) and \(S_{2j_{1}k_{1}}\) are conditionally independent given \(X_{1},X_{2}\), the conditional expectation factorizes as
\begin{align*}  
\mathbb{E}[T_{1jk}S_{2j_1k_1}|X_1,X_2]&=\mathbb{E}[T_{1jk}|X_1]\mathbb{E}[S_{2j_1k_1}|X_2]\\
&=\mathbb{E}[(A_{13}A_{34}-A_{13}A_{14})|X_1]\mathbb{E}[(A_{25}A_{26}-A_{25}A_{26}A_{56})|X_2].
\end{align*}
This implies that
\begin{align}  \nonumber
&\sum_{j\neq k\neq j_1\neq k_1\geq3}\mathbb{E}\left[\frac{(n+1)\mathbb{E}\left[A_{12}\mid X_1\right]\mathbb{E}[T_{1jk}S_{2j_1k_1}|X_1,X_2]}{\mu_1^2\mu_2^2}\right]\\ \nonumber
&=\frac{(n^5+O(n^4))}{(n-1)^4}\mathbb{E}\left[\frac{\mathbb{E}[(A_{13}A_{34}-A_{13}A_{14})|X_1]\mathbb{E}[(A_{25}A_{26}-A_{25}A_{26}A_{56})|X_2]}{\mathbb{E}\left[A_{12}\mid X_1\right](\mathbb{E}\left[A_{12}\mid X_2\right])^2}\right]\\ \label{214eq1}
&=n\mathbb{E}\left[\frac{\mathbb{E}[(A_{13}A_{34}-A_{13}A_{14})|X_1]}{\mathbb{E}\left[A_{12}\mid X_1\right]}\right]\mathbb{E}\left[\frac{\mathbb{E}[(A_{25}A_{26}-A_{25}A_{26}A_{56})|X_2]}{(\mathbb{E}\left[A_{12}\mid X_2\right])^2}\right]+O\left(r_n^3\right).
\end{align}

Suppose $j=j_1\neq k\neq k_1\geq3$. Then
\begin{align*}  
&\mathbb{E}[|T_{1jk}S_{2jk_1}| ]\\
&\leq\mathbb{E}[ A_{1j}A_{jk}A_{2j}A_{2k_1}+ A_{1j}A_{jk}A_{2j}A_{2k_1}A_{jk_1}+A_{1j}A_{1k}A_{2j}A_{2k_1}+A_{1j}A_{1k}A_{2j}A_{2k_1}A_{jk_1}] \\
&=O(r_n^4).
\end{align*}
This implies that
\begin{align}  \label{214eq2}
\sum_{j\neq k\neq k_1\geq3}\mathbb{E}\left[\frac{(n+1)\mathbb{E}\left[A_{12}\mid X_1\right]\mathbb{E}[T_{1jk}S_{2jk_1}|X_1,X_2]}{\mu_1^2\mu_2^2}\right]=O\left(\frac{n^4r_n^5}{n^4r_n^4}\right)=O(r_n).
\end{align}

Suppose $j=k_1\neq k\neq j_1\geq3$. Then
\begin{align*}  
&\mathbb{E}[|T_{1jk}S_{2j_1j}| ]\\
&\leq\mathbb{E}[ A_{1j}A_{jk}A_{2j_1}A_{2j}+ A_{1j}A_{jk}A_{2j_1}A_{2j}A_{j_1j}+A_{1j}A_{1k}A_{2j_1}A_{2j}+A_{1j}A_{1k}A_{2j_1}A_{2j}A_{j_1j}] \\
&=O(r_n^4).
\end{align*}
The sum over these indices in (\ref{214eq0}) is bounded by $O\left(r_n\right)$.

Suppose $k=j_1\neq j\neq k_1\geq3$. Then
\begin{align*}  
&\mathbb{E}[|T_{1jk}S_{2kk_1}| ]\\
&\leq\mathbb{E}[ A_{1j}A_{jk}A_{2k}A_{2k_1}+ A_{1j}A_{jk}A_{2k}A_{2k_1}A_{kk_1}+A_{1j}A_{1k}A_{2k}A_{2k_1}+A_{1j}A_{1k}A_{2k}A_{2k_1}A_{kk_1}] \\
&=O(r_n^4).
\end{align*}
The sum over these indices in (\ref{214eq0}) is bounded by $O\left(r_n\right)$.

Suppose $k=k_1\neq j\neq j_1\geq3$. Then
\begin{align*}  
&\mathbb{E}[|T_{1jk}S_{2j_1k}| ]\\
&\leq\mathbb{E}[ A_{1j}A_{jk}A_{2j_1}A_{2k}+ A_{1j}A_{jk}A_{2j_1}A_{2k}A_{j_1k}+A_{1j}A_{1k}A_{2j_1}A_{2k}+A_{1j}A_{1k}A_{2j_1}A_{2k}A_{j_1k}] \\
&=O(r_n^4).
\end{align*}
The sum over these indices in (\ref{214eq0}) is bounded by $O\left(r_n\right)$.

Suppose $\{j,k\}=\{j_1,k_1\}$ and $j,k,j_1,k_1\geq3$. Then
\begin{align*}  
\mathbb{E}[|T_{1jk}S_{2jk}| ]\leq2\mathbb{E}[ A_{1j}A_{jk}+ A_{1j}A_{1k}]=O(r_n^2).
\end{align*}
The sum over these indices in (\ref{214eq0}) is bounded by $O\left(\frac{1}{nr_n}\right)$.

Therefore,  we have
\begin{align}  \nonumber
&\mathbb{E}\left[\frac{(n+1)\mathbb{E}\left[A_{12}\mid X_1\right]\sum_{j\neq k\geq3, j_1\neq k_1\geq3}T_{1jk}S_{2j_1k_1}}{\mu_1^2\mu_2^2}\right]\\ \label{214eq3}
&=n\mathbb{E}\left[\frac{\mathbb{E}[(A_{13}A_{34}-A_{13}A_{14})|X_1]}{\mathbb{E}\left[A_{12}\mid X_1\right]}\right]\mathbb{E}\left[\frac{\mathbb{E}[(A_{25}A_{26}-A_{25}A_{26}A_{56})|X_2]}{(\mathbb{E}\left[A_{12}\mid X_2\right])^2}\right]+O\left(r_n+\frac{1}{nr_n}\right).
\end{align}

Suppose $j=2$, $k\neq j_1\neq k_1\geq3$. Then
\begin{align*}  
&\mathbb{E}[|T_{12k}S_{2j_1k_1}| ]\\
&\leq\mathbb{E}[ A_{12}A_{2k}A_{2j_1}A_{2k_1}+ A_{12}A_{2k}A_{2j_1}A_{2k_1}A_{j_1k_1}+A_{12}A_{1k}A_{2j_1}A_{2k_1}+A_{12}A_{1k}A_{2j_1}A_{2k_1}A_{j_1k_1}] \\
&=O(r_n^4).
\end{align*}
The sum over these indices in (\ref{214eq0}) is bounded by $O\left(r_n\right)$.

Suppose $j=2$, $k=j_1\neq k_1\geq3$ or $k=k_1\neq j_1\geq3$. Then
\begin{align*}  
\mathbb{E}[|T_{12k}S_{2j_1k_1}| ]\leq2\mathbb{E}[ A_{12}A_{2k}+ A_{12}A_{1k}]=O(r_n^2).
\end{align*}
The sum over these indices in (\ref{214eq0}) is bounded by $O\left(\frac{1}{nr_n}\right)$.

Suppose $j=2$ and $j_1=1$. If $k\neq k_1$, then
\begin{align*}  
\mathbb{E}[|T_{12k}S_{2j_1k_1}| ]\leq2\mathbb{E}[ A_{12}A_{2k}+ A_{12}A_{1k}]=O(r_n^2).
\end{align*}
If $k=k_1$, then
\begin{align*}  
\mathbb{E}[|T_{12k}S_{21k}| ]\leq2\mathbb{E}[ A_{2k}+A_{1k}]=O(r_n).
\end{align*}
The sum over these indices in (\ref{214eq0}) is bounded by $O\left(\frac{1}{nr_n}\right)$. The same bound applies to the case $j=2$ and $k_1=1$.

Suppose $k=2$, $j\neq j_1\neq k_1\geq3$. Then
\begin{align*}  
&\mathbb{E}[|T_{1j2}S_{2j_1k_1}| ]\\
&\leq\mathbb{E}[ A_{1j}A_{j2}A_{2j_1}A_{2k_1}+ A_{1j}A_{j2}A_{2j_1}A_{2k_1}A_{j_1k_1}+A_{1j}A_{12}A_{2j_1}A_{2k_1}+A_{1j}A_{j2}A_{2j_1}A_{2k_1}A_{j_1k_1}] \\
&=O(r_n^4).
\end{align*}
The sum over these indices in (\ref{214eq0}) is bounded by $O\left(r_n\right)$.

Suppose $k=2$, $j=j_1\neq k_1\geq3$ or $j=k_1\neq j_1\geq3$. Then
\begin{align*}  
\mathbb{E}[|T_{1j2}S_{2j_1k_1}| ]\leq2\mathbb{E}[ A_{1j}A_{j2}+ A_{1j}A_{12}]=O(r_n^2).
\end{align*}
The sum over these indices in (\ref{214eq0}) is bounded by $O\left(\frac{1}{nr_n}\right)$.

Suppose $k=2$ and $j_1=1$. If $j\neq k_1$, then
\begin{align*}  
\mathbb{E}[|T_{1j2}S_{2j_1k_1}| ]\leq2\mathbb{E}[ A_{1j}A_{j2}+ A_{1j}A_{12}]=O(r_n^2).
\end{align*}
If $j=k_1$, then
\begin{align*}  
\mathbb{E}[|T_{1j2}S_{21j}| ]\leq2\mathbb{E}[ A_{1j}]=O(r_n).
\end{align*}
The sum over these indices in (\ref{214eq0}) is bounded by $O\left(\frac{1}{nr_n}\right)$. The same bound applies to the case $k=2$ and $k_1=1$.

In summary, we obtain
\begin{align}  \label{214eq4}
\mathbb{E}\left[\frac{(n+1)\mathbb{E}\left[A_{12}\mid X_1\right]\sum_{\substack{j\neq k\neq1, j_1\neq k_1\neq 2\\ j=2\ or\ k=2}}T_{1jk}S_{2j_1k_1}}{\mu_1^2\mu_2^2}\right]=O\left(r_n+\frac{1}{nr_n}\right).
\end{align}

Combining (\ref{214eq3}) and (\ref{214eq4}) yields the result of Lemma \ref{lempq}.

\qed

\end{document}